\documentclass[a4paper,10pt]{article}
\usepackage{a4}
\usepackage[ansinew]{inputenc}
\usepackage{epsfig}
\usepackage{amssymb}
\usepackage{bbm}
\usepackage{amsmath}
\usepackage{amsthm} 
\usepackage{amscd}
\usepackage{enumerate}
\usepackage{float}
\usepackage{xpatch}
\usepackage[toc,page]{appendix}
\usepackage{mathtools}

\usepackage{multirow}
\usepackage{colortbl}
\usepackage{tabularx}
\newcolumntype{L}[1]{>{\raggedright\arraybackslash}p{#1}} 
\newcolumntype{C}[1]{>{\centering\arraybackslash}p{#1}} 
\newcolumntype{R}[1]{>{\raggedleft\arraybackslash}p{#1}} 

  \newcommand{\E}{\mathcal{E}}
  \newcommand{\F}{\mathcal{F}}
  \newcommand{\R}{\mathcal{R}}  
  \newcommand{\A}{\mathcal{A}}
  \newcommand{\D}{\mathcal{D}}  
  \newcommand{\C}{\mathcal{C}}  
  \renewcommand{\P}{\mathcal{P}} 

\makeatletter
\xpatchcmd{\proof}{\@addpunct{.}}{\normalfont\,\@addpunct{:}}{}{} 
\makeatother

  \newtheoremstyle{dotless}{}{}{\itshape}{}{\bfseries}{:}{ }{}

    \newtheoremstyle{dotlessrem}{}{}{}{}{\bfseries}{:}{ }{}

  \theoremstyle{dotless}

\newtheorem{lemma}{Lemma}[section]
\newtheorem{proposition}[lemma]{Proposition}
\newtheorem{definition}{Definition}[section]
\newtheorem{theorem}[lemma]{Theorem}
\newtheorem{corollary}[lemma]{Corollary}
 \theoremstyle{dotlessrem}
\newtheorem*{remark}{Remark}

\renewenvironment{abstract}
 {\small
  \begin{center}
  \bfseries \abstractname\vspace{-.5em}\vspace{0pt}
  \end{center}
  \list{}{
    \setlength{\leftmargin}{1.5cm}%
    \setlength{\rightmargin}{\leftmargin}%
  }%
  \item\relax}
 {\endlist}

\begin{document}
\begin{center}
\Large\textbf{Spectral asymptotics for Stretched Fractals}\\
\large Elias Hauser\footnote{Institute of Stochastics and Applications, University of Stuttgart, Pfaffenwaldring 57, 70569 Stuttgart, Germany, E-mail: elias.hauser@mathematik.uni-stuttgart.de}\end{center}
\tableofcontents \newpage
\begin{abstract}
The Stretched Sierpinski Gasket (or Hanoi attractor) was subject of several prior works. In this work we use this idea of \textit{stretching} self-similar sets to obtain non-self-similar ones. We are able to do this for a subset of the connected p.c.f. self-similar sets that fulfill a certain connectivity condition. We construct Dirichlet forms and study the associated self-adjoint operators by calculating the Hausdorff dimension w.r.t. the resistance metric as well as the leading term of the eigenvalue counting function.
\end{abstract}
\section{Introduction}
In this work we introduce the so called \textit{stretched fractals} which originate by altering the construction of p.c.f. self-similar sets. We construct Dirichlet forms on these non-self-similar sets and conduct spectral asymptotics on the associated self-adjoint operators.\\

Studying asymptotic behavior of the spectrum of laplacians is an important tool in physics, for example, to understand the behavior of heat and waves in the underlying media. We are in particular interested in the asymptotic growing of the eigenvalues. For the classical laplacian on bounded domains $\Omega\subset \mathbb{R}^d$ we know that the Dirichlet eigenvalue counting function $N_D^\Omega$ grows asymptotically like
\begin{align*}
N_D^\Omega(x)\sim x^\frac{d}{2}
\end{align*} 
This result is originally due to Weyl \cite{we11}. Nature, however, is not build up of smooth structures, but rather porous, disordered and finely structured material. In the 70s, therefore, the interest grew in studying fractals, which are much better suited to describe natural structures and phenomena. To understand the physical behavior on such objects we need a laplacian. Kigami constructed such an operator first on the Sierpinski Gasket \cite{kig89} and later on the so called p.c.f. self-similar sets \cite{kig93} by a sequence of operators on the approximating graphs. This is called the analytical approach. Shima \cite{shim} and Fukushima-Shima \cite{fushim} calculated the leading term of the eigenvalue counting function of the laplacian on the Sierpinski Gasket. Later Kigami and Lapidus calculated the asymptotic growing for p.c.f. self-similar sets in \cite{kig93}. Contrary to a conjecture by Berry in \cite{ber1,ber2}, the leading term does not coincide with $\frac{d_H}2$, where $d_H$ denotes the Hausdorff dimension. \\

In this work we want to examine some non-self-similar sets and calculate the leading term for operators in such a case. One example is the Stretched Sierpinski Gasket which was analyzed geometrically in \cite{af12}.
\begin{figure}[ht]
\centering
\includegraphics[width=\textwidth]{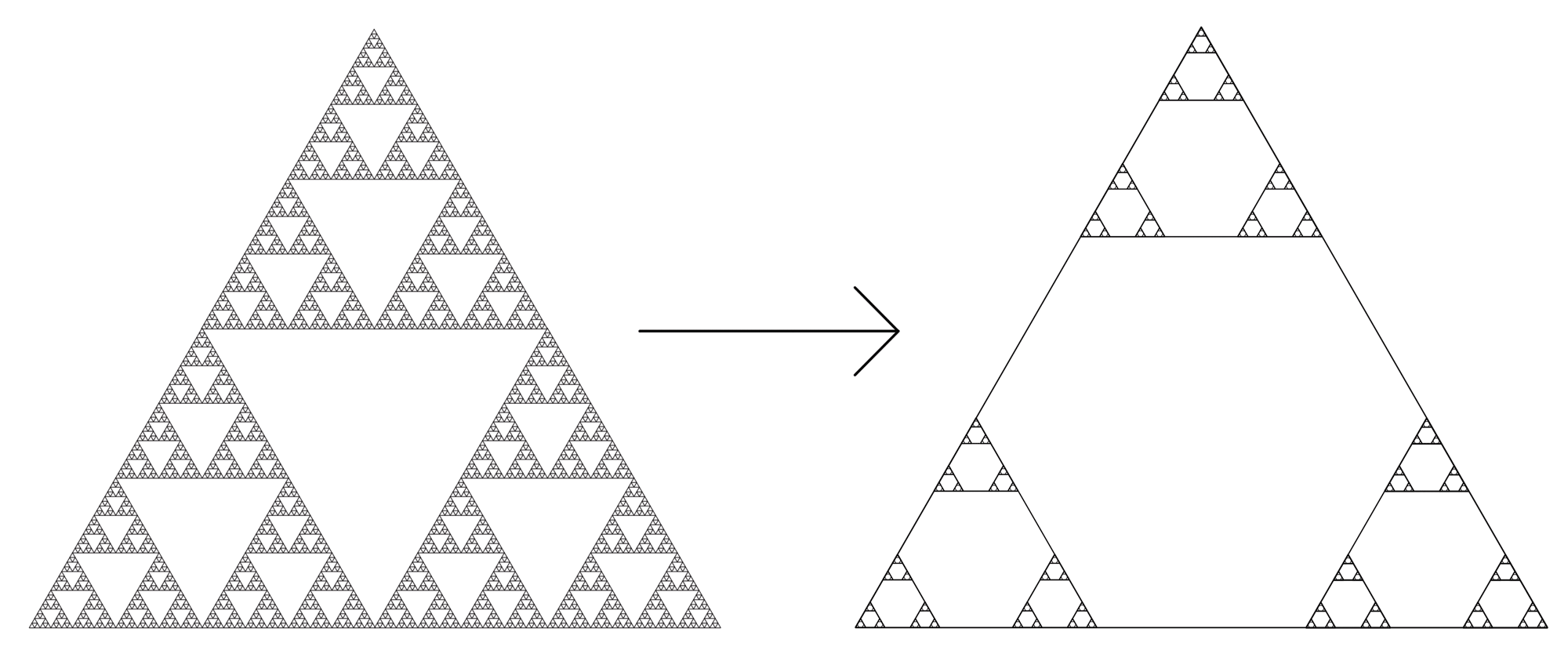}
\caption{Stretching the Sierpinski Gasket}\label{ssgfigure}
\end{figure}
 Let $p_1,p_2,p_3$ be the vertex points of an equilateral triangle with side length 1 and $\alpha \in(0,1)$
\begin{align*}
G_i(x)&:=\frac{1-\alpha}2(x-p_i) +p_i, \ i\in \{1,2,3\}	\\[0.1cm]
e_1&:=\{\lambda G_2(p_3)+(1-\lambda)G_2(p_3) \ : \lambda \in  (0,1)\}\quad e_2,e_3 \ \text{analog}
\end{align*} 
The unique compact set $K$ which fulfills 
\begin{align*}
K=G_1(K)\cup G_2(K) \cup G_3(K) \cup e_1\cup e_2\cup e_3 
\end{align*}
is called the Stretched Sierpinski Gasket. In \cite{afk17} the authors constructed resistance forms on this set and in \cite{haus17} the author calculated the leading term of the associated operators after introducing measures. These results were refined in \cite{haus18} where it was shown that there are oscillations in the leading term of the eigenvalue counting function. These oscillations are typical for such highly symmetrical sets. \\

In this work we want to generalize the idea of \textit{stretching} to more p.c.f. self-similar sets which we will call \textit{stretched fractals}. This work is structured as follows. In chapter~\ref{chapter2} we construct stretched fractals and include some examples. In chapter~\ref{chapter3} we build a sequence of approximating graphs and introduce the notion of \textit{regular sequences of harmonic structures} which is a generalization of the term \textit{harmonic structure} in the self-similar case. Afterwards we construct resistance forms on stretched fractals in chapter~\ref{chapter4}, presumed we have a harmonic structure. In chapter~\ref{chap_meas_oper} we describe measures on stretched fractals which allows us to get Dirichlet forms from the resistance forms and thus the self-adjoint operators that we want to study. In chapter~\ref{chap_cond_haus} we introduce some conditions that are necessary to calculate both Hausdorff dimension in resistance metric as well as the leading term of the eigenvalue counting function which is done in chapter~\ref{chap_cond_haus} resp. chapter~\ref{chapter7}. Lastly we list some open problems and further ideas in chapter~\ref{chapter8}.

\section{Stretched fractals}\label{chapter2}
In \cite{haus17,haus18} we analyzed the Stretched Sierpinski Gasket analytically. It is constructed by lowering the contraction ratios of the similitudes of the self-similar Sierpinski Gasket and filling the arising holes with one-dimensional lines (Figure~\ref{ssgfigure}). 

We want to generalize this construction of \textit{stretching} to more self-similar fractals. For the Sierpinski Gasket $S$ it was essential that two copies $F_i(S)$ and $F_j(S)$ only intersect at a single point. Therefore, it is clear how we have to connect these copies if we stretch them apart. In general we need the fractal that we want to stretch to be finitely ramified. In this case we can connect the copies that get stretched away from each other by one-dimensional lines. These are the so called \textit{p.c.f. self-similar fractals} introduced by Kigami in \cite{kig93}. The notion of p.c.f. self-similar sets is well known thus we only want to recall the most important properties and also alter the notion slightly. 

\subsection{Definition of stretched fractals}
Let $(F_1,\ldots,F_N)$ be the IFS of a connected p.c.f. self-similar fractal $F\subset \mathbb{R}^d$. That means
\begin{align*}
F=\bigcup_{i=1}^N F_i(F)
\end{align*}
where $F_i$ are contracting similitudes with distinct unique fixed points $q_i$.\\

We will introduce some notation that is commonly used. We denote the alphabet by $\A:=\{1,\ldots,N\}$ and all words of finite length $\A^\ast:=\bigcup_{n\geq 1}\A^n$ and $\A^\ast_0:=\bigcup_{n\geq 0}\A^n$ if we also want to include the empty word. For $w=(w_1,\ldots,w_n)$ we denote by $F_w:=F_{w_1}\circ\ldots\circ F_{w_n}$ the composition of the similitudes and by $F_w:=\operatorname{id}$ the identity if $w=\emptyset$ is the empty word.\\

We can now define the critical set $\C$. This set plays an important role in the construction of \textit{stretched fractals}:
\begin{align*}
\C:=\bigcup_{\substack{i,j\in \A \\ i\neq j}} F_i(F)\cap F_j(F)
\end{align*}
That means $\C$ are the points where $1$-cells meet. With this we define the so called post critical set $\P$.
\begin{align*}
\P:=\{x\in F \ | \ \exists w\in \A^\ast: F_w(x)\in \C  \}
\end{align*} 
The post critical set $\P$ consists of all points that get mapped to the critical set $\C$ by finite compositions of the similitudes $F_1,\ldots,F_N$. For p.c.f. self-similar sets we know that $\#\P<\infty$. For nested fractals $\P$ is made up of the essential fixed points (see \cite[Example 8.5]{kig93}). In general $\P$ can have elements that are no fixed points (see Hata's tree in chapter~\ref{hataex}). To still be able to stretch these fractals we need to make an assumption for $\P$. We only consider connected p.c.f. self-similar sets, such that
\begin{align}
\forall p \in \P \ \exists w\in \A^\ast_0 \text{ and a fixed point } q_i\in\P \text{ such that } p=F_w(q_i) \tag{C1}\label{pcfcond}\\
q_i\notin\C \ \forall i\in\A \tag{C2} \label{pcfcond2}
\end{align}
That means, each post critical point is the image of a fixed point under finite composition of the similitudes or one itself. This is obviously true for nested fractals and it is also true for Hata's tree which is not a nested fractal. The fixed points are not allowed to be critical points themselves.\\

We want to introduce a quantity that describes the \textit{level of connectedness} at $\C$. This value is called the \textit{multiplicity} of a point $c\in \C$ and it counts how many $1$-cells meet at $c$. We can define this value for all $x\in F$:
\begin{align*}
\rho(x):=\#\{i\in\A \ | \ x\in F_i(F)\}
\end{align*}
We can also count the $n$-cells that meet at $x$:
\begin{align*}
\rho_n(x):=\#\{w\in\A^n\ |\ x\in F_w(F)\}
\end{align*}
As it turns out this gives us the same value $\rho(c)=\rho_n(c)$ for $c\in \C$. This fact was proved by Lindstr\o m in \cite[Prop. IV.16]{lin90} for nested fractals but it only used the nesting property which is also true for all p.c.f. self-similar sets \cite{kig93}.

Now the critical set $\C$ with $\rho(c)$ for all $c\in \C$ describes how the fractal is connected. In particular we have $\rho(c)\geq 2$ for all $c\in\C$. Next we want to be able to say which post critical points get mapped to $c\in \C$.
\begin{align*}
\forall c\in\C \ : \ &\exists w^{c,1},\ldots, w^{c,\rho(c)}\in \A^\ast \text{ and fixed points } q^c_1,\ldots,q^c_{\rho(c)}\in\P \text{ such that}\\
&F_{w^{c,l}}(q^c_l)=c, \ \forall l\in\{1,\ldots,\rho(c)\} \\
&\text{where } w^{c,l}_1 \text{ are pairwise distinct.}
\end{align*}
We can do this since we consider p.c.f. self-similar sets that fulfill (\ref{pcfcond}). The first letters $w^{c,l}_1$ are different which indicates that $c$ belongs to $\rho(c)$ many different $1$-cells.\\

We can now define a new IFS $(G_1,\ldots,G_N)$ in the following way with $0<\alpha<1$:
\begin{align*}
G_i:=\alpha(F_i-q_i)+q_i
\end{align*}
This procedure lowers the contraction ratio of $F_i$ by multiplying it with $\alpha$ and it does this in a way that preserves the fixed point. It sort of compresses the image of $F_i$ linearly into its fixed point. The attractor of $(G_1,\ldots,G_N)$ will be denoted by $\Sigma_\alpha$.\\

By lowering the contraction ratios the copies get disconnected. I.e. the attractor of the new IFS becomes totally disconnected. The copies were connected at the critical set $\mathcal{C}$. We want to save the degree of connectedness by introducing one-dimensional lines connecting the copies.\\[.2cm]
$\forall c\in\mathcal{C}$ define
\begin{align*}
e_{c,l}:=\{\lambda G_{w^{c,l}}(q^c_l)+(1-\lambda)c, \ \lambda\in[0,1]\}, \ \forall l\in\{1,\ldots,\rho(c)\}
\end{align*}
$G_{w^{c,l}}(q^c_l)$ is the point that got stretched away from $c$. Due to (\ref{pcfcond2}) $e_{c,l}$ is a one-dimensional object for all $c,l$. For $w\in\A^\ast$ we denote $e^w_{c,l}:=G_w(e_{c,l})$. Now we can define the stretched fractal associated to the p.c.f. self-similar set $F$:
\begin{definition} The unique compact set $K_\alpha$ that fulfills the equation
\begin{align*}
K_\alpha =\bigcup_{i=1}^N G_i(K_\alpha) \cup \bigcup_{c\in \mathcal{C}}\bigcup_{l=1}^{\rho(c)} e_{c,l}
\end{align*}
is called the \textit{stretched fractal} associated to $F$.\\

\noindent The unique compact set $\Sigma_\alpha$ that fulfills
\begin{align*}
\Sigma_\alpha=\bigcup_{i=1}^N G_i(\Sigma_\alpha)
\end{align*}
is called the fractal part of $K_\alpha$.
\end{definition}
We can imagine the construction by fixing the points of $\mathcal{C}$ and stretching the copies away from each $c\in\mathcal{C}$ and then adding lines connecting the copies with $c$ like a spider's web. Since the fixed points of $G_i$ and $F_i$ are the same we ensure that $q^c_l$ and thus $G_{w^{c,l}}(q^c_l)$ are elements of $\Sigma_\alpha$. By this $K_\alpha$ is a connected set. Therefore, (C1) ensures connectedness. We will include a few examples of stretched fractals at the end of this chapter.\\

Solutions of equations like the one in this definition are already known. Barnsley denoted such a setting in \cite[Chapter 3.4]{bar06} by \textit{IFS with condensation} where $\bigcup_{c\in \mathcal{C}}\bigcup_{l=1}^{\rho(c)} e_{c,l}$ is called the condensation set. Since this is compact, so is the unique solution $K_\alpha$. In \cite{fra18} Jonathan Fraser called such a solution an \textit{inhomogeneous self-similar set} and calculated the box dimension. This is much harder than to calculate the Hausdorff dimension since the box dimension is not countably stable. In particular the lower box dimension is not even finitely stable. The result for the Hausdorff dimension with respect to the Euclidean metric is calculated very easily due to its countable stability. From \cite[Lemma 3.9]{sn08} we know with the so called \textit{orbital set} $\mathcal{O}$:
\begin{align*}
\mathcal{O}=\bigcup_{w\in\A^\ast_0}G_w\left(\bigcup_{c\in \mathcal{C}}\bigcup_{l=1}^{\rho(c)} e_{c,l}\right)
\end{align*}
that
\begin{align*}
K_\alpha=\Sigma_\alpha\cup \mathcal{O}=\overline{\mathcal{O}}
\end{align*}
\begin{proposition}\label{prop21}
\begin{align*}
\dim_{H,e}(K_\alpha)=\max\{\dim_{H,e}(\Sigma_\alpha),1\}
\end{align*}
\end{proposition}
This value strongly depends on the stretching parameter $\alpha$. The resistance forms, however, will only depend on the topology which does not depend on $\alpha$:
\begin{proposition}\label{prop22}
The $K_\alpha$ are pairwise homeomorphic for different $\alpha$.
\end{proposition}
\begin{proof}
We denote by $G_w^\alpha$ the similitudes which correspond to $K_\alpha$ as well as $e_{c,l}^{\alpha,w}$ for $w\in\A^\ast_0$. We know that $\Sigma_\alpha$ is homeomorphic to $\A^\mathbb{N}$ by the coding map $\iota^\alpha$ which maps $\A^\mathbb{N}$ to $\Sigma_\alpha$ by
\begin{align*}
\iota^\alpha(w)=\bigcap_{n\geq 1} G^\alpha_{w_1,\ldots,w_n}(K_\alpha)
\end{align*}
For $\alpha_1,\alpha_2\in (0,1)$ with $\alpha_1\neq \alpha_2$ we thus know that $\Sigma_{\alpha_1}$ and $\Sigma_{\alpha_2}$ are homeomorphic by the homeomorphism
\begin{align*}
\varphi_{\alpha_1,\alpha_2}:=\iota_{\alpha_2}\circ (\iota_{\alpha_1})^{-1}
\end{align*}
Also we know that $e_{c,l}^{\alpha}$ is homeomorphic to $[0,1]$ for all $c\in \C $ and $l\in \{1,\ldots,\rho(c)\}$. We denote the homeomorphism by $\iota_{c,l}^\alpha$. We can extend $\varphi_{\alpha_1,\alpha_2}$ to all $e_{c,l}^{\alpha_1,w}$ with $w\in \A^\ast_0$ by
\begin{align*}
\varphi_{\alpha_1,\alpha_2}|_{e^{\alpha_1,w}_{c,l}}:=G_w^{\alpha_2}\circ \iota_{c,l}^{\alpha_2}\circ (\iota_{c,l}^{\alpha_1})^{-1}\circ (G_w^{\alpha_1})^{-1}
\end{align*}
We see that the extended $\varphi_{\alpha_1,\alpha_2}$ is a homeomorphism between $K_{\alpha_1}$ and $K_{\alpha_2}$.
\end{proof}
We therefore omit the parameter $\alpha$ in the notation and only write $K$ for the stretched fractal. Similar we only write $\Sigma$ for $\Sigma_\alpha$. This also means that the Hausdorff dimension with respect to the Euclidean metric is not a very good quantity to describe the analysis of $K$. In chapter~\ref{chap_cond_haus} we are going to calculate the Hausdorff dimension with respect to the resistance metric, which is much better suited for this job.\\

We give some further notation. For $w\in \A^\ast_0$ 
\begin{itemize}
\item $K_w:=G_w(K)$,  $K_n:=\bigcup_{w\in\mathcal{A}^n}K_w$
\item $J_n:=\overline{K\backslash K_n}$ 
\item $\Sigma_w:=G_w(\Sigma)$, $\Sigma_n:=\bigcup_{w\in\mathcal{A}^n}\Sigma_w$
\end{itemize}
We take the closure of $K\backslash K_n$ when defining $J_n$ to include the endpoints of the one-dimensional lines. 
$K_w$ is called an $n$-cell if $|w|=n$.

\subsection{Examples}
In this section we include some examples of p.c.f. self-similar fractals that we can stretch.
\subsubsection{Stretched Sierpinski Gasket}
The Stretched Sierpinski Gasket was subject of prior work. It was analyzed geometrically in \cite{af12} and analytically in \cite{af13} and \cite{akt16}. In \cite{afk17} the authors introduced so called completely symmetric resistance forms which satisfy full symmetry of the set. In \cite{haus17} the leading term of the eigenvalue counting function of the associated operators was calculated.

\begin{figure}[H]
\centering
\includegraphics[width=0.5\textwidth]{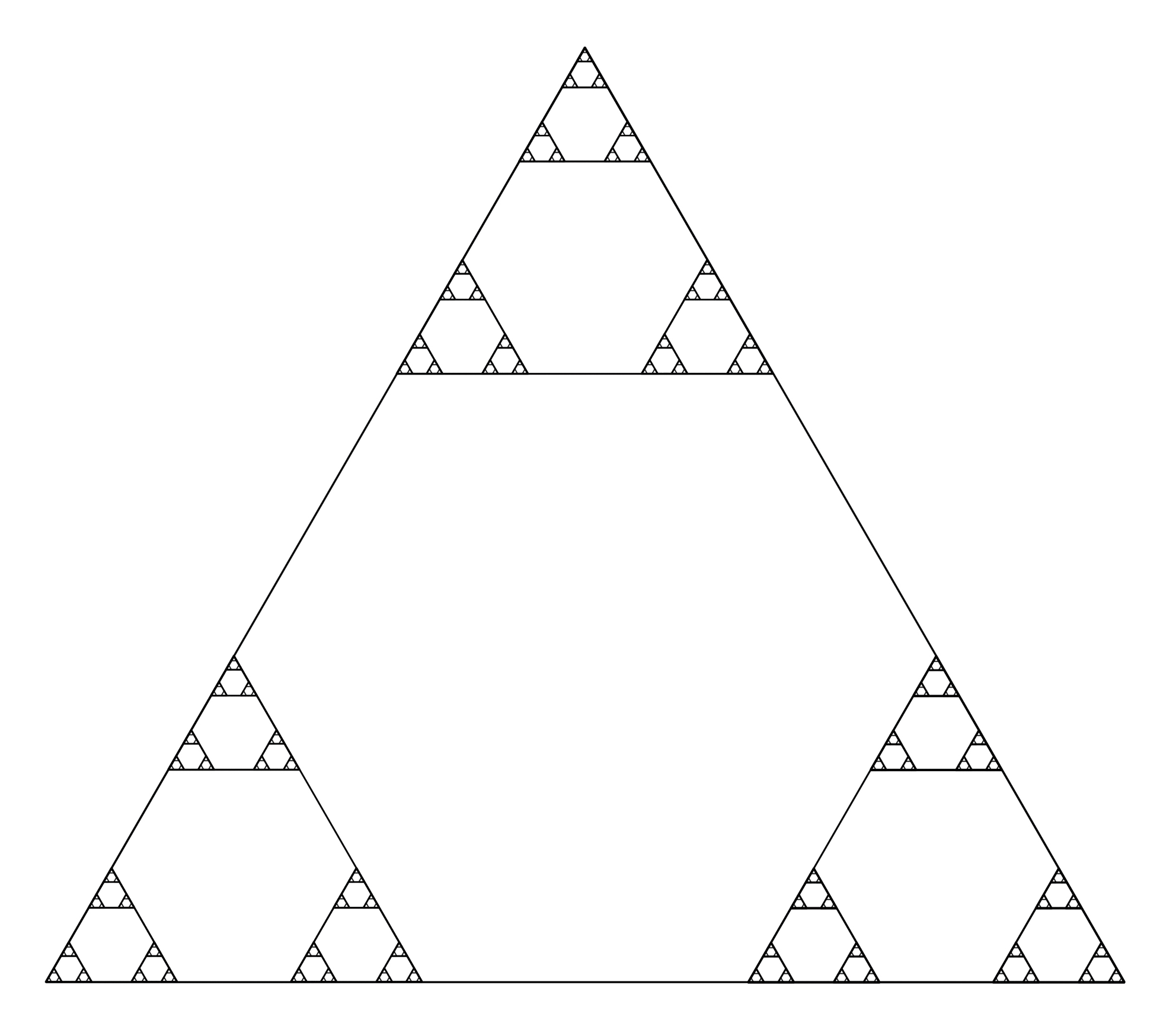}
\caption{Stretched Sierpinski Gasket}
\end{figure}
The Sierpinski Gasket has three similitudes and three critical points which all have multiplicity 2. The post critical set consists of all three fixed points of the similitudes which are the corner points of the big triangle (see \cite[Example 8.2]{kig93}). The words $w^{c,l}$ all have length one since the Sierpinski Gasket is nested.

\subsubsection{Stretched Level 3 Sierpinski Gasket}

\begin{figure}[h]
\centering
\includegraphics[width=0.5\textwidth]{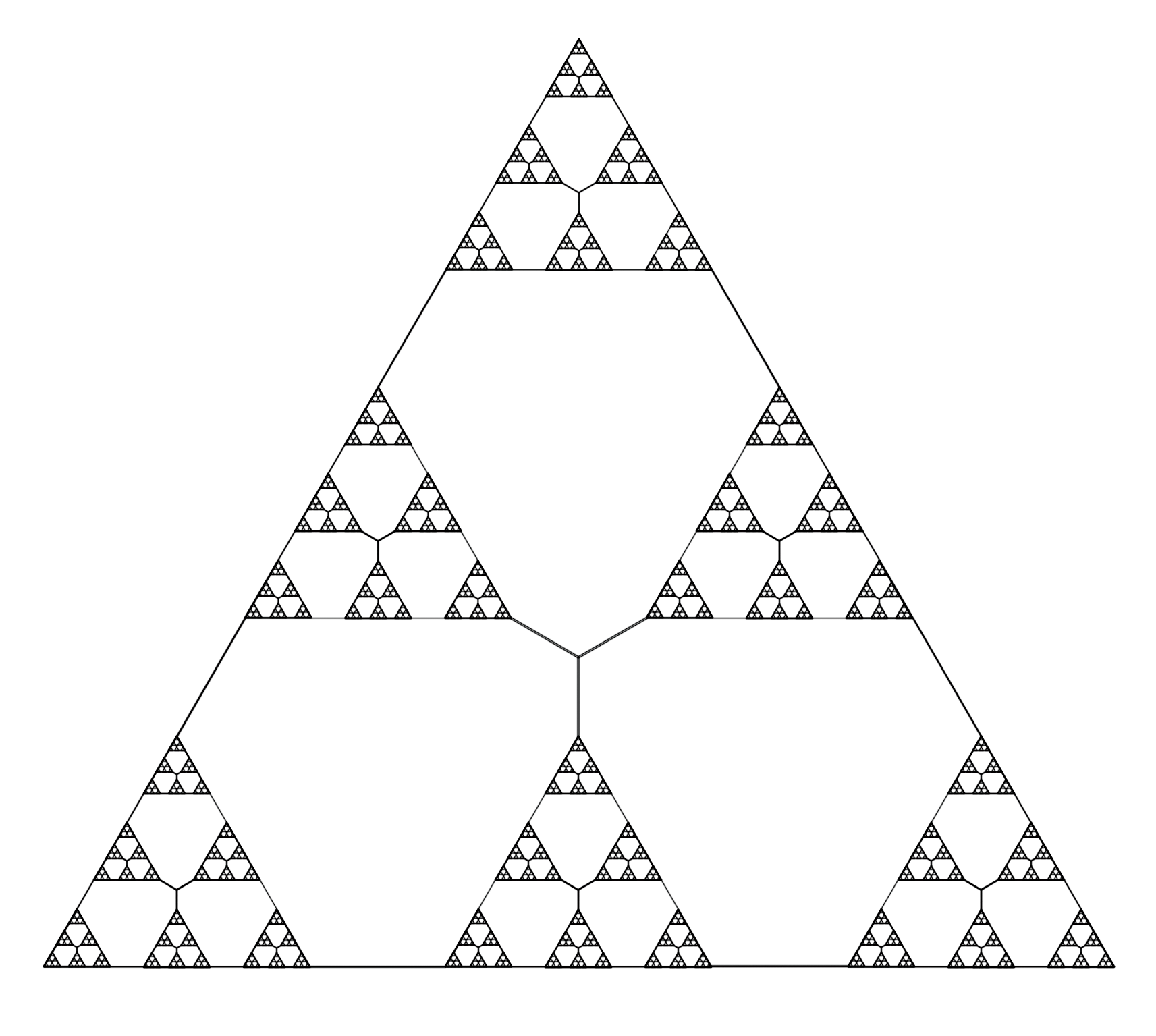}
\caption{Stretched Level 3 Sierpinski Gasket}
\end{figure}
The Level 3 Sierpinski Gasket has six similitudes and seven critical points. Six of the critical points have multiplicity 2, the inner critical point, however, has multiplicity 3. By connecting the copies of all three 1-cells that got stretched away from this point to it, we keep the level of connectedness. The post critical set consists of the essential fixed points which are again the corner points of the outer triangle.


\subsubsection{Stretched Sierpinski Gasket in higher dimensions}
There is a generalization of the Sierpinski Gasket to higher dimensions \cite{dh77}. These are nested fractals that can be stretched. 
\begin{figure}[H]
\centering
\includegraphics[width=0.5\textwidth]{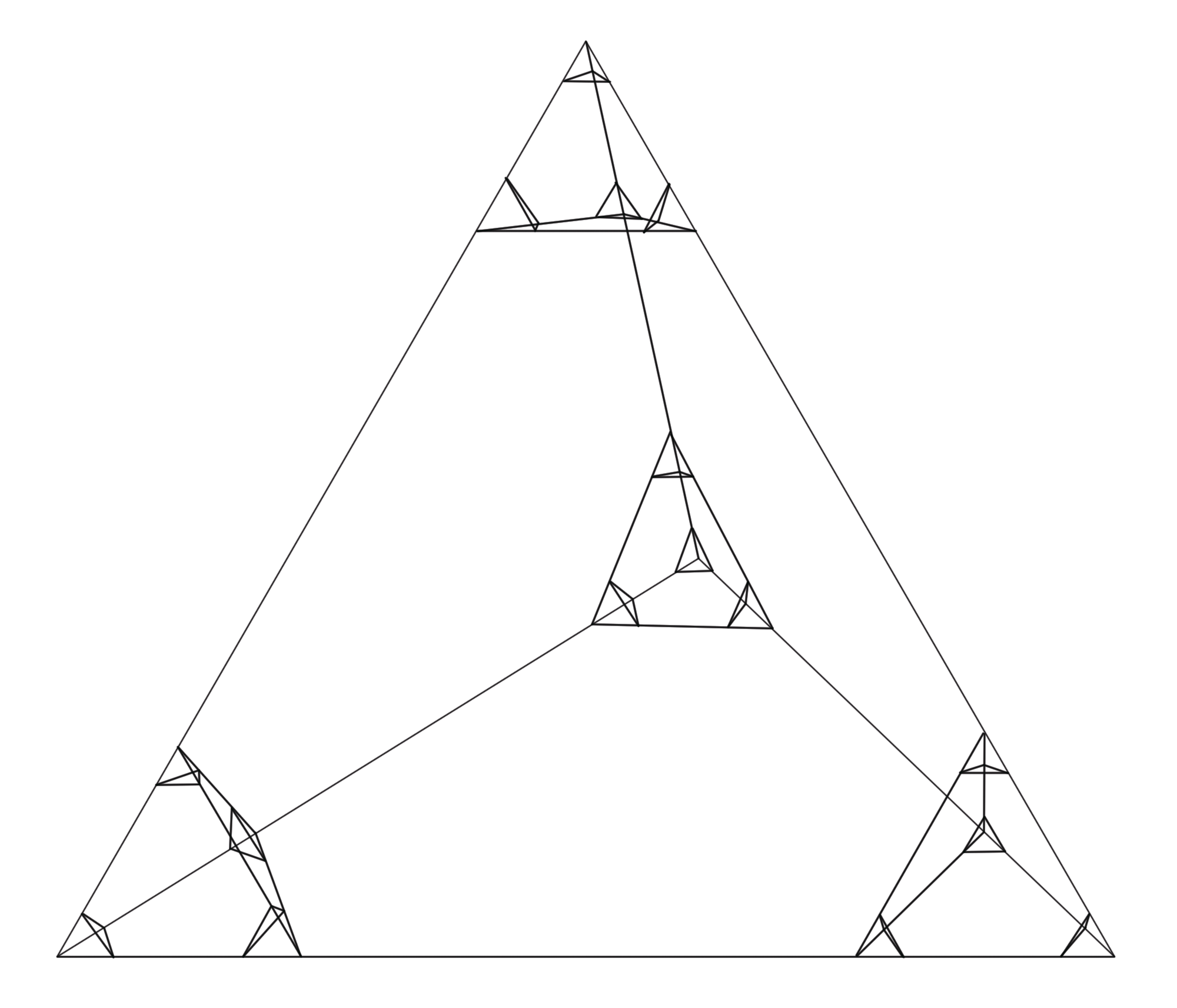}
\caption{Stretched Sierpinski Gasket in $\protect\mathbb{R}^3$}
\end{figure}
In $\mathbb{R}^d$ we have $d+1$ similitudes, where each copy of the stretched fractal is connected to all other copies by connecting lines over the critical points. We have $\frac{d(d+1)}2$ many critical points which all have multiplicity 2. The post critical points are all the fixed points of the similitudes. 

\subsubsection{Stretched Lindstr\o m Snowflake}
\begin{figure}[H]
\centering
\includegraphics[width=0.5\textwidth]{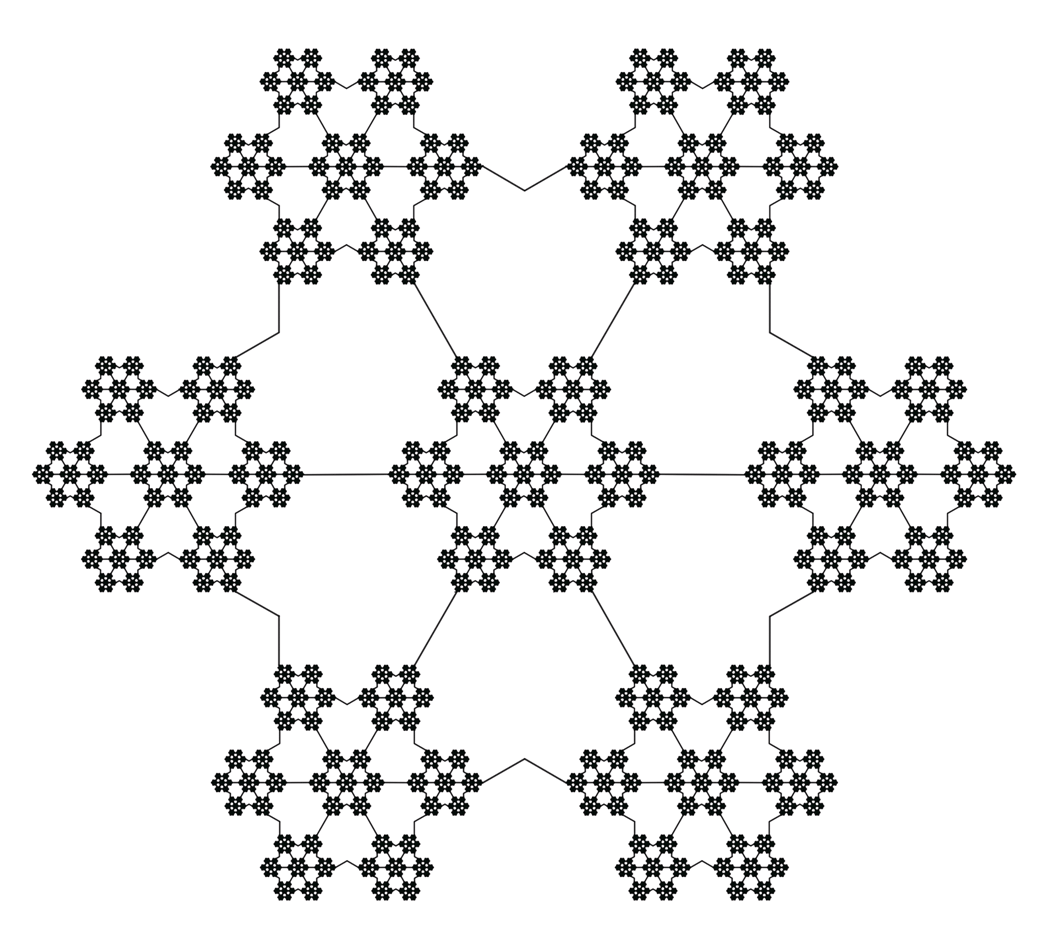}
\caption{Stretched Lindstr\o m Snowflake}
\end{figure}

The Lindstr\o m Snowflake was introduced by Lindstr\o m in \cite{lin90} as an example for nested fractals. It has seven similitudes and 12 critical points which all have multiplicity 2. The post critical set consists of the essential fixed points which are the fixed points of the outer six similitudes.

\subsubsection{Stretched Vicsek Set}

\begin{figure}[H]
\centering
\includegraphics[width=0.5\textwidth]{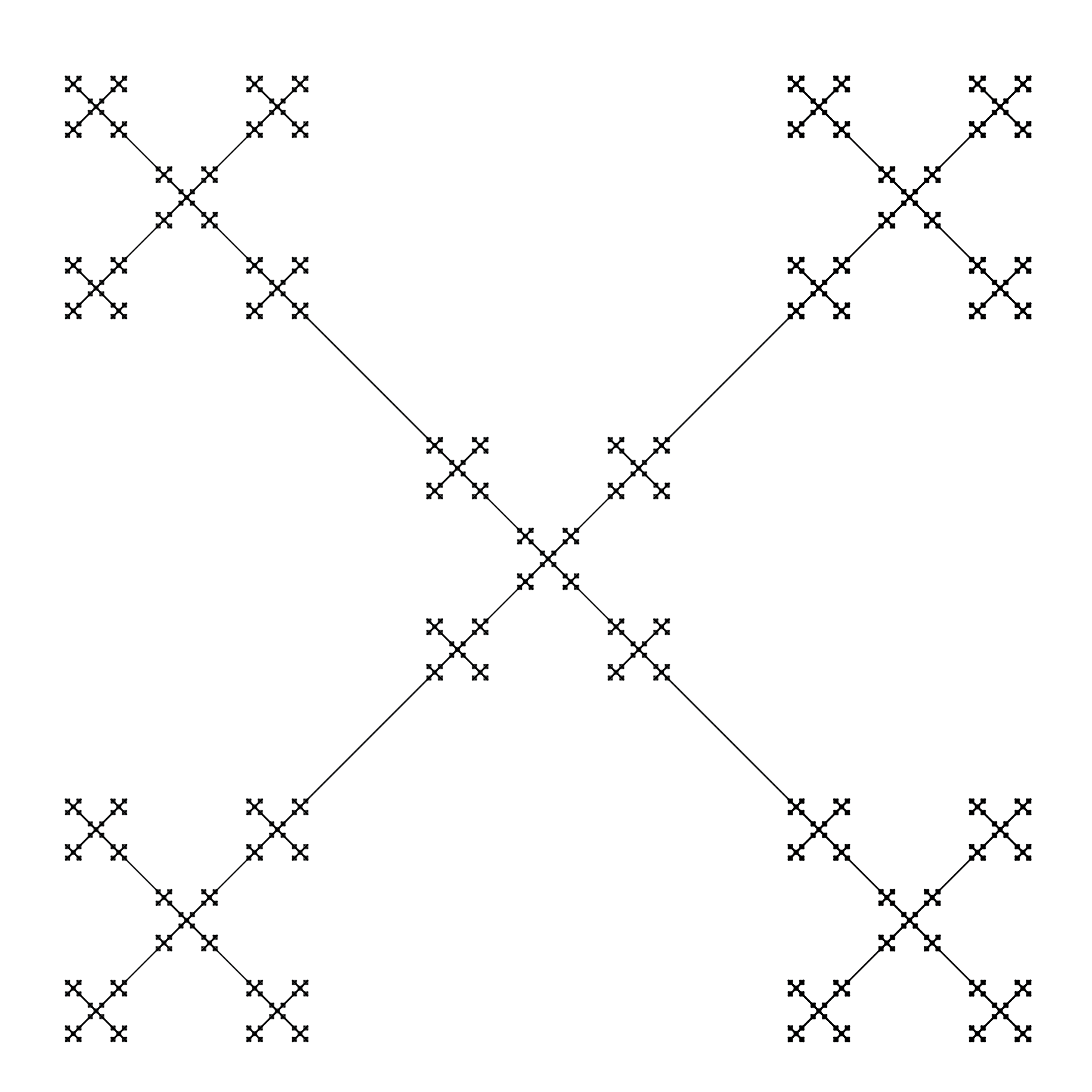}
\caption{Stretched Vicsek Set}
\end{figure}
The Vicsek Set consists of five similitudes and four critical points with multiplicity 2. The post critical points are the fixed points of the outer four similitudes.

\subsubsection{Stretched Hata's tree}\label{hataex}
We have the following similitudes
\begin{align*}
F_1(x)&=\frac 1{\sqrt{12}}\begin{pmatrix}\sqrt{3} &1\\1&-\sqrt{3}\end{pmatrix}\begin{pmatrix}x_1\\x_2\end{pmatrix} \\
F_2(x)&=\frac 23\begin{pmatrix}1&0\\ 0&-1\end{pmatrix}\begin{pmatrix}x_1\\ x_2\end{pmatrix}+\begin{pmatrix}\frac 13\\ 0\end{pmatrix}
\end{align*}
This corresponds to the case $\alpha=\frac 12 +\frac{\sqrt{3}}6 i$ in \cite{hata85}. 
\begin{figure}[H]
\centering
\includegraphics[width=0.7\textwidth]{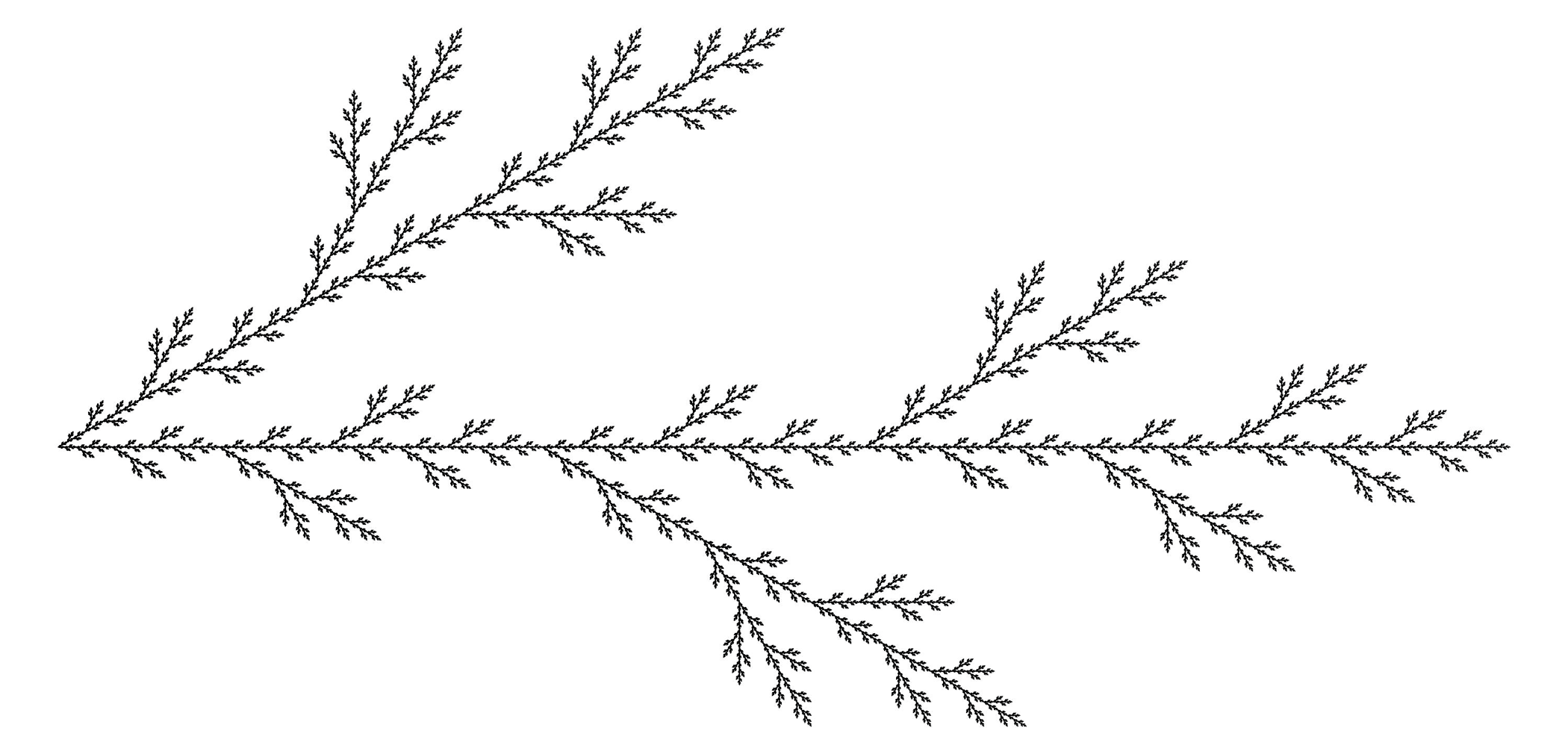}
\caption{Hata's tree}
\end{figure}
There is one critical point at $c=\begin{pmatrix}\frac 13\\0\end{pmatrix}$ with multiplicity 2 and the post critical set is 
\begin{align*}
\P:=\left\{\begin{pmatrix}0\\0 \end{pmatrix}, \begin{pmatrix}1\\0 \end{pmatrix}, \begin{pmatrix} 1/2 \\ 1/\sqrt{12}  \end{pmatrix}\right\}
\end{align*}
We have $\begin{pmatrix} 1/2 \\ 1/\sqrt{12}  \end{pmatrix}=F_1\begin{pmatrix}1\\0 \end{pmatrix}$
and $c=F_1\begin{pmatrix} 1/2 \\ 1/\sqrt{12}  \end{pmatrix}=F_2\begin{pmatrix}0\\0\end{pmatrix}$ which means we have the words $w^{c,1}=2$ and $w^{c,2}=11$. Therefore, even though Hata's tree is not nested (since it lacks the symmetry axiom) it fulfills our conditions (\ref{pcfcond}) and (\ref{pcfcond2}) and thus we are able to stretch it.\\

Stretching this set with $\alpha=\frac 9{10}$ gives us the following two similitudes
\begin{align*}
G_1(x)&=\frac 9{10\sqrt{12}}\begin{pmatrix}\sqrt{3} &1\\1&-\sqrt{3}\end{pmatrix}\begin{pmatrix}x_1\\x_2\end{pmatrix} \\
G_2(x)&=\frac 35\begin{pmatrix}1&0\\ 0&-1\end{pmatrix}\begin{pmatrix}x_1\\ x_2\end{pmatrix}+\begin{pmatrix}\frac 25\\ 0\end{pmatrix}
\end{align*}
According to the construction we need to connect the points $G_1^2\begin{pmatrix}1\\0 \end{pmatrix}$ and $G_2\begin{pmatrix}0\\0\end{pmatrix}$ with $c$. This leads to the connecting lines
\begin{align*}
e_1=\{\lambda(\tfrac 9{10})^2\begin{pmatrix}\frac 13\\0\end{pmatrix}+(1-\lambda)\begin{pmatrix}\frac 13\\0\end{pmatrix}, \lambda\in(0,1)\}\\
e_2=\{\lambda \begin{pmatrix}\frac 13\\0\end{pmatrix} +(1-\lambda) \begin{pmatrix}\frac 25\\0\end{pmatrix}, \lambda\in(0,1)\}
\end{align*}
\begin{figure}[H]
\centering
\includegraphics[width=0.7\textwidth]{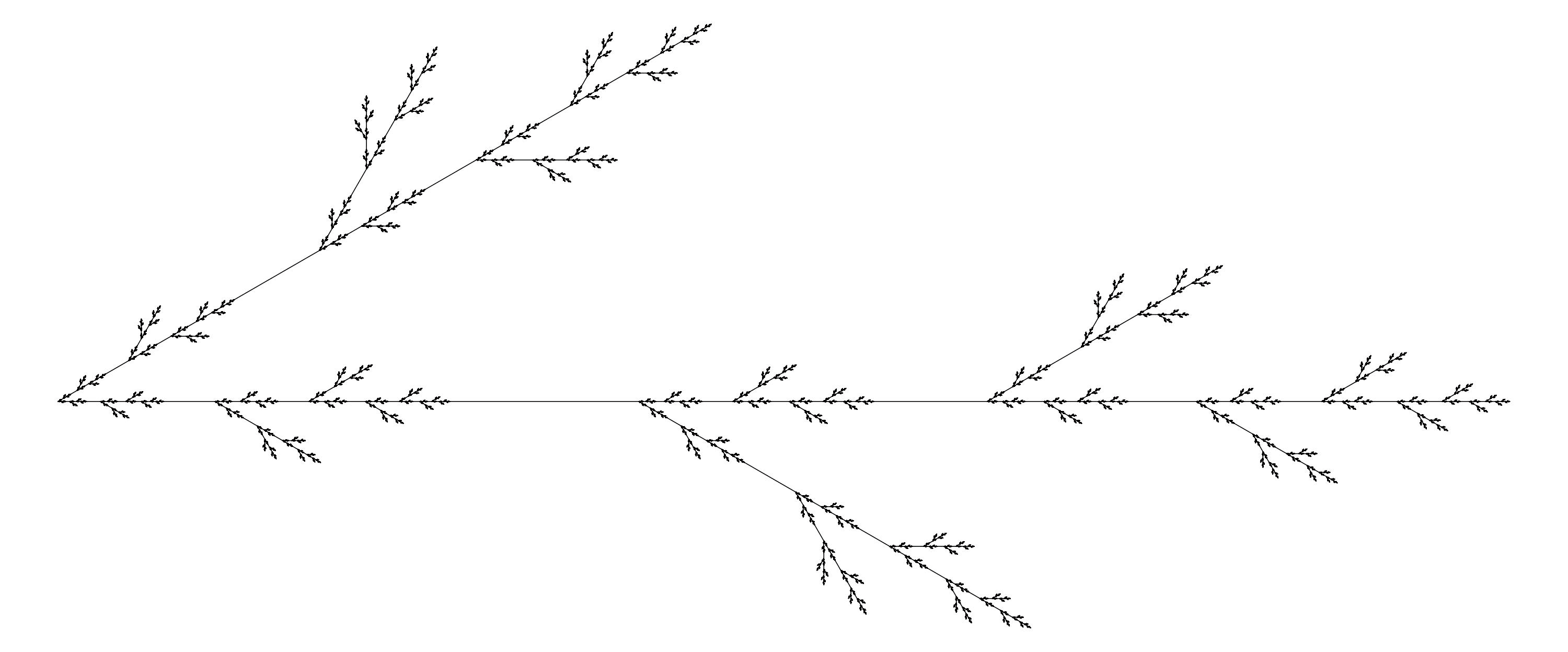}
\caption{Stretched Hata's tree}
\end{figure}

\section{Graph approximation and harmonic structures}\label{chapter3}
To be able to introduce Dirichlet forms on stretched fractals we need to approximate $K$ by a sequence of finite graphs and choose resistances on the graph edges. This is the goal of this chapter.
\subsection{Graph approximation}
We start with the post critical set as vertices and connect all of them pairwise.
\begin{align*}
V_0:=\mathcal{P}, \ E_0:=\{\{x,y\}\ | \ x,y\in V_0, \ x\neq y\}
\end{align*}
In the next graph we have two kinds of vertices. One originate by applying the similitudes $G_j$ to the points of $V_0$:
\begin{align*}
P_1:=\bigcup_{j=1}^N G_j(V_0)
\end{align*}
The other kind are the critical points describing how we want to connect the cells $G_j(V_0)$:
\begin{align*}
C_1:=\mathcal{C}
\end{align*}
The union of these two parts gives us the set of vertices $V_1:=P_1\cup C_1$. \\

Now we want to describe how these vertices are connected. The points of $G_j(V_0)$ should be connected in the same way as $V_0$ was. This gives us the edge relation on $P_1$:
\begin{align*}
E_1^\Sigma :=\{\{G_ix,G_iy\}\ | \ \{x,y\}\in E_0, \ i\in\A\}
\end{align*}
The points that got stretched away from points in $\mathcal{C}$ should again be connected with these points to reflect the geometry of $K$.
\begin{align*}
E_1^I:=E_{1,1}^I:=\{\{c,G_{w^{c,l}}(q^c_l)\}\ | \ c\in\mathcal{C}, l\in\{1,\ldots,\rho(c)\}\}
\end{align*}
We know that $G_{w^{c,l}}(q^c_l)$ is an element of $P_1$. This gives us the graph $(V_1,E_1)$ with vertices $V_1=P_1\cup C_1$ and edge set $E_1:=E_1^\Sigma\cup E_1^I$.

\begin{figure}[H]
\centering
\includegraphics[width=.6\textwidth]{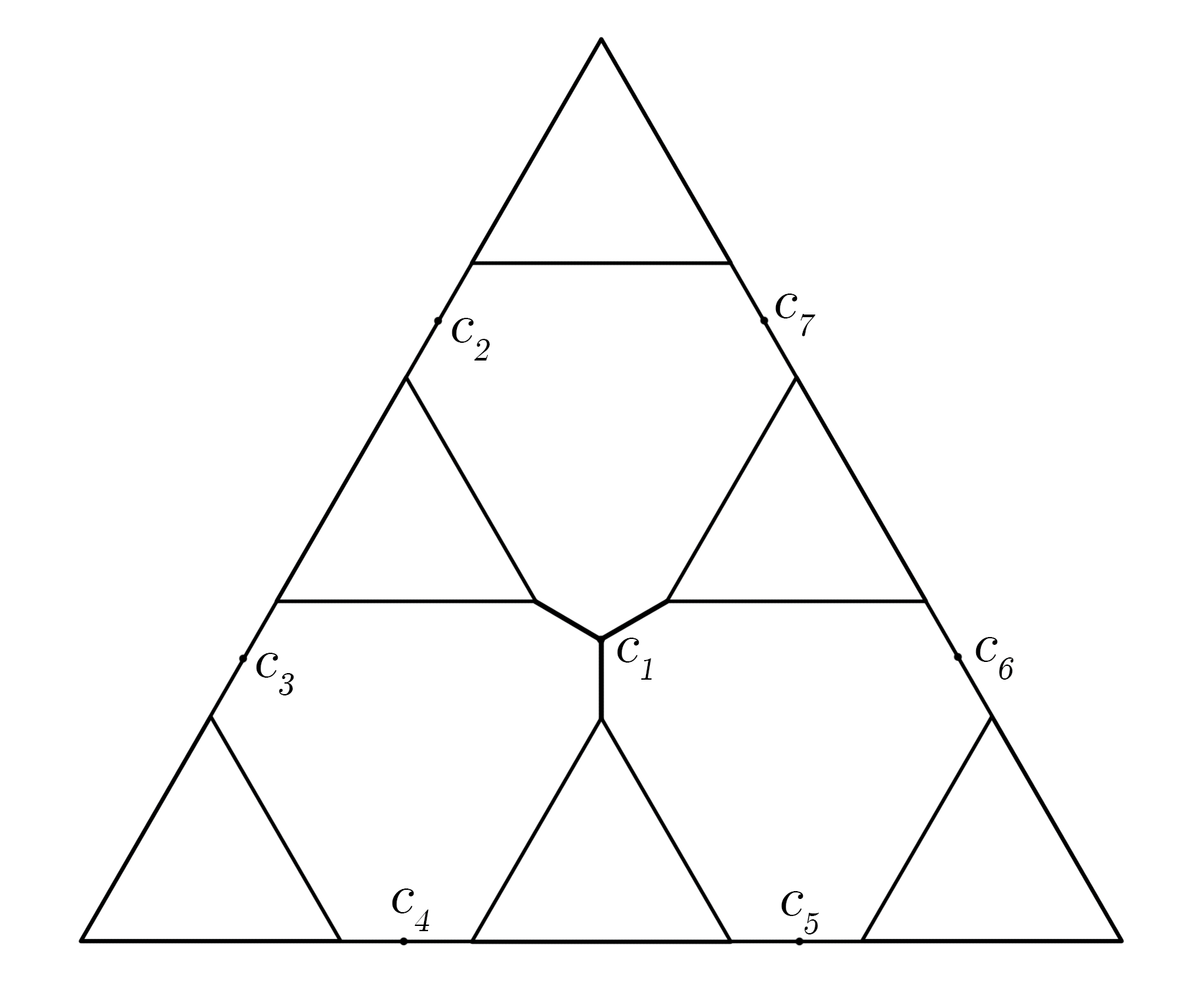}
\caption{$\protect(V_1,E_1)$ for the Stretched Level 3 Sierpinski Gasket}
\end{figure}

We are now ready to define the whole sequence of graphs. In general the vertices will consist of two different kinds of points.
\begin{align*}
P_n&:=\bigcup_{w\in\mathcal{A}^n}G_w (V_0)\\
C_{k,k}&:=\bigcup_{w\in\mathcal{A}^{k-1}}G_w (\mathcal{C})\\
C_n&:=\bigcup_{k=1}^n C_{k,k}\\
\Rightarrow V_n&:=P_n\cup C_n
\end{align*}
Similar we define the edge set.
\begin{align*}
E_n^\Sigma&:=\{\{G_wx,G_wy\}\ | \ \{x,y\}\in E_0, \ w\in\mathcal{A}^n\}\\
E_{k,k}^I&:=\{\{G_wx,G_wy\}\ | \ \{x,y\}\in E_{1,1}^I, \ w\in\mathcal{A}^{k-1}\}\\
E_n^I&:=\bigcup_{k=1}^n E_{k,k}^I\\
\Rightarrow E_n&:=E_n^\Sigma\cup E_n^I
\end{align*}
We will call the edges in $E_n^I$ connecting edges and the ones in $E_n^\Sigma$ fractal edges.\\

This leads to a sequence of graphs $\Gamma_n:=(V_n,E_n)$. We introduce some notation:
\begin{align*}
V_\ast:=\bigcup_{n\geq 0} V_n\\
P_\ast:=\bigcup_{n\geq 1} P_n\\
C_\ast:=\bigcup_{n\geq 1} C_n
\end{align*}
We know from general theory that $P_\ast $ is dense in $\Sigma$.
\begin{align*}
\Rightarrow \overline{V}_\ast=\Sigma\cup C_\ast
\end{align*}
This can be seen with \cite[Lemma 3.9]{sn08} if we choose $\C$ as the condensation set or inhomogeneity.
\subsection{Harmonic structures}
Until now we have the approximating graphs. We need resistances on the edges to define quadratic forms and thus operators. Define resistance functions
\begin{align*}
r_n:E_n\rightarrow [0,\infty]
\end{align*}
that assign each edge in $E_n$ a resistance.\\[.2cm]
We want to choose resistances on $E_n$ in such a way that the electrical networks $(V_n,E_n,r_n)$ are all equivalent and thus a compatible sequence (compare \cite[Def. 2.5]{kig03}). Similar to the self-similar case it suffices to have the existence of $r_0$ and $r_1$ such that $(V_0,E_0,r_0)$ and $(V_1,E_1,r_1)$ are equivalent. Such values (or functions) will be called a harmonic structure in analogy to the self-similar case (compare \cite[Def. 9.5]{kig03}). \\

For an edge $e=\{x,y\}$ we write $G_i(e):=\{G_i(x),G_i(y)\}$. Choose values
\begin{align*}
r_e:=r_0(e)\in(0,\infty], \ \forall e \in E_0\\
\rho_e:=r_1(e)\in(0,\infty) , \ \forall e\in E_1^I
\end{align*}
and $0<\lambda<1$. Then define $r_1$ on the remaining edges in $E^\Sigma_1$ as follows:
\begin{align*}
r_1(G_i(e)):=\lambda r_0(e), \ \forall e \in E_0, i\in \A
\end{align*}
With this we have chosen all values for $r_0$ and $r_1$. Since we allow that $r_0(e)=\infty$ we need to make sure that the network is connected.

 \begin{definition} Let $(V,E)$ be a finite graph and $r: E\rightarrow [0,\infty]$. \\We call an electrical network $(V,E,r)$ connected if for all $p,\tilde p\in V$ there exist $\{p_0,p_1\},\ldots,\{p_{n-1},p_n\}\in E$ with $p_0=p$ and $p_n=\tilde p$ such that $r(\{p_i,p_{i+1}\})<\infty$ for all $i\in\{0,\ldots,n-1\}$.\\
\end{definition}

If the electrical networks $(V_0,E_0,r_0)$ and $(V_1,E_1,r_1)$ are equivalent and the network $(V_0,E_0,r_0)$ is connected we call $(r_0,\lambda,\{\rho_e\}_{e\in E_1^I})$ a harmonic structure for $(G_1,\ldots,G_N)$. \\[.2cm]
\textit{Electrically equivalent} can also be expressed in terms of quadratic forms. We also write $r_0(x,y):=r_0(\{x,y\})$.
\begin{align*}
E_0(f):=\sum_{\{x,y\}\in E_0} \frac 1{r_0(x,y)}(f(x)-f(y))^2\\
E_1(f):=\sum_{\{x,y\}\in E_1} \frac 1{r_1(x,y)}(f(x)-f(y))^2
\end{align*}
with $r_0$ and $r_1$ chosen like before. The trace of $E_1(\cdot)$ on $V_0$ is
\begin{align*}
E_1|_{V_0}(g)=\inf\{E_1(f)\ |  \ f:V_1\rightarrow \mathbb{R}, \ f|_{V_0}=g\}
\end{align*}
We can now give a definition of harmonic structures.

\begin{definition}[Harmonic structure]
$(r_0,\lambda,\{\rho_e\}_{e\in E_1^I})$ is a harmonic structure on~$K$ if and only if
\begin{enumerate}
\item $(V_0,E_0,r_0)$ is connected
\item $E_1|_{V_0}(g)=E_0(g)\quad \text{for all}\quad g:V_0\rightarrow \mathbb{R}$
\end{enumerate}
\end{definition}
For fixed $r_0$ we cannot expect that $\lambda$ and $\{\rho_e\}_{e\in E_1^I}$ are unique. In fact, this is a major feature of stretched fractals (see \cite{afk17} or chapter~\ref{chapexamples}).\\

For the Stretched Level 3 Sierpinski Gasket you can see the resistances in the following figure.\\

\begin{figure}[H]
\centering
\includegraphics[width=1\textwidth]{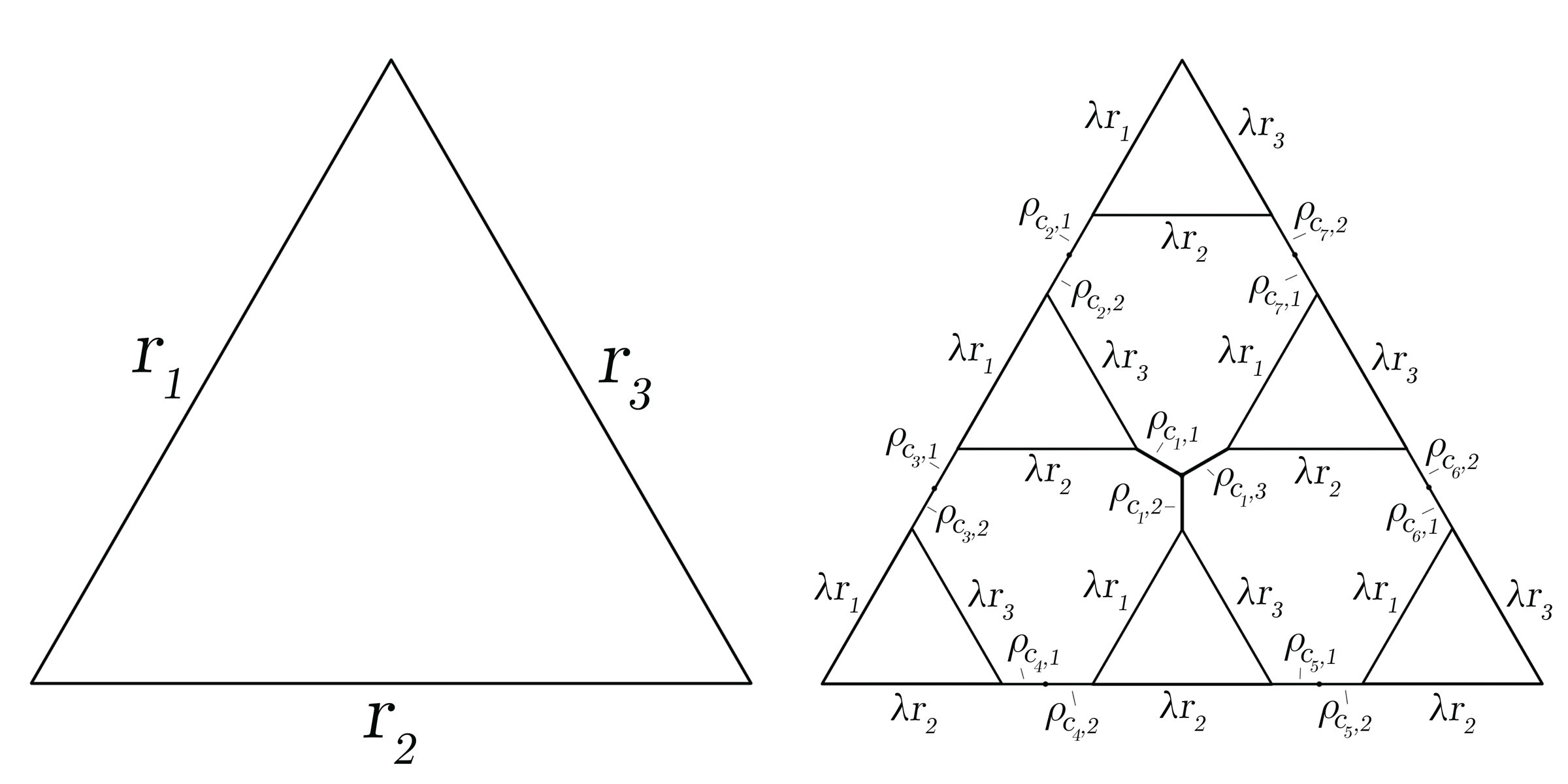}
\caption{Resistances on $\protect(V_0,E_0)$ and $\protect(V_1,E_1)$}
\end{figure}

In the next graph approximation the electrical network $(V_2,E_2,r_2)$ has to be equivalent to $(V_1,E_1,r_1)$ and thus to $(V_0,E_0,r_0)$. The edges in $E_1^I$ are still part of $E_2$ and are not transformed in any way, so they will have the same resistance as in $(V_1,E_1,r_1)$.
The cells $G_iV_0$ get divided in the same fashion as $V_0$ was in the first step but the resistances are now scaled by $\lambda$ compared to the values on $(V_0,E_0,r_0)$. We can therefore choose another \textit{harmonic structure} to get electrically equivalent networks. We choose the same resistances for all $1$-cells.\\

For the Stretched Level 3 Sierpinski Gasket you can see the second graph approximation in the following picture. The dotted lines indicate that the problem of choosing resistances is exactly the same as before in the first graph approximation.

\begin{figure}[H]
\centering
\includegraphics[width=\textwidth]{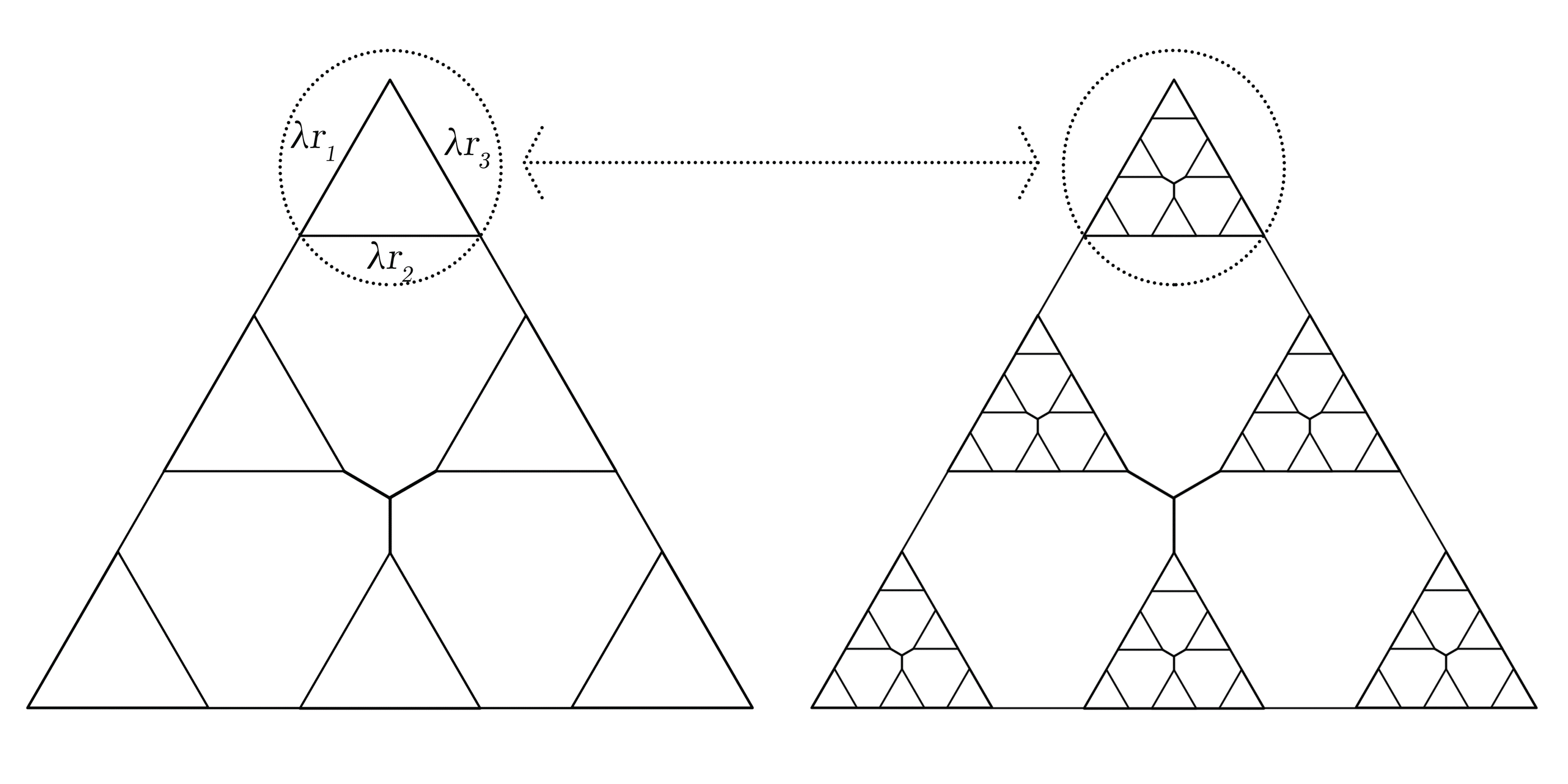}
\caption{Resistances on second graph approximation}
\end{figure}

 We can follow this procedure in each step and thus we have to choose a sequence of harmonic structures:
\begin{align*}
\mathcal{R}:=(r_0,\lambda_i,\{\rho_e^i\}_{e\in E_1^I})_{i\geq 1}
\end{align*}
such that $(r_0,\lambda_i,\{\rho_e^i\}_{e\in E_1^I})$ is a harmonic structure for all $i$. 
Notice that $r_0$ has to be the same for all harmonic structures.\\ 

With this sequence we can define the values for $r_n$:
\begin{enumerate}
\item $r_n$ on $E_n^\Sigma$
\begin{align*}
r_n(G_we):=\lambda_1\cdots\lambda_n r_0(e), \\ e\in E_0, w\in\mathcal{A}^n
\end{align*}
\item $r_n$ on $E_n^I$
\begin{align*}
r_1(e)&=\rho^1_e, \ e\in E^I_1\\[.4cm] 
r_n(G_we)&=\lambda_1\cdots \lambda_{|w|}\rho_e^{|w|+1}, \\
 w&\in \bigcup_{k=1}^{n-1} \mathcal{A}^k, \ e\in E_1^I, \ n\geq 2
\end{align*}
\end{enumerate}
By the definition of harmonic structures $(V_n,E_n,r_n)_{n\geq 0}$ is a sequence of equivalent electrical networks.\\[.2cm]
Since $r_0$ is fixed for the whole sequence of harmonic structures we omit it in the notation of $\R$ whenever we don't explicitly need it. Additionally for the sake of notation we denote $\boldsymbol{\rho}^i:=\{\rho_e^i\}_{e\in E_1^I}$.
\begin{definition}[Regular sequence of harmonic structures]
Let $\R=(\lambda_i,\boldsymbol{\rho}^i)_{i\geq 1}$ be a sequence of harmonic structures (with fixed $r_0$). We call $\R$ a regular sequence of harmonic structures if it fulfills the following two conditions:
\begin{enumerate}[(1)]
\item $\exists \lambda^\ast<1$ such that $\lambda_i\leq \lambda^\ast$ for all $i$
\item $\rho^\ast:=\sup\{\rho\ | \ \rho \in\boldsymbol{\rho}^i, \ i\geq 1\}<\infty$
\end{enumerate}
\end{definition}
The condition $(1)$ is an immediate generalization of regular harmonic structures from \cite[Def. 9.5]{kig03}. The condition $(2)$ is a technical condition that we need to show the existence of resistance forms on $K$.

\subsection{Examples}\label{chapexamples}

We only have to consider the graphs $(V_0,E_0)$ and $(V_1,E_1)$ and choose resistances on the edges in accordance to this chapter such that the electrical networks are equivalent.
\subsubsection{Stretched Sierpinski Gasket}
This has been handled in \cite{afk17}. However, the edge set $E^I_1$ was slightly different since the copies got connected by only one end-to-end edge. 
\begin{figure}[H]
\centering
\includegraphics[width=\textwidth]{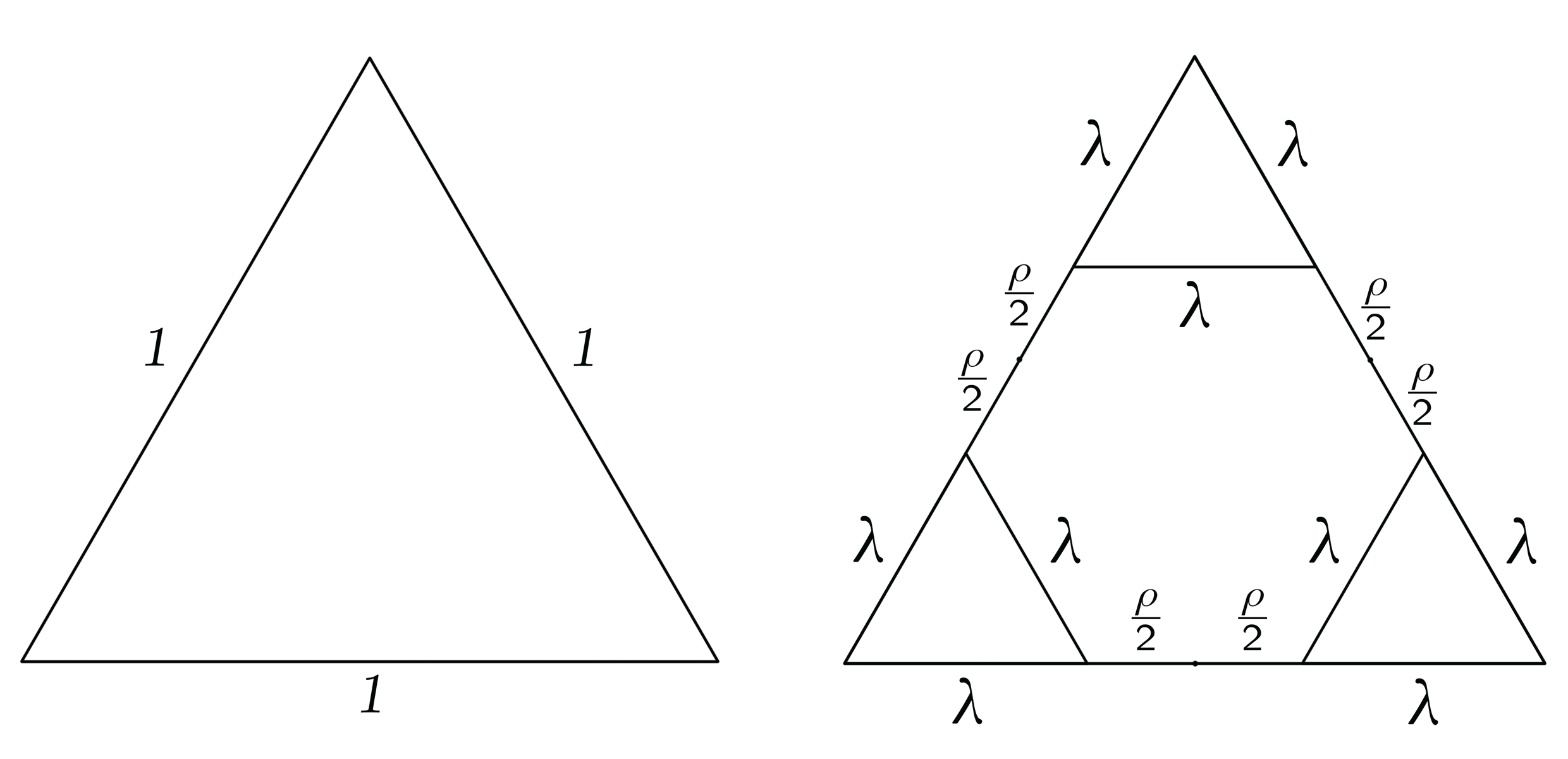}
\caption{Harmonic structure on Stretched Sierpinski Gasket}
\label{harm_ssg}
\end{figure}
Let us choose the resistances as in Figure~\ref{harm_ssg}. That means $r_0\equiv 1$. With this choice we are in the framework of \cite{afk17}. From this work we know that
\begin{align}
\frac 53\lambda+\rho=1 \label{ssgharmeq}
\end{align}
This can be seen by a quick calculation with the help of the $\Delta$-Y transformation. All sequences $\R=(\lambda_i,\boldsymbol{\rho}^i)_{i\geq 1}$ with $\boldsymbol{\rho}^i\equiv \frac {\rho_i} 2$ and $\tfrac 53\lambda_i+\rho_i=1$ for all $i$ are regular sequences of harmonic structures. From (\ref{ssgharmeq}) we know that $0<\lambda_i< \frac 35$ where the upper bound $\tfrac 35$ is exactly the renormalization factor in the self-similar case \cite[Example 8.2]{kig93}.
\subsubsection{Stretched Level 3 Sierpinski Gasket}
We choose the resistances on $(V_0,E_0)$ and $(V_1,E_1)$ in the following way.
\begin{figure}[H]
\centering
\includegraphics[width=\textwidth]{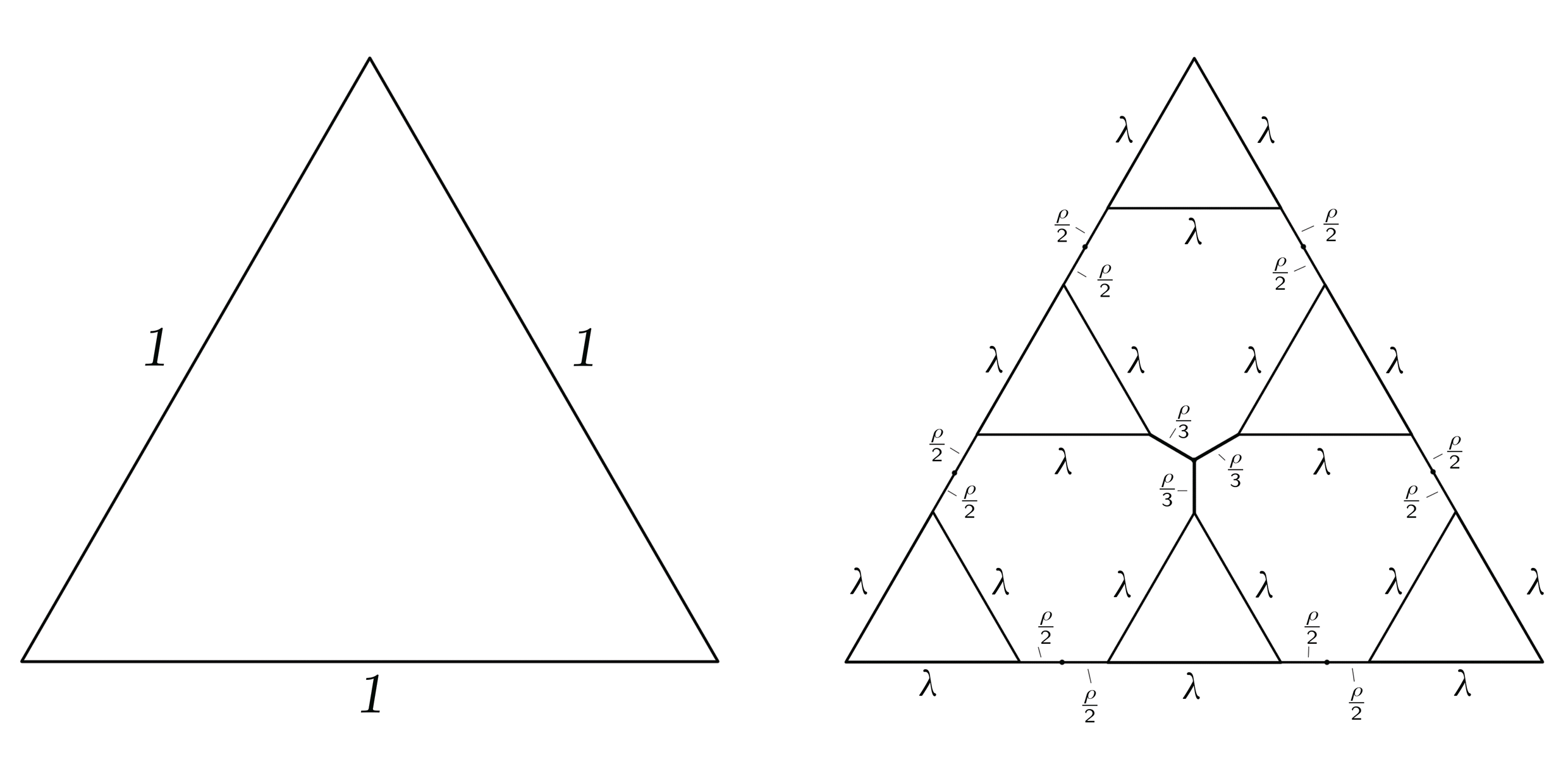}
\caption{Harmonic structure on Stretched Level 3 Sierpinski Gasket}
\label{harm_ssg3}
\end{figure}
These networks are electrically equivalent if and only if the following equation holds.
\begin{align*}
5\lambda^2+\lambda\left(\frac{25}3\rho-\frac 73\right)+5\rho^2-\rho=0
\end{align*}
We can show that this allows pairs $(\lambda,\rho)$ for all $\lambda\in(0,\frac 7{15})$.
If we choose such a harmonic structure in each step we get regular sequences of harmonic structures. The upper limit $\frac 7{15}$ for $\lambda$ is exactly the renormalization in the self-similar case \cite{str00}. 

\subsubsection{Stretched Sierpinski Gasket in higher dimensions}
The graph $(V_0,E_0)$ consists of the complete graph with $d+1$ knots where all edges have resistance $1$. In the first graph approximation we have $d+1$ complete graphs which are all connected over a critical point to all other $d$ complete graphs. The remaining knots are the fixed points of the similitudes. The resistances of the edges in the complete graphs are $\lambda$ and the ones on the connecting edges are $\frac \rho 2$.
\begin{figure}[H]
\centering
\includegraphics[width=\textwidth]{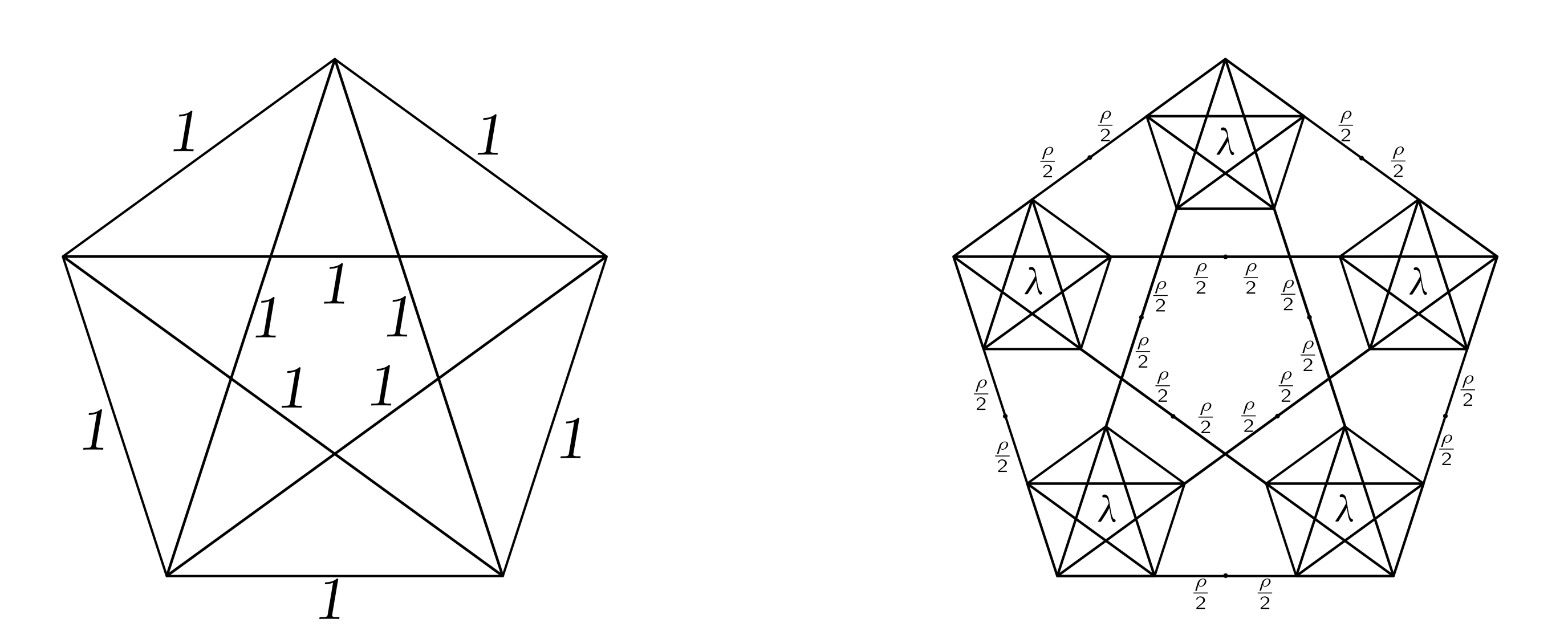}
\caption{Harmonic structure on Sierpinski Gasket in $\protect\mathbb{R}^4$}
\label{harm_ssgd}
\end{figure}
With the help of the star-mesh-transformation, which is a generalization of the $\Delta-Y$-transformation and originally due to Campbell \cite{ca11}, we can show that these networks are equivalent if and only if
\begin{align*}
\lambda\cdot \frac{d+3}{d+1}+\rho=1
\end{align*}
which means that we can reach every $\lambda\in(0,\frac{d+1}{d+3})$ by a pair $(\lambda,\rho)$. The upper limit is the renormalization in the self-similar case \cite{dh77}.

\subsubsection{Stretched Vicsek Set}
Let us choose the following resistances:
\begin{figure}[H]
\centering
\includegraphics[width=\textwidth]{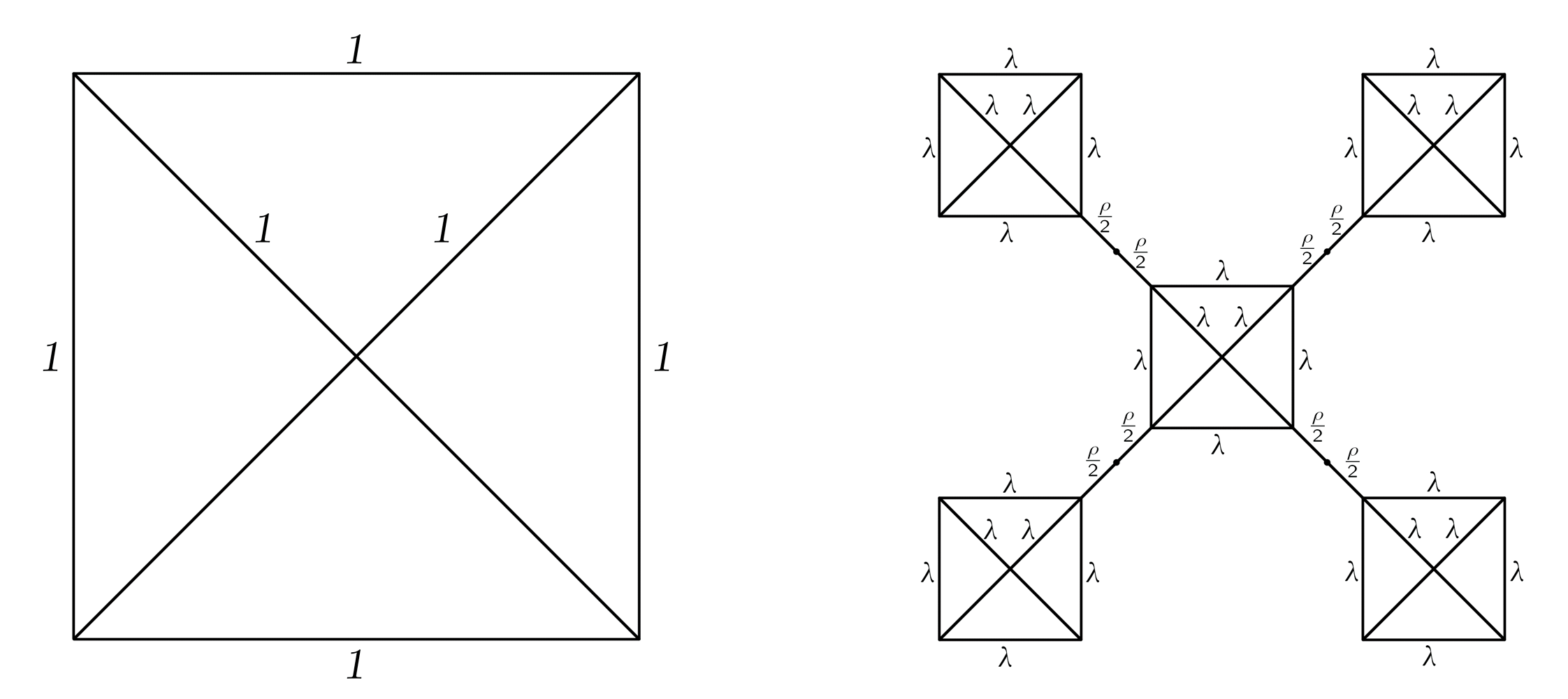}
\caption{Harmonic structure on Stretched Vicsek Set}
\label{harm_svicsek}
\end{figure}
A quick calculation shows that these networks are equivalent if for $(\lambda,\rho)$ it holds that
\begin{align*}
3\lambda+ 4\rho=1
\end{align*}
The choice of $r_0\equiv 1$ comes from the resistances in the self-similar case. In this case this is the only choice which gives us a non-degenerate harmonic structure (see \cite[pp. 85--86]{barl98}). So we use this information and generalize it to the stretched case. We can reach all $\lambda\in(0,\frac 13)$ where again the upper bound is the renormalization in the self-similar case \cite{barl98}.

\subsubsection{Stretched Hata's tree}
We view the graphs $(V_0,E_0)$ and $(V_1,E_1)$ with the following resistances:
\begin{figure}[H]
\centering
\includegraphics[width=\textwidth]{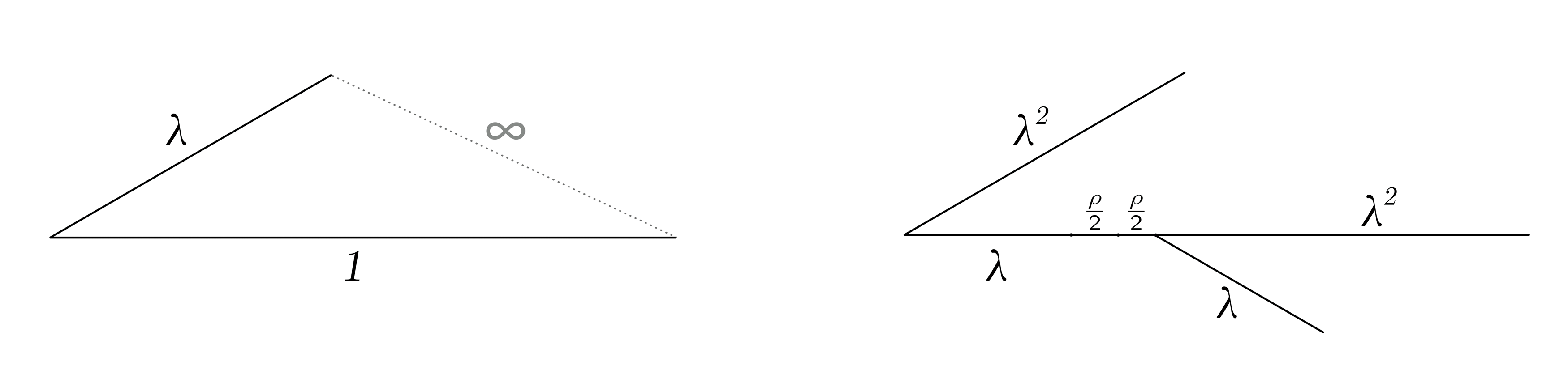}
\caption{Harmonic structure on Stretched Hata's tree}
\label{harm_shata}
\end{figure}
Note that the resistance on the dotted edge is $\infty$ but the graph is still connected. The networks are equivalent if
\begin{align*}
\lambda^2+\lambda+\rho=1
\end{align*}
This is solvable for all $\lambda\in(0,\frac{\sqrt{5}-1}2)$ where the upper bound is the renormalization in the self-similar case \cite[Example 8.4]{kl93}. Note, however, that $r_0$ explicitly depends on $\lambda$. Since $r_0$ has to stay the same when we choose a sequence of harmonic structures we see that here in this case $(\lambda_i)_{i\geq 1}$ has to be constant.

\section{Resistance forms}\label{chapter4}
In this chapter we want to construct resistance forms on stretched fractals. The theory of resistance forms can be found in \cite{kig12} and we want to include a definition at this point.

\begin{definition}\label{defires}
Let $X$ be a set. A pair $(\E,\F)$ is called a resistance form on $X$ if it satisfies the following conditions (RF1) through (RF5):
\begin{enumerate}[(RF1)]
\item $\F$ is a linear subspace of $\ell(X)=\{u|u:X\rightarrow\mathbb{R}\}$ containing constants and $\E$ is a non-negative symmetric quadratic form on $\F$. $\E(u)=0$ if and only if $u$ is constant on $X$. 
\item Let $\sim$ be the equivalence relation on $\F$ defined by $u\sim v$ if and only if $u-v$ is constant on $X$. Then $(\F/{\sim},\E)$ is a Hilbert space.
\item If $x\neq y$, then there exists $u\in \F$ such that $u(x)\neq u(y)$.
\item For any $p,q\in X$,
\begin{align*}
\sup\left\{\frac{|u(p)-u(q)|^2}{\E(u)}\ |\ u\in \F,\E(u)>0\right\}
\end{align*}  
is finite. The above supremum is denoted by $R_{(\E,\F)}(p,q)$ and it is called the resistance metric on $X$ associated with the resistance form $(\E,\F)$.
\item For any $u\in \F, \overline{u}\in \F$ and $\E(\overline{u})\leq \E(u)$, where $\overline{u}$ is defined by
\begin{align*}
\overline{u}(p)=\begin{cases}
1 & \text{if } u(p)\geq 1\\
u(p) & \text{if } 0<u(p)<1\\
0 & \text{if } u(p)\leq 0
\end{cases}
\end{align*}
\end{enumerate}
\end{definition}
\vspace*{1cm}
We consider regular sequences of harmonic structures (with fixed $r_0$)
\begin{align*}
\mathcal{R}=(\lambda_i,\underbrace{\{\rho^i_{\{x,y\}}, \ \{x,y\}\in E_1^I\}}_{\boldsymbol{\rho}^i:=})_{i\geq 1}
\end{align*}
i.e. $\exists \lambda^\ast <1$ with $\lambda_i\leq \lambda^\ast, \ \forall i$. Also we have $r_n$ like before and write $r_n(e)=r_n(\{x,y\})=:r_n(x,y)$.\\

The resistance form will consist of two parts that represent the fractal and the line part that is present in these stretched fractals. The fractal part is very similar to the usual resistance form on the self-similar set (i.e. attractor of the $F_i$).\\[.2cm]
We will first construct a resistance form on $V_\ast$ which doesn't consider the one-dimensional lines. This can be extended to a quadratic form on the closure of $V_\ast$ w.r.t. the resistance metric. Next we show that the Euclidean and resistance metric introduce the same topology, that means the closure of $V_\ast$ is the same with either one. We thus have a resistance form on $\Sigma\cup C_\ast$. The next step is to substitute parts of the resistance form and introduce Dirichlet energies on the one-dimensional lines. This is then shown to be a resistance form on the whole set $K$. Again the resistance metric (on $K$) introduces the same topology as the Euclidean metric.
\subsection{Resistance form on $\protect V_\ast$}
First we define a quadratic form on the approximating graphs that is associated to the energy on the electrical network.\\

\textbf{1. Fractal part}\\[.1cm]
Let $u: V_0\rightarrow \mathbb{R}$.
\begin{align*}
\hat{\mathcal{E}}_{\mathcal{R},0}(u):=Q_{r_0}^\Sigma(u):=\sum_{\{x,y\}\in E_0} \frac 1{r_0(x,y)}(u(x)-u(y))^2
\end{align*}
With this define a quadratic form for $u:V_n\rightarrow \mathbb{R}$ by
\begin{align*}
\hat{\mathcal{E}}_{\mathcal{R},n}^\Sigma(u):=\sum_{w\in\mathcal{A}^n} \frac 1{\lambda_1\cdots \lambda_n}Q_{r_0}^\Sigma (u\circ G_w)
\end{align*}
Use the abbreviation $\delta_n:=\lambda_1\cdots \lambda_n$.\\

\textbf{2. Line part}\\[.1cm]
For $u:V_1\rightarrow \mathbb{R}$.
\begin{align*}
Q_{\boldsymbol{\rho}}^I(u):=\sum_{\{x,y\}\in E_1^I}\frac 1{\rho_{\{x,y\}}} (u(x)-u(y))^2
\end{align*}
Again we can use this to define a quadratic form for $u:V_n\rightarrow \mathbb{R}$:
\begin{align*}
\hat{\mathcal{E}}_{\mathcal{R},n}^I(u):=Q_{\boldsymbol{\rho}^1}^I(u)+\sum_{k=2}^n \frac 1{\lambda_1\cdots \lambda_{k-1}}\underbrace{\sum_{w\in\mathcal{A}^{k-1}}Q_{\boldsymbol{\rho}^k}^I(u\circ G_w)}_{Q_{\boldsymbol\rho^k,k}^I(u):=}
\end{align*}
We denote $\gamma_1:=1$ and $\gamma_k:=\delta_{k-1}=\lambda_1\cdots \lambda_{k-1}$ for $k\geq 2$. Then this writes as follows for $n\geq 1$
\begin{align*}
\hat{\mathcal{E}}_{\mathcal{R},n}^I(u):=\sum_{k=1}^n \frac 1{\gamma_k}Q_{\boldsymbol \rho^k,k}^I(u)
\end{align*}

We can now define a quadratic form that is defined on $\ell(V_n)=\{u|u:V_n\rightarrow \mathbb{R}\}$ for $n\geq 1$.
\begin{align*}
\hat{\mathcal{E}}_{\mathcal{R},n}(u):=\hat{\mathcal{E}}^\Sigma_{\mathcal{R},n}(u)+\hat{\mathcal{E}}^I_{\mathcal{R},n}(u)
\end{align*}
Since $V_n$ is finite, all these quadratic forms are resistance forms and since the graphs form a sequence of equivalent electrical networks the sequence of resistance forms $(\hat{\mathcal{E}}_{\mathcal{R},n},\ell(V_n))_{n\geq 0}$ builds a sequence of compatible resistance forms. That means $(\hat{\mathcal{E}}_{\mathcal{R},n}(u|_{V_n}))_{n\geq 0}$ is a non-decreasing sequence for all $u\in \ell( V_\ast)$ and, therefore, a limit exists in $[0,\infty]$.
\begin{align*}
\hat{\mathcal{E}}_\mathcal{R}(u):=\lim_{n\rightarrow\infty}\hat{\mathcal{E}}_{\mathcal{R},n}(u|_{V_n})
\end{align*}
which is defined on 
\begin{align*}
\hat{\mathcal{F}}_{\mathcal{R}}:=\{u\ |\ u\in\ell(V_\ast), \ \lim_{n\rightarrow\infty}\hat{\mathcal{E}}_{\mathcal{R},n}(u|_{V_n})<\infty\}
\end{align*}
From general theory it follows that $(\hat{\mathcal{E}}_\mathcal{R},\hat{\mathcal{F}}_\mathcal{R})$ is a resistance form on $V_\ast$ (see \cite[Theorem 3.13]{kig12}).
\subsection{Resistance form on $\protect \Sigma\cup C_\ast$} 
By general theory (again \cite[Theorem 3.13]{kig12}) we know that this can be extended to a resistance form on $\overline{V}_\ast$ where the closure is taken w.r.t. the resistance metric of $(\hat{\mathcal{E}}_\mathcal{R},\hat{\mathcal{F}}_\mathcal{R})$. We will denote this metric by $\hat{R}_{\mathcal{R}}(\cdot,\cdot)$. We are, however, interested in a resistance form on $\Sigma\cup C_\ast$. We want to show that the resistance metric and the Euclidean metric are inducing the same topology and, therefore, we can get a resistance form on $\overline{V}_\ast=\Sigma\cup C_\ast$ where the closure is taken with either metric.\\[.2cm]
Denote by $R_{\R,m}(\cdot,\cdot)$ the resistance metric on $V_m$ from $(\hat{\mathcal{E}}_{\mathcal{R},m},\ell(V_m))$. The diameter of a set $X$ w.r.t. a metric $d$ is denoted by $\operatorname{diam}(X,d):=\sup\{d(x,y) \ : \ x,y\in X\}$.
\begin{lemma}\label{lem41} $\operatorname{diam}(V_m,R_{\R,m})\leq c<\infty, \ \forall m$ with a constant $c$ that only depends on $\lambda^\ast$, $\rho^\ast$ and $r_0$.
\end{lemma}
\begin{proof}
Define a constant:
\begin{align*}
C:=\left(\sum_{c \in \mathcal{C}}\rho(c)\right)\rho^\ast+ N\sum_{\substack{e\in E_0\\ r_0(e)<\infty}} r_0(e)
\end{align*}
The first sum is exactly the number of connecting edges in $(V_1,E_1)$ that means we could also write $\#E_1^I$. We then multiply it by an upper bound for all $\rho^i_e$. The second part is the sum of all finite resistances in $E_0$ and then multiplied by the number of similitudes. Note that $\lambda_1\leq 1$. That means $C$ is an upper bound for the sum of all finite resistances on $(V_1,E_1)$ independent of the choice of $\boldsymbol \rho$ from $\boldsymbol{\rho}^i$.\\

Now let $q$ be any point of $V_1$, then it holds that
\begin{align*}
R_{\R,1}(q,p)\leq C, \ \forall p\in V_0
\end{align*}
Since $(V_1,E_1)$ is connected there is a path from $q$ to $p$ where each edge has finite resistance and is only used once. In $C$ we count each edge of $E_1$ with finite resistance and, therefore, get an upper bound of the summed up resistances along this path. Due to the triangle inequality and the fact that the effective resistance is always less or equal to the direct resistance we get the desired inequality.\\

\begin{figure}[H]
\centering
\includegraphics[width=0.6\textwidth]{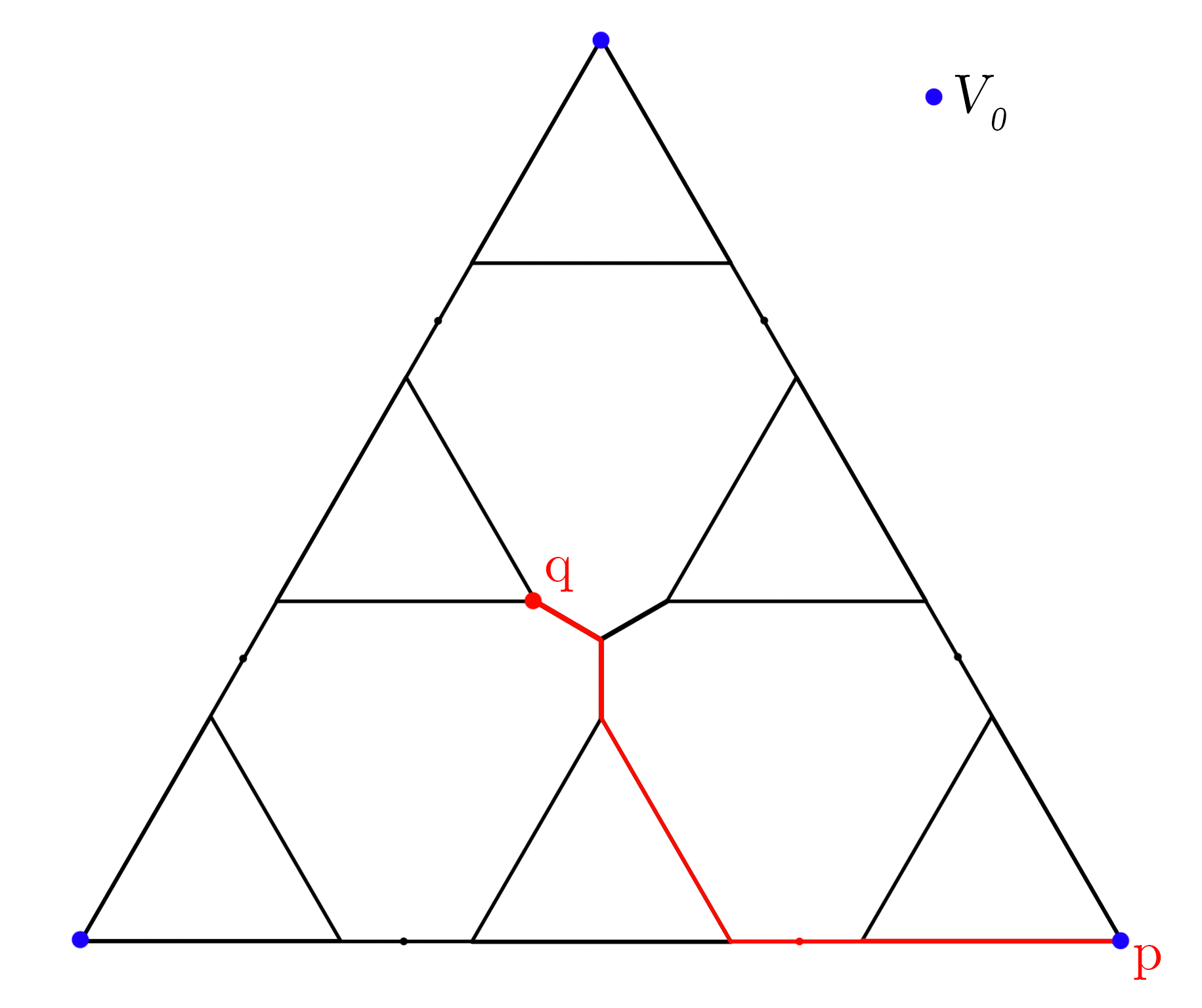}
\caption{Connect $\protect V_1$ with $\protect V_0$}
\end{figure}

Next let $q_1$ be any point of $V_2$ and look for a path to the next point in $V_1$ (let's call it $p_1$). The problem is the same as from $V_1$ to $V_0$ but the resistances are multiplied by $\lambda_1$. That means
\begin{align*}
R_{\R,2}(q_1,p_1)\leq \lambda_1 C
\end{align*}

\begin{figure}[H]
\centering
\includegraphics[width=0.6\textwidth]{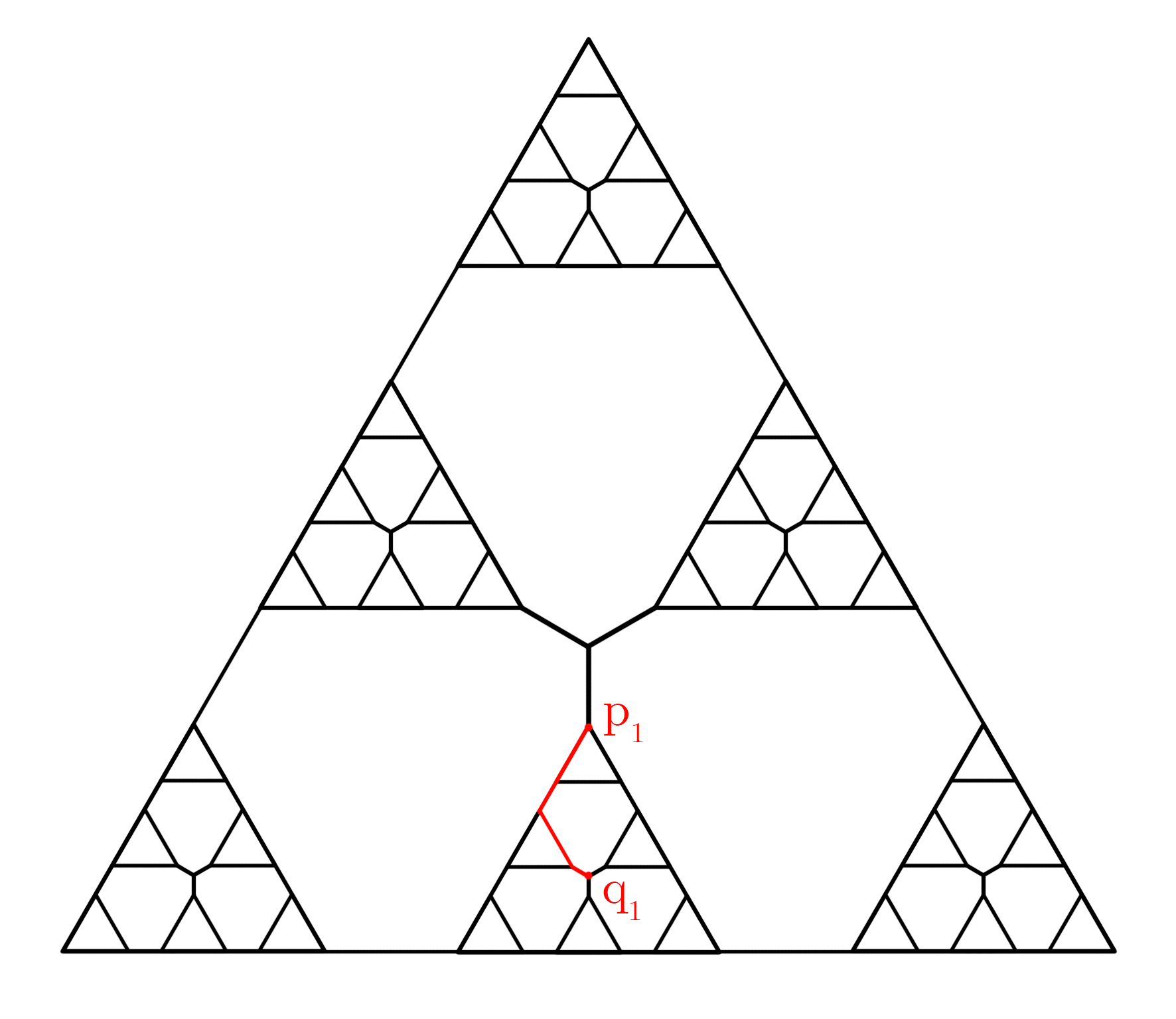}
\caption{Connect $\protect V_2$ with $\protect V_1$}
\end{figure}

Now let $q\in V_n$ and we want to define a sequence of points in $V_k$ from some $p\in V_0$ to $q$. First assume that $q\in P_n$, that means $q=G_{w_1\cdots w_n}(\tilde p)$ for some $\tilde p\in V_0$.
Then define 
\begin{align*}
q_n&:=q\\
q_k&:=G_{w_1\cdots w_k}(\tilde p), \ k =1,\ldots,n-1\\
q_0&:=p\in V_0,\  (\text{arbitrary})
\end{align*}

\begin{figure}[H]
\centering
\includegraphics[width=0.6\textwidth]{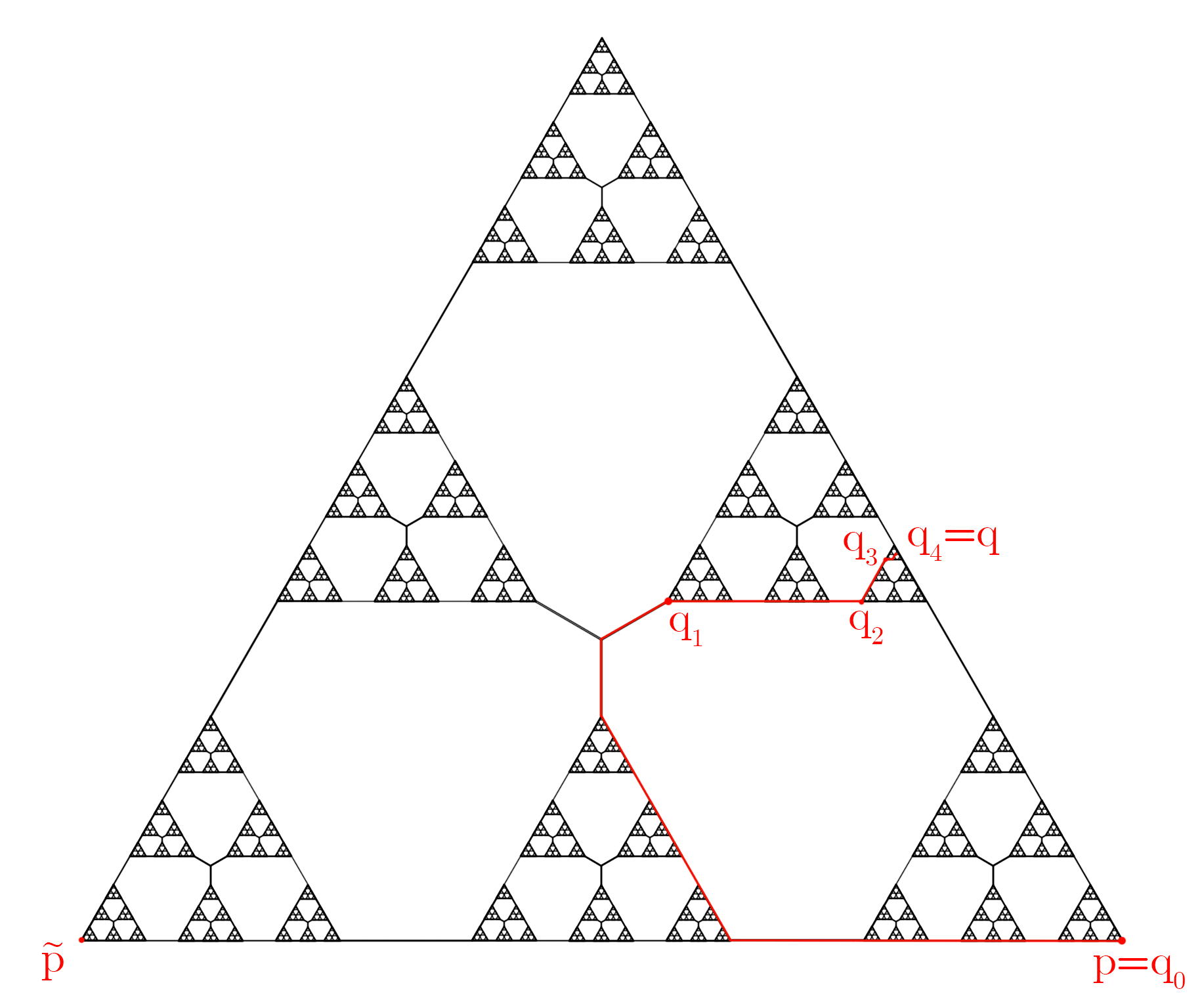}
\caption{Path from $\protect q$ to $\protect p$}
\end{figure}

Actually we can choose any point $\tilde p\in V_0$ for the definition of $q_k$, it is only important that $q_k$ and $q_{k+1}$ are in the same $k+1$-cell.

If $q$ is not in $P_n$, that means $q\in C_n$, we have to add an additional point $q_{n}\in P_n$. Choose one that is connected to $q$ in $\Gamma_n$ and define $q_{n+1}=q$. This is always possible and the resistance is always $\leq \rho^\ast$.
\begin{align*}
\Rightarrow R_{\R,n}(q,p)&\leq \underbrace{p^\ast}_{\text{if $q$ is not in $P_n$}} + \sum_{k=1}^n  R_{\R,n}(q_k,q_{k-1})\\
&\leq p^\ast +\sum_{k=1}^nR_{\R,k}(q_k,q_{k-1})\\
&\leq p^\ast + \sum_{k=1}^n\underbrace{\lambda_1\cdots \lambda_{k-1}}_{:=1 \text{ for } k=1}C\\
&\leq p^\ast + C\sum_{k=1}^n (\lambda^\ast)^{k-1}\\
&\leq p^\ast + C\sum_{k=0}^\infty (\lambda^\ast)^{k}=:\tilde C<\infty
\end{align*}
This holds, since the sequence of harmonic structures is regular and therefore $\lambda^\ast<1$.\\

Now if $q,\tilde q\in V_n$ then choose any point $p\in V_0$.
\begin{align*}
\Rightarrow R_{\R,n}(q,\tilde q)&\leq R_{\R,n}(q,p)+R_{\R,n}(p,\tilde q)\\
&\leq 2\tilde C 
\end{align*}
Therefore
\begin{align*}
\operatorname{diam}(V_n,R_{\R,n})\leq 2\tilde C, \ \forall n
\end{align*}
The points $q_0,\ldots,q_n$ can be chosen very arbitrarily, the only condition is that $q_{k-1}$ and $q_k$ are in the same $k$-cell. Because of this we are allowed to choose the same point in $V_0$ for $q$ and $\tilde q$. \end{proof}
In the self-similar case some rescaling property of the resistance form was very important. We have something similar here, but not quite as nice.
\begin{lemma}[Rescaling of $\hat{\E}_\R$] \label{lemrescalinghat} Let $\mathcal{R}=(\lambda_i,\boldsymbol \rho^i)_{i\geq 1}$ be a sequence of harmonic structures for $r_0$ and let $\mathcal{R}^{(n)}:=(\lambda_{n+i},\boldsymbol\rho^{n+i})_{i\geq 1}$ be the sequence that starts at $n+1$. Then it holds for $u\in \hat{\mathcal{F}}_\mathcal{R}$ that $u\circ G_w\in\hat{\mathcal{F}}_{\mathcal{R}^{(n)}}$ for all $w\in\mathcal{A}^n$ and:
\begin{align*}
\hat{\mathcal{E}}_{\mathcal{R}}(u)=\sum_{w\in\mathcal{A}^n}\frac 1{\delta_n}\hat{\mathcal{E}}_{\mathcal{R}^{(n)}}(u\circ G_w)+\sum_{k=1}^n\frac 1{\gamma_k} Q_{\boldsymbol \rho^k,k}^I(u)
\end{align*}
\end{lemma}
\begin{proof}
\begin{align*}
\hat{\mathcal{E}}_{\mathcal{R},n+m}(u)&=\hat{\mathcal{E}}^\Sigma_{\mathcal{R},n+m}(u)+\hat{\mathcal{E}}^I_{\mathcal{R},n+m}(u)\\
&=\sum_{w\in\mathcal{A}^{n+m}}\frac 1{\delta_{n+m}}Q_{r_0}^\Sigma(u\circ G_w)+\sum_{k=1}^{m+n}\frac 1{\gamma_k} \sum_{w\in\mathcal{A}^{k-1}}Q_{\boldsymbol\rho^k}^I(u\circ G_w)\\
&=\sum_{w\in\mathcal{A}^n} \frac 1{\delta_n}\sum_{\tilde w\in\mathcal{A}^m}\frac 1{\lambda_{n+1}\cdots\lambda_{n+m}}Q_{r_0}^\Sigma(u\circ G_{ w}\circ G_{\tilde w})\\
&\hspace*{.5cm} + \sum_{k=1}^n \frac 1{\gamma_k} Q_{\boldsymbol{\rho}^k,k}(u)\\
&\hspace*{.5cm} + \sum_{w\in\mathcal{A}^n}\frac 1{\delta_n}\sum_{k=1}^{m}\underbrace{\frac 1{\lambda_{n+1}\cdots \lambda_{n+k-1}}}_{:=1 \text{ for } k=1}\sum_{\tilde w\in\mathcal{A}^{k-1}}Q_{\boldsymbol\rho^k}^I(u\circ G_w\circ G_{\tilde w})\\
&=\sum_{w\in\mathcal{A}^n}\frac 1{\delta_n}\left(\hat{\mathcal{E}}^\Sigma_{\mathcal{R}^{(n)},m}(u\circ G_w)+\hat{\mathcal{E}}^I_{\mathcal{R}^{(n)},m}(u\circ G_w)\right)+\sum_{k=1}^n \frac 1{\gamma_k}Q_{\boldsymbol\rho^k,k}(u)\\
&=\sum_{w\in\mathcal{A}^n}\frac 1{\delta_n}\hat{\mathcal{E}}_{\mathcal{R}^{(n)},m}(u\circ G_w)+\sum_{k=1}^n\frac 1{\gamma_k}Q_{\boldsymbol\rho^k,k}(u)
\end{align*}
By taking the limit for $m\rightarrow\infty$ we get the desired result. 
\end{proof}

\begin{lemma} \label{lem43} $\operatorname{diam}(G_wV_\ast,\hat R_{\mathcal{R}})\leq c\delta_n$ for all $w\in\mathcal{A}^n$ with a constant $c$ only depending on $\lambda^\ast$, $\rho^\ast$ and $r_0$.
\end{lemma}
\begin{proof}
From the rescaling we immediately get for all $w\in\mathcal{A}^n$:
\begin{align*}
\frac 1{\delta_n}\hat{\mathcal{E}}_{\mathcal{R}^{(n)}}(u\circ G_w) \leq \hat{\mathcal{E}}_{\mathcal{R}}(u)
\end{align*}
Let $p,q\in G_w(V_\ast)$, that means there exist $x,y\in V_\ast$ such that $p=G_w(x)$ and $q=G_w(y)$. For $u\in\hat{\mathcal{F}}_\R$
\begin{align*}
\frac{|u(p)-u(q)|^2}{\hat{\mathcal{E}}_\mathcal{R}(u)}\leq \delta_n \frac{|u(G_w(x))-u(G_w(y))|^2}{\hat{\mathcal{E}}_{\mathcal{R}^{(n)}}(u\circ G_w)}\leq \delta_n \hat R_{\mathcal{R}^{(n)}}(x,y)
\end{align*}
Since $x,y\in V_\ast$ there exists $k\in\mathbb{N}$ with $x,y\in V_k$. Then the effective resistance between $x$ and $y$ can be calculated with the effective resistance on the graph $(V_k,E_k)$ with the resistance function that belongs to the sequence $\mathcal{R}^{(n)}$. From Lemma~\ref{lem41} we know that there is a constant $c$ that only depends on $\lambda^\ast$, $ \rho^\ast$ and $r_0$, and thus it is valid for $\mathcal{R}^{(n)}$ for all $n$.
\begin{align*}
\hat R_{\mathcal{R}^{(n)}}(x,y)=R_{\R^{(n)},k}(x,y)\leq c
\end{align*}
This leads to
\begin{align*}
\frac{|u(p)-u(q)|^2}{\hat{\mathcal{E}}_\mathcal{R}(u)}\leq \delta_n c, \ \forall u \in\hat{\mathcal{F}}_\mathcal{R}
\end{align*}
Taking the supremum over all $u\in \hat{\mathcal{F}}_\mathcal{R}$ leads to $\hat R_{\mathcal{R}}(p,q)\leq \delta_n c$. This holds for all $p,q\in G_w(V_\ast)$ which gives us the desired result.
\end{proof}

This means the diameter of $n$-cells goes to $0$ for smaller cells (small in terms of big $n$). This is very important to compare Cauchy sequences. Roughly: Cauchy sequences have to be in smaller getting cells (or in some fixed $C_k$). The diameter of small $n$-cells (i.e. big $n$) goes to zero for either resistance and also Euclidean metric. We now want to give an exact proof of this fact.
\begin{lemma} \label{lem44}
The topology of $V_\ast$ is the same with either resistance metric $\hat{R}_\mathcal{R}$ or Euclidean metric. 
\end{lemma}
\begin{proof} We show that $(V_\ast,\hat{R}_\mathcal{R})$ and $(V_\ast,|\cdot|_e)$ have the same Cauchy sequences.\\

First let $(x_i)_{i\geq 1}$ be a Cauchy sequence with respect to the Euclidean metric $|\cdot|_e$ in $V_\ast$. Then there are two possibilities:\\

Since $P_n$ and $C_n$ are positively separated w.r.t. $|\cdot|_e$, there is either an $i_0\geq 1$ with $x_i=x\in C_n$ for all $i\geq i_0$ or there is a $w\in\mathcal{A}^\ast$ $(w=(w_1w_2\cdots))$ with $\forall m \exists i_m\geq 1$ such that $\forall i\geq i_m: x_i\in G_{w_1\cdots w_m}(V_\ast)$.

In fact, this is only true since the $n$-cells are also positively separated. This is due to the stretching and it is not true in the self-similar case!\\

In the first case it is obviously also a Cauchy sequence with respect to the resistance metric. Let's therefore look at the second case. \\

From Lemma~\ref{lem43} we know that $\operatorname{diam}(G_{w_1\cdots w_m}(V_\ast),\hat{R}_{\mathcal{R}})\leq \delta_m c \rightarrow 0$. Therefore, we have for all $k,l\geq i_m$ that $\hat{R}_{\mathcal{R}}(x_k,x_l)\leq \delta_m c$ which makes $(x_i)_{i\geq 1}$ a Cauchy sequence w.r.t. the resistance metric.\\

Now take any Cauchy sequence $(x_i)_{i\geq 1}$ with respect to the resistance metric $\hat{R}_\mathcal{R}$. We have
\begin{align*}
V_\ast=\sum_{w\in\mathcal{A}^n}G_w(V_\ast)\dot\cup C_n
\end{align*}
For $w\in\mathcal{A}^n$ define
\begin{align*}
u&\equiv 1, \ \text{on } G_w(V_\ast)\\
u&\equiv 0, \ \text{on } G_w(V_\ast)^c
\end{align*}
We can easily see that $u\in \hat{\mathcal{F}}_\mathcal{R}$. and thus
\begin{align*}
\hat{R}_{\mathcal{R}}(x,y)\geq \frac{|u(x)-u(y)|^2}{\hat{\mathcal{E}}_\mathcal{R}(u)}=\frac 1{\hat{\mathcal{E}}_\mathcal{R}(u)}>0
\end{align*}
for all $x\in G_w(V_\ast)$ and $y\in G_w(V_\ast)^c$. Therefore
\begin{align*}
\inf\{\hat{R}_{\mathcal{R}}(x,y) \ : \ x\in G_w(V_\ast), y \in C_n\}&>0 \\
\text{and also } \ \inf\{\hat{R}_{\mathcal{R}}(x,y) \ : \ x\in G_w(V_\ast), y \in G_{\tilde w}(V_\ast)\}&>0
\end{align*}
Since we have only finitely many $n$-cells we can even find a common bound for all $n$-cells. This means, that the $n$-cells are positively separated w.r.t. to $\hat{R}_\mathcal{R}$ and also positively separated away from $C_n$. We can therefore use the same argument as before: There is either an $x$ in some $C_n$ such that $(x_i)_{i\geq 1}$ gets trapped in $x$ or we have smaller getting cells where all but finitely many $x_i$ lie. In either case $(x_i)_{i\geq 1}$ is also a Cauchy sequence w.r.t. the Euclidean metric. 
\end{proof}
Due to this Lemma we know that $\overline V_\ast=\Sigma\cup C_\ast$ where the closure is taken with respect to the resistance metric $\hat R_\R$ if $\R$ is a regular sequence of harmonic structures. If $\R$ is not regular we are not able to prove this result. In this case it could happen that $\overline V_\ast$ is a proper subset of $\Sigma\cup C_\ast$ and thus we don't get a resistance form on $\Sigma\cup C_\ast$. 

This is an analogy to the self-similar case (compare \cite[Prop. 20.7]{kig12}). Therefore, the choice of the terms \textit{regular} and \textit{harmonic structure} is justified.
\subsection{Resistance form on $\protect K$}
Until now we have resistance forms on $\Sigma\cup C_\ast$. However, we want to have resistance forms on $K$, that means we need to substitute the differences along the edges that represent connecting lines with some form that considers all values of $u$ along this line and not just the endpoints. For these one-dimensional lines we can use the usual Dirichlet energy. \\

Consider the edges in $E_1^I$. These have a one-to-one correspondence with the connecting lines $e_{c,l}$. For $\{x,y\}\in E_1^I$ we know that $x$ and $y$ are the endpoints of $e_{c,l}$ and thus define $\xi_{e_{c,l}}(t):=\xi_{xy}(t):=tx+(1-t)y$, for $t\in [0,1]$. That means $u\circ \xi_{xy}$ maps $u|_{e_{c,l}}$ to be a function on $[0,1]$. Look at the Dirichlet energy on this line
\begin{align*}
\mathcal{D}_{e_{c,l}}(u):=\mathcal{D}_{xy}(u):=\int_0^1 \left(\frac{d(u\circ \xi_{xy})}{dz}\right)^2dz
\end{align*}
This can be defined if $u|_{e_{c,l}}\circ \xi_{e_{c,l}}$ is in $H^1[0,1]$. We see that this doesn't depend on the orientation of $\xi_{xy}$, therefore, the choice of endpoints of $e_{c,l}$ is not important.\\

For the first summand of the quadratic form we need to sum over all edges in $E_1^I$. 
\begin{align*}
\mathcal{D}_{\boldsymbol\rho}(u):=\sum_{\{x,y\}\in E_1^I} \frac 1{\rho_{\{x,y\}}} \mathcal{D}_{xy}(u)
\end{align*}
Now we define the quadratic form $\mathcal{E}_\mathcal{R}$ that will replace $\hat{\mathcal{E}}_\mathcal{R}$ similar to its definition:
\begin{align*}
\mathcal{E}^I_{\mathcal{R},n}(u):=\sum_{k=1}^n \frac 1{\gamma_k} \underbrace{\sum_{w\in\mathcal{A}^{k-1}} \mathcal{D}_{\boldsymbol \rho^k}(u\circ G_w)}_{\mathcal{D}_{\boldsymbol \rho^k,k}(u):=}
\end{align*}
and the whole form
\begin{align*}
\mathcal{E}_{\mathcal{R},n}(u):=\mathcal{E}_{\mathcal{R},n}^\Sigma(u)+\mathcal{E}_{\mathcal{R},n}^I(u)
\end{align*}

Again we want to define a limit of these quadratic forms but we don't know if this is well defined. We introduce some further notation. 
\begin{align*}
H^1(e^w_{c,l}):=\{u|u: e^w_{c,l}\rightarrow \mathbb{R},  \ u\circ \xi_{e^w_{c,l}} \in H^1[0,1]\}
\end{align*}
\begin{lemma}\label{lem45}
$(\mathcal{E}_{\mathcal{R},n}(u))_{n\geq 0}$ is non-decreasing for $u\in C(K)$ with $u|_{e^w_{c,l}}\in H^1(e^w_{c,l})$ for all $c\in \C, l\in\{1,\ldots,\rho(c)\}$ and $w\in\A^\ast_0$.
\end{lemma}
\newpage \begin{proof}
\textit{(1.)} 
\begin{align*}
\hat{\E}^I_{\mathcal{R},n}(u)\leq \E^I_{\R,n}(u), \ \forall u
\end{align*}
This is true since $(u(1)-u(0))^2\leq \int_0^1 \left(\frac{du}{dx}\right)^2 dx$ for all $u\in H^1[0,1]$. This is the only difference between $\hat{\E}^I_{\mathcal{R},n}$ and $ \E^I_{\R,n}$. The sums and prefactors are the same. This holds in general for $Q_{\boldsymbol\rho}^I(u)\leq \mathcal{D}_{\boldsymbol\rho}(u)$.\\[.2cm]
\textit{(2.)}
\begin{align*}
\E_{\R,n}(u)\leq \E_{\R,n+1}(u), \ \forall u
\end{align*}
Since $(\lambda_{n+1},\boldsymbol\rho^{n+1})$ is a harmonic structure we have
\begin{align*}
Q_{r_0}^\Sigma(u)&\leq \frac 1{\lambda_{n+1}} \sum_{i\in \mathcal{A}}Q_{r_0}^\Sigma(u\circ G_i)+Q_{\boldsymbol \rho^{n+1}}^I(u)\\
\Rightarrow \sum_{w\in\A^n}Q_{r_0}^\Sigma(u\circ G_w)&\leq \frac 1{\lambda_{n+1}} \sum_{w\in \A^{n+1}}Q_{r_0}^\Sigma(u\circ G_w)+\sum_{w\in\A^n}Q_{\boldsymbol \rho^{n+1}}^I(u\circ G_w)
\end{align*}
Applying $(1.)$ for $Q_{\boldsymbol\rho^{n+1}}^I$ and multiplying by $\frac 1{\delta_n}=\frac 1{\gamma_{n+1}}$ on both sides we get
\begin{align*}
\E_{\R,n}^\Sigma(u)\leq \E_{\R,n+1}^\Sigma(u)+\frac 1{\gamma_{n+1}}\sum_{w\in\A^n}\D_{\boldsymbol\rho^{n+1}}(u\circ G_w)
\end{align*}
Now if we add $\sum_{k=1}^n\frac 1{\gamma_k} \D_{\boldsymbol\rho^k,k}(u)$ on both sides we get the desired result.
\end{proof}
We see that taking the limit is a well defined object in $[0,\infty]$ and, therefore, write 
\begin{align*}
\mathcal{E}_{\mathcal{R}}(u):=\lim_{n\rightarrow\infty} \mathcal{E}_{\mathcal{R},n}(u)
\end{align*}
and we define the domain as the functions with finite energy.
\begin{align*}
\mathcal{F}_\mathcal{R}:=\left\{u\ | \ u \in C(K), \ \begin{array}{l} u|_{e^w_{c,l}}\in H^1(e^w_{c,l})\ \forall c\in\C, l\in\{1,\ldots,\rho(c)\},w\in \A^\ast_0, \\  \lim_{n\rightarrow \infty}\mathcal{E}_{\mathcal{R},n}(u)<\infty\end{array}\right\}
\end{align*}

This form fulfills the same rescaling as $\hat{\E}_\R$:
\begin{lemma}[Rescaling of $\E_\R$] \label{lemrescaling}For all $u\in\F_\R$ and $w\in\A^n$ we have $u\circ G_w\in\F_{\R^{(n)}}$ and
\begin{align*}
\E_\R(u)=\sum_{w\in\A^n}\frac 1{\delta_n} \E_{\R^{(n)}}(u\circ G_w) + \sum_{k=1}^n \frac 1{\gamma_k} \D_{\boldsymbol{\rho}^k,k}(u)
\end{align*}
\end{lemma}
\begin{proof}This works exactly the same as for $\hat{\E}_\R$ in Lemma~\ref{lemrescalinghat}.
\end{proof}
We now want to show that $(\E_\R,\F_\R)$ is indeed a resistance form.
\begin{theorem}\label{theoresform} Let $\R=(\lambda_i,\boldsymbol{\rho}^i)_{i\geq 1}$ be a regular sequence of harmonic structures on $K$. Then $(\E_\R,\F_\R)$ is a resistance form on $K$ where the associated resistance metric is inducing the same topology as the Euclidean metric.
\end{theorem}
In order to show Theorem~\ref{theoresform} we have to show (RF1)--(RF5) of Definition~\ref{defires}.
\begin{lemma}\label{lem48} There is a constant $c$ only depending on $\lambda^\ast$, $\rho^\ast$ and $r_0$ such that we have for all $u\in\F_\R$ and $x,y\in K$
\begin{align*}
|u(x)-u(y)|^2\leq c \E_\R(u)
\end{align*}
\end{lemma}
\begin{proof} We have three distinct cases:
\begin{enumerate}[(1)]
\item $x,y\in \overline V_\ast$
\item $x\in \overline V_\ast$, $y\notin \overline V_\ast$
\item $x,y\notin \overline V_\ast$
\end{enumerate}

For case (1) notice: If $u\in \F_\R$ we have that $u|_{\overline V_\ast}\in \hat \F_\R$ since
\begin{align*}
\hat\E_\R(u|_{\overline V_\ast})\leq \E_\R(u)
\end{align*}
Since $(\hat\E_\R,\hat\F_\R)$ can be extended to a resistance form on $\overline V_\ast$ we get from Lemma~\ref{lem43}
\begin{align*}
|u(x)-u(y)|^2\leq c_1 \hat \E_\R(u|_{\overline V_\ast})\leq c_1 \E_\R(u)
\end{align*}
with a constant $c_1$ only depending on $\lambda^\ast$, $\rho^\ast$ and $r_0$.\\

Now consider case (2): $x\in \overline V_\ast$ and $y\notin \overline V_\ast$:
That means $y$ is in some $e^w_{c,l}$ with $c\in\C$, $l\in\{1,\ldots,\rho(c)\}$, $w\in\A^\ast_0$ and in particular it is not one of the endpoints. Let $p$ be one of the endpoints of $e^w_{c,l}$, we may choose $p:=G_w(c)$. Then $p\in \overline V_\ast$ which means
\begin{align*}
|u(p)-u(x)|^2\leq c_1\E_\R(u)
\end{align*}
Now $y,p\in e^w_{c,l}$, the resistance of $e^w_{c,l}$ is $\gamma_{|w|}\rho^{|w|+1}_{c,l}$. We thus have
\begin{align*}
|u(y)-u(p)|^2\leq \underbrace{\gamma_{|w|}\rho^{|w|+1}_{c,l}}_{\leq \rho^\ast} \frac 1{\gamma_{|w|}\rho^{|w|+1}_{c,l}}\D_{e^w_{c,l}}(u)\leq \rho^\ast \E_\R(u)
\end{align*}
The last inequality holds since the Dirichlet energy on $e_{c,l}^w$ is only one part of the whole energy. Since $(a+b)^2\leq 2a^2+2b^2$ we get
\begin{align*}
|u(x)-u(y)|^2&=|u(x)-u(p)+u(p)-u(y)|^2\\
&\leq 2|u(x)-u(p)|^2+2|u(p)-u(y)|^2\\
&\leq 2(c_1+\rho^\ast) \E_\R(u)
\end{align*}

We can use these two cases to handle the last one (3): $x,y\notin \overline V_\ast$ \\
Choose any $p\in \overline V_\ast$, then
\begin{align*}
|u(x)-u(p)|^2\leq 2(c_1+\rho^\ast)\E_\R(u)\\
|u(y)-u(p)|^2\leq 2(c_1+\rho^\ast)\E_\R(u)
\end{align*}
and thus
\begin{align*}
|u(x)-u(y)|^2&\leq 2 |u(x)-u(p)|^2+2|u(y)-u(p)|^2\\
&\leq \underbrace{8(c_1+\rho^\ast)}_{c:=}\E_\R(u)
\end{align*}
Therefore, it holds for all $x,y\in K$ and $u\in \F_\R$:
\begin{align*}
|u(x)-u(y)|^2\leq c\E_\R(u)
\end{align*}
\end{proof}
In analogy to Lemma~\ref{lem43} we can refine these results with the help of the rescaling property.
\begin{corollary} \label{cor49} $x,y\in K_w$ with $w\in \A^n$ and $u\in \F_\R$:
\begin{align*}
|u(x)-u(y)|^2\leq c\delta_n\E_R(u)
\end{align*}
\end{corollary}
\begin{proof} Since the constant $c$ from Lemma~\ref{lem48} only depends on $\lambda^\ast$, $\rho^\ast$ and $r_0$ it holds also for $(\E_{\R^{(n)}},\F_{\R^{(n)}})$. There are $x^\prime, y^\prime \in K$ with $x=G_w(x^\prime)$ and $y=G_w(y^\prime)$.
From the rescaling we know $u\circ G_w\in \F_{\R^{(n)}}$ and thus
\begin{align*}
|u(x)-u(y)|^2&=|u(G_w(x^\prime))-u(G_w(y^\prime))|^2\\
&\leq c\E_{\R^{(n)}}(u\circ G_w)\\
&\leq c\delta_n \E_\R(u)
\end{align*}
\end{proof}
\begin{proof}[Proof of Theorem~\ref{theoresform}:] \ \\
(RF1): $\F_\R$ is a linear space and $\E_\R(u)\geq 0$ is obviously satisfied. If $\E_\R(u)=0$, then $\hat\E_\R(u|_{\overline V_\ast})=0$. Since this is a resistance form on $\overline V_\ast$ we know that $u$ is constant on $\overline V_\ast$. Also we know that $\D_{e^w_{c,l}}(u)=0$ for all $e^w_{c,l}$ and thus $u$ is constant on all of them. Since $G_w(c)\in e^w_{c,l}\cap \overline V_\ast$ the constants have to be the same on all parts and, therefore, $u$ is constant on $K$.\\[.2cm]
(RF2): Fix any $p\in V_0$, then it is enough to show that $\F_{\R,0}:=\{u|u\in\F_\R, \ u(p)=0\}$ is complete with respect to $\E_\R$. Let $(u_n)_{n\geq 1}$ be a Cauchy sequence with respect to $\E_\R$. I.e.
\begin{align*}
\E_\R(u_n-u_m)\rightarrow 0, \ \text{for } n\geq m, \ m\rightarrow\infty
\end{align*}
\begin{align*}
|u_n(x)-u_m(x)|^2&=|(u_n-u_m)(x)-(u_n-u_m)(p)|^2\\
&\leq c\E_\R(u_n-u_m)
\end{align*}
That means we have uniform convergence for $(u_n)_{n\geq 1}$ and, therefore, there is a $u\in C(K)$ with $u_n\rightarrow u$. Since $\D_{e^w_{c,l}}$ is a resistance form itself and $\D_{e^w_{c,l}}(u_n-u_m)\rightarrow 0$ we get that $u|_{e^w_{c,l}}\in H^1(e^w_{c,l})$. \\[.2cm]
It remains to show that $u_n\rightarrow u$ with respect to $\E_\R$ and that $\E_\R(u)<\infty$.
\begin{align*}
\E_{\R,k}(u_n-u_m)\leq \E_\R(u_n-u_m)\leq \underbrace{\sup_{m\geq n}\E_\R(u_n-u_m)}_{<\infty}
\end{align*}
If we let $m$ go to infinity we get
\begin{align*}
\E_{\R,k}(u_n-u)\leq \sup_{m\geq n}\E_\R(u_n-u_m)
\end{align*}
We were able to substitute $u_m$ for $u$ in the limit since in $\E_{\R,k}$ only squared differences of $u$ and Dirichlet energies appear. We already know that $u_m$ converges to $u$ with respect to them. Next we let $k\rightarrow\infty$
\begin{align*}
\E_\R(u_n-u)\leq \sup_{m\geq n}\E_\R(u_n-u_m)
\end{align*}
That means $u_n-u\in\F_\R$ and since $u_n\in\F_\R$ this implies $u\in\F_\R$ because this is a linear space. Also for $n\rightarrow\infty $ we get
\begin{align*}
\E_\R(u_n-u)\rightarrow 0\quad \text{ and thus } \quad u_n\xrightarrow{\E_\R}u.
\end{align*}
(RF3): (1) $x$ or $y\notin \overline V_\ast$ \\
Without loss of generality let $x$ be this point.
Then there exists an $e^w_{c,l}$ with $x\in e^w_{c,l}$ but $x\notin \overline V_\ast$. 
\begin{align*}
\Rightarrow \exists u\in\F_\R: \ u(x)=1, \ u|_{(e^w_{c,l})^c}\equiv 0
\end{align*}
For example we could use linear interpolation between $x$ and the endpoints of $e^w_{c,l}$. Then $u\in \F_\R$ and $u(\tilde y)<1$ for all $\tilde y\neq x$.\\

(2) $x,y\in \overline V_\ast$:\\
We find a $u$ in the extended domain of $\hat F_\R$ with $u(x)\neq u(y)$. We can extend $u$ to a function $\tilde u$ by linear interpolation on all $e^w_{c,l}$. Then 
\begin{align*}
\E_\R(\tilde u)=\hat\E_\R(u)
\end{align*}
and thus $\tilde u\in \F_\R$ with $\tilde u(x)=u(x)\neq u(y)=\tilde u(y)$.\\[.2cm]
(RF4): This follows with Lemma~\ref{lem48}. \\[.2cm]
(RF5): We have $\overline u=(0\vee u)\wedge 1$. It is clear that $\overline u \in C(K)$. Also $\overline u|_{e^w_{c,l}}\in H^1(e^w_{c,l})$.\\
We see that
\begin{align*}
|\overline u(x)-\overline u(y)|^2\leq |u(x)-u(y)|^2, \ \forall u,x,y
\end{align*}
and also 
\begin{align*}
\mathcal{D}_{e^w_{c,l}}(\overline u)\leq \mathcal{D}_{e^w_{c,l}}( u)
\end{align*}
for all $e^w_{c,l}$. 
\begin{align*}
&\E_{\R,n}(\overline u)\leq \E_{\R,n}(u)\\
\Rightarrow \ &\E_\R(\overline u)\leq \E_\R(u)
\end{align*}
$\Rightarrow \overline u \in \F_\R$.\\

So far we have shown that $(\E_\R,\F_\R)$ is a resistance form on $K$. It remains to show that the topologies with respect to either resistance or Euclidean metric are the same.\\

Let $\iota:(K,|\cdot|_e)\rightarrow (K,R_\R)$ be the identity mapping and $(x_n)_{n\geq 1}$ a sequence in $K$ with $x_n\xrightarrow{|\cdot|_e} x$. We have to show that $(x_n)_{n\geq 1}$ also converges to $x$ with the resistance metric to show that $\iota$ is continuous. \\[.2cm]
We have three cases: 
\begin{enumerate}[(1)]
\item $x$ lies in the interior of some $e^w_{c,l}$
\item $x\in C_\ast$
\item $x\in \overline V_\ast\backslash C_\ast=\Sigma$
\end{enumerate}

Consider (1) $\Rightarrow \exists n_0\geq 0 \ : \ \forall n\geq n_0 \ : \ x_n\in e^w_{c,l}$. Let $u\in\F_\R$
\begin{align*}
\frac{|u(x_n)-u(x)|^2}{\E_\R(u)}\leq \frac{|u(x_n)-u(x)|^2}{(\gamma_{|w|}\rho^{|w|+1}_{c,l})^{-1}\D_{e^w_{c,l}}(u)}
\end{align*}
Now $\D_{e^w_{c,l}}$ itself is a resistance form and its associated resistance metric $\frac{|x-y|_e}{\operatorname{diam}(e^w_{c,l},|\cdot|_e)}$. 
\begin{align*}
\Rightarrow \frac{|u(x_n)-u(x)|^2}{\E_\R(u)}\leq \gamma_{|w|}\rho^{|w|+1}_{c,l} \frac{|x_n-x|_e}{\operatorname{diam}(e^w_{c,l},|\cdot|_e)}
\end{align*}
\begin{align*}
\Rightarrow R_\R(x_n,x)\leq \frac{\gamma_{|w|}\rho^{|w|+1}_{c,l}}{\operatorname{diam}(e^w_{c,l},|\cdot|_e)} |x_n-x|_e \xrightarrow{n\rightarrow\infty} 0
\end{align*}

(2) $x\in C_\ast$ i.e. $x=G_w(c)$ for some $c\in\C$ and $w\in\A^\ast_0$. Then $\exists n_0\geq 0 $ such that
\begin{align*}
\forall n \geq n_0\ : \ x_n\in \bigcup_{l\in \{1,\ldots,\rho(c)\}}e^w_{c,l}
\end{align*}
That means it may jump around the various lines that are connected to $x=G_w(c)$ in this ``spider-net''.
We decompose the sequence $(x_n)_{n\geq 1}$ into various subsequences $\{x_n| \ x_n\in e^w_{c,l}\}$, $\forall l$ and $\{x_n|\ x_n=x\}$. For the latter it is clear that $R_\R(x_n,x)\rightarrow 0$ and for the first we can apply (1). We thus have $R_\R(x_n,x)\rightarrow x$ for all subsequences and thus for the whole sequence $(x_n)_{n\geq 1}$.\\

(3) There is a word $w\in \A^\mathbb{N}$ such that $x=\lim_{m\rightarrow\infty} G_{w_1\cdots w_m}(p)$. \\
Now either (I) $\forall m$ we have $x_n\in G_{w_1\cdots w_m}(K)$ for $n$ big enough \\
or (II) the sequence $(x_n)_{n\geq 1}$ can be divided into two parts where the first contains all points that behave like (I) and in the second are all points that do not (i.e. are in some edges $e^w_{c,l}$). \\
For the first case (I) we know that the diameter of $n$-cells goes to $0$ by Corollary~\ref{cor49} and thus it converges in resistance metric. For the second case (II) we can apply the ideas we already introduced. \\

That means the identity map $\iota: (K,|\cdot|_e)\rightarrow(K,R)$ is continuous. Since $(K,|\cdot|_e)$ is compact, so is $(K,R)$ and thus $\iota^{-1}$ is also continuous. Therefore, the topologies are the same.
\end{proof}

\section{Measures and operators}\label{chap_meas_oper}
Until now we have resistance forms. To get Dirichlet forms and thus operators we need to introduce measures. These measures have to fulfill some requirements. They have to be locally finite (i.p. finite due to the compactness of $K$) and to be supported on the whole set $K$.
\subsection{Measures}
We want to describe the measures on $K$ as the sum of a fractal- and a line-part in accordance to the geometric appearance of $K$. \\

It is clear how the fractal part of the measure has to look like. We want as much symmetry as possible. Therefore, we use the normalized self-similar measure on $K$ which distributes mass equally onto the $m$-cells:
\begin{align*}
\mu_\Sigma(K_w)=\mu_\Sigma(\Sigma_w)=\left(\frac 1N\right)^{|w|}
\end{align*}
This gives us a measure on $K$ that fulfills:
\begin{align*}
\mu_\Sigma=\sum_{i=1}^N \frac 1N \cdot \mu_\Sigma \circ G_i^{-1}
\end{align*}
We see, however, that $\mu_\Sigma$ is only supported on the fractal dust $\Sigma$ which is the attractor of $(G_1,\ldots,G_N)$. This is a proper subset of $K$. Therefore, $\mu_\Sigma$ doesn't have full support. That means we can't use $\mu_\Sigma$ to get Dirichlet forms. This measure is too rough to measure the one-dimensional lines. We therefore need another measure that is able to measure these lines.\\

For this line part we want to ignore the length of $e^w_{c,l}$ according to the one-dimensional Lebesgue measure $\lambda^1$. Since we are analyzing $K$ only topologically, this value is not giving us much information. We need a measure that assigns these lines some weight such that these weights are finite when summed up. For the initial lines $e_{c,l}$ we set 
\begin{align*}
\mu_I(e_{c,l}):=a_{c,l}
\end{align*} with $a_{c,l}>0$ for $c\in \C$ and $l\in \{1,\ldots,\rho(c)\}$.\\

How should this measure scale for lines $e^w_{c,l}$. For symmetry reasons we want that the scaling is independent of the $m$-cell that we consider. We thus define
\begin{align*}
\mu_I(e^w_{c,l}):=\beta^{|w|}a_{c,l}
\end{align*}
with some $\beta>0$. We easily see that we need $\beta<\frac 1N$ to get a finite measure on $J:=\bigcup_{n\geq 1}J_n$. On the lines we define the measure as follows
\begin{align*}
\mu_I|_{e^w_{c,l}}:= \beta^{|w|}a_{c,l} \cdot \frac{\lambda^1}{\lambda^1(e^w_{c,l})}, \ w\in \A^\ast_0
\end{align*}
This means it behaves like the one-dimensional Lebesgue measure on $e^w_{c,l}$ but it is normalized and then scaled by $\beta^{|w|}a_{c,l}$. Therefore, it doesn't depend on the value of $\lambda^1(e^w_{c,l})$. If $\beta<\frac 1N$ we have $a:=\mu_I(J)<\infty$. We choose the $a_{c,l}$ such that $\mu_I(J)=1$, by dividing with $a$. 
Calculating $\mu_I(J)$ gives us
\begin{align*}
\sum_{\substack{c\in\C\\l\in\{1,\ldots,\rho(c)\}}} a_{c,l}=1-\beta N
\end{align*}
If $\beta\rightarrow 0$ then more mass is distributed to bigger edges (big in the sense of short words $w$) and if $\beta\rightarrow \frac 1N$ the mass is distributed more equally which displays the geometry better. As a matter of fact $\beta=\frac 1N$ is not possible, that means the real geometry of $K$ is distorted by $\mu_I$.\\

We know that $J$ is dense in $K$. Therefore, $\mu_I$ has full support and can be used to get Dirichlet forms. The measures that we will consider will be convex combinations of the two measures:\\

Let $\eta \in (0,1]$
\begin{align*}
\mu_\eta:=\eta \mu_I+(1-\eta)\mu_\Sigma
\end{align*}
$\eta=0$ is not allowed, since $\mu_\Sigma$ doesn't have full support. $\mu_1=\mu_I$, however, can be used alone. In this case we don't have any fractal part in the measure and this will reflect in the spectral asymptotics.\\

Now we want to know how $\mu_\eta$ scales with $w$. For the fractal part this is clear. We have
\begin{align*}
\mu_\Sigma(K_w)=\left(\frac 1N\right)^{|w|}
\end{align*}
For the line part we have the following.
\begin{align*}
1&=N^{|w|} \mu_I(K_w)+ \sum_{\substack{c\in \C\\l\in\{1,\ldots,\rho(c)\}\\ \tilde w: |\tilde w|<|w|}}\mu_I(e^{\tilde w}_{c,l})\\
&=N^{|w|} \mu_I(K_w)+\sum_{k=0}^{|w|-1}N^k\cdot\sum_{\substack{c\in\C\\l\in\{1,\ldots,\rho(c)\}}}a_{c,l}\beta^k\\
&=N^{|w|} \mu_I(K_w)+\frac{1-(\beta N)^{|w|}}{1-\beta N}\sum_{\substack{c\in\C\\l\in\{1,\ldots,\rho(c)\}}}a_{c,l}\\
&=N^{|w|} \mu_I(K_w)+1-(\beta N)^{|w|}
\end{align*}
That means $\mu_I(K_w)=\beta^{|w|}$. For $\mu_\eta$ this leads to
\begin{align*}
\beta^{|w|}\leq \mu_\eta(K_w)\leq \left(\frac 1N\right)^{|w|}
\end{align*}
\subsection{Operators}
With these measures we can define Dirichlet forms and, therefore, operators on $L^2(K,\mu_\eta)$. Let $\R$ be a regular sequence of harmonic structures. \\
\indent Now since $(K,R_\R)$ is compact we have $\D_\R:=\overline{\F_\R\cap C_0(K)}^{\E^{\frac 12}_{\R,1}}=\F_\R$.
\begin{lemma} \label{lem51} Let $\R$ be a regular sequence of harmonic structures. Then $(\E_\R,\D_\R)$ is a regular Dirichlet form on $L^2(K,\mu_\eta)$.
\end{lemma}
\begin{proof} 
From Theorem~\ref{theoresform} we know that $(\E_\R,\F_\R)$ is a resistance form on $K$. By \cite[Cor. 6.4]{kig12} $(\E_\R,\F_\R)$ is regular and then the statement follows with \cite[Theo. 9.4]{kig12}.
\end{proof}
Introducing Dirichlet boundary conditions we get another Dirichlet form with $\mathcal{D}_\mathcal{R}^0:=\{u|u\in\mathcal{D}_\mathcal{R}, \ u|_{V_0}\equiv 0\}$.
\begin{lemma} \label{lem52}Let $\R$ be a regular sequence of harmonic structures.\\ Then $(\mathcal{E}_\mathcal{R}|_{\mathcal{D}_\mathcal{R}^0\times\mathcal{D}_\mathcal{R}^0}, \mathcal{D}_\mathcal{R}^0)$ is a regular Dirichlet form on $L^2(K,\mu_\eta|_{K\backslash V_0})$.
\end{lemma}
\begin{proof} This follows with Lemma~\ref{lem51} and \cite[Theorem 10.3]{kig12} or \cite[Theorem 4.4.3]{fot}.
\end{proof}
We denote the associated self-adjoint operators with dense domains by $-\Delta_N^{\mu_\eta,\mathcal{R}}$ resp. $-\Delta_D^{\mu_\eta,\mathcal{R}}$. 
\begin{lemma}\label{lem53}$-\Delta_N^{\mu_\eta,\mathcal{R}}$ and $-\Delta_D^{\mu_\eta,\mathcal{R}}$ have discrete non-negative spectrum.
\end{lemma}
\begin{proof}
Since $(K,R_\mathcal{R})$ is compact it follows with \cite[Lemma 9.7]{kig12} that the inclusion map $\iota: \mathcal{D}_\mathcal{R}\hookrightarrow C(K)$ with the norms $\mathcal{E}_\mathcal{R}^{\frac 12}$ resp. $||\cdot ||_{\infty}$ is a compact operator. Since the inclusion map from $C(K)$ to $L^2(K,\mu)$ is continuous the inclusion from $\mathcal{D}_\mathcal{R}$ to $L^2(K,\mu)$ is a compact operator and, therefore, the spectrum of $-\Delta_N^{\mu,\mathcal{R}}$ is discrete and non-negative with \cite[Theo. 5 Chap. 10]{bs87}. Since $\mathcal{D}_\mathcal{R}^0\subset \mathcal{D}_\mathcal{R}$ the same follows for $-\Delta_D^{\mu,\mathcal{R}}$ by \cite[Theo. 4 Chap. 10]{bs87}. 
\end{proof}

\section{Conditions and Hausdorff dimension in resistance metric}\label{chap_cond_haus}
In chapter~\ref{chap_meas_oper} we constructed Dirichlet forms and thus self-adjoint operators on stretched fractals. We needed regular sequences of harmonic structures to do so. Now we want to analyze these operators by calculating some values that give a further description of the underlying fractal. These values are the Hausdorff dimension calculated with respect to the resistance metric and the asymptotic growing of the eigenvalue counting function. But to be able to do this we need to introduce some conditions on the sequences of harmonic structures.
\subsection{Conditions}\label{conditions}
We need the following conditions: We only consider regular sequences of harmonic structures $\R=(\lambda_k,\boldsymbol{\rho}^k)_{k\geq 1}$ such that there exists a $\lambda\in(0,\lambda^\ast]$ with
\setcounter{equation}{0} 
\begin{align}
\sum_{k=1}^\infty| \lambda- \lambda_k|<\infty \label{eqcond}
\end{align}
With the limit comparison test we can easily show that then $\sum_{k=1}^\infty |\ln(\lambda^{-1}\lambda_k)|$ converges and thus
\begin{align*}
\prod_{i=1}^\infty \lambda^{-1}\lambda_i \in (0,\infty)
\end{align*}
This means the sequence $a_m:=\prod_{i=1}^m \lambda^{-1}\lambda_i$ is bounded from above and below:
\begin{align*}
\tilde\kappa_1 \lambda^m \leq \delta_m \leq \tilde\kappa_2 \lambda^m
\end{align*}
For $\delta_m^{(n)}=\lambda_{n+1}\cdots \lambda_{n+m}=\frac{\delta_{n+m}}{\delta_n}$ this means
\begin{align*} 
\frac{\tilde \kappa_1}{\tilde \kappa_2} \lambda^m \leq \delta^{(n)}_m \leq \frac{\tilde \kappa_2}{\tilde \kappa_1} \lambda^m
\end{align*}
Without loss of generality we can assume that $\tilde \kappa_1\leq 1 \leq \tilde \kappa_2$ and thus with $\kappa_1:=\frac{\tilde \kappa_1}{\tilde \kappa_2} $ and $\kappa_2:=\frac{\tilde \kappa_2}{\tilde \kappa_1} $ we get for all $n$ and $m$: 
\begin{align*}
\kappa_1\lambda^m\leq \delta^{(n)}_m\leq \kappa_2\lambda^m
\end{align*}
This means if we have a regular sequence of harmonic structures with (\ref{eqcond}) we have control over the resistances that appear in the rescaling of the quadratic form in Lemma~\ref{lemrescaling}. And we have this control for all sequences $\R^{(n)}$ with the same constants $\kappa_1$ and $\kappa_2$.

\subsection{Hausdorff dimension in resistance metric}
The Hausdorff dimension is a value which describes the size of a set. It strongly depends on the metric that we choose to calculate it. In Proposition~\ref{prop21} we calculated the Hausdorff dimension of stretched fractals with respect to the Euclidean metric. This, however, is not a very meaningful value to describe the analysis of a set. We saw that the resistance forms do not depend on the stretching parameter but only on the topology of $K$. The resistance metric is a better choice to describe the analytic structure of the stretched fractal, so we want to calculate the Hausdorff dimension with respect to this resistance metric.
\begin{align*}
J:=\hspace*{0.4cm}\smashoperator[lr]{\bigcup_{\substack{c\in \C, w\in \A^\ast_0,\\ l\in \{1,\ldots,\rho(c)\}}}}\hspace*{0.3cm} e^w_{c,l}=\bigcup_{n\geq 1}J_n
\end{align*}
Then $K=\Sigma\cup J$ (note: not disjoint). We calculate the dimension of the two parts and due to the stability the Hausdorff dimension of the union will be the bigger of the two.
\begin{lemma} \label{lem61}For any regular sequence of harmonic structures $\R$ we have \begin{align*}\dim_{H,R_\R} J=1\end{align*}
\end{lemma}
\begin{proof}
We show that $\dim_{H,R_\R}(e^w_{c,l})=1$ for all $e^w_{c,l}$. The result follows with $\sigma$-stability.\\[.2cm]
To show this we want to find constants $a,b$ with
\begin{align*}
a|x-y|_e\leq R_\R(x,y)\leq b|x-y|_e
\end{align*}
for all $x,y\in e^w_{c,l}$.\\

(1.) $R_\R(x,y)\leq b|x-y|_e$\\[.1cm]
For this we consider $u\in \F_\R$ with $u(x)=1$ and $u(y)=0$.
\begin{align*}
\Rightarrow \E_\R(u)&\geq \frac 1{\gamma_{|w|}\rho^{|w|+1}_{c,l}} \D_{e^w_{c,l}}(u)\\
&\geq \frac 1{\gamma_{|w|}\rho^{|w|+1}_{c,l}} \frac{\operatorname{diam}(e^w_{c,l},|\cdot|_e)}{|x-y|_e} 
\end{align*}
for all such $u$. This means for the resistance metric
\begin{align*}
R_\R(x,y)\leq \frac{\gamma_{|w|}\rho^{|w|+1}_{c,l}}{\operatorname{diam}(e^w_{c,l},|\cdot|_e)} |x-y|_e
\end{align*}

(2.) $R_\R(x,y)\geq a|x-y|_e$\\[.1cm]
Without loss of generality let $|x-G_w(c)|_e>|y-G_w(c)|_e$. Then define $u$ as follows:
\begin{align*}
u(x)=0 , \ u(y)=1, \ \text{ and linear interpolation between them}
\end{align*}
Also continue $u$ constant $0$ from $x$ to the endpoint which is not $G_w(c)$ and from $y$ to $G_w(c)$ with $1$. \\
Now we want to copy this behavior onto the other lines $e^w_{c,\tilde l}$ with $\tilde l\in\{1,\ldots,\rho(c)\}$ and $\tilde l\neq l$. That means we want that
\begin{align*}
u\circ \xi_{e^w_{c,\tilde l}}(t) = u\circ \xi_{e^w_{c,l}}(t) ,\ \forall t\in[0,1]
\end{align*}
Outside of these edges, we set the function constant $0$. You can see the construction in Figure~\ref{resspider}.
\begin{figure}[H]
\centering
\includegraphics[width=0.5\textwidth]{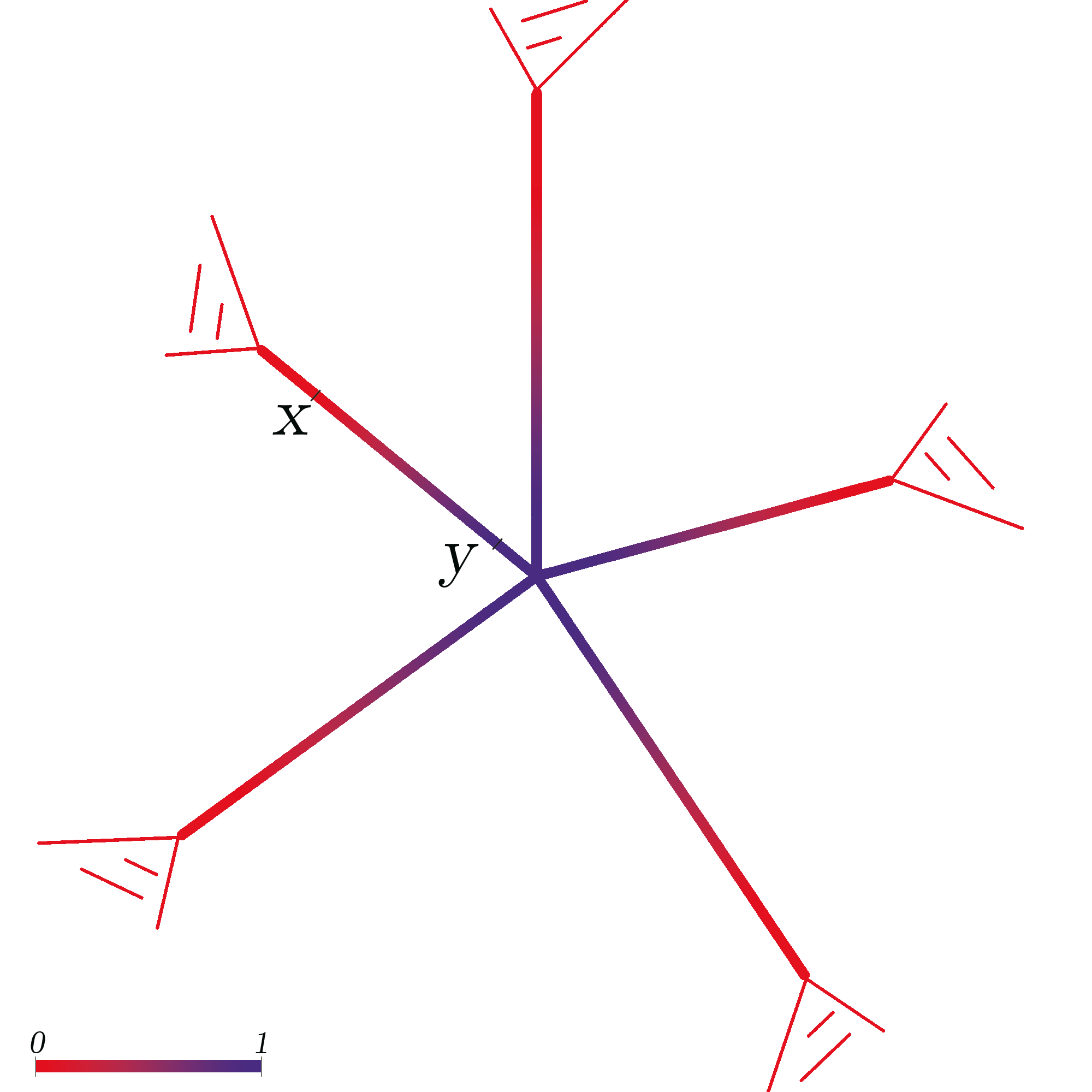}
\caption{Construction of $\protect u$ on connecting lines}
\label{resspider}
\end{figure}
Then $u\in\F_\R$ and we can calculate the energy of $u$.
\begin{align*}
\E_\R(u)=\left(\sum_{\tilde l\in \{1,\ldots,\rho(c)\}} \frac 1{\gamma_{|w|}\rho^{|w|+1}_{c,\tilde l}}\right) \cdot \frac{\operatorname{diam}(e^w_{c,l},|\cdot|_e)}{|x-y|_e}
\end{align*}
Note that the different lines $e^w_{c,\tilde l}$ can have different length (w.r.t. to $|\cdot|_e$), but since we stretched the function in such a way that the proportion of the different parts of $u$ stays the same, the energy is calculated in this way. 
Since $u$ is one of the functions for which the supremum is taken at
\begin{align*}
R_\R(x,y)=\sup\left\{\frac{|u(x)-u(y)|^2}{\E_\R(u)}, \ u\in \F_\R, \ \E_\R(u)>0\right\}
\end{align*}
we get
\begin{align*}
R_\R(x,y)\geq a |x-y|_e
\end{align*}
The constants $a,b$ depend on various things, but they are constant for a fixed $e^w_{c,l}$ and hold for all $x,y\in e^w_{c,l}$.
\end{proof}
Next we want to calculate the Hausdorff dimension of $\Sigma$. This is a self-similar set and $R_\R|_\Sigma$ is a metric on $\Sigma$. We can apply the ideas of \cite{kig95} to calculate this value.
\begin{lemma} \label{lem62}Let $\R$ be a regular sequence of harmonic structures that fulfills the conditions. Then 
\begin{align*}\operatorname{diam}(\Sigma_w,R_\R)\leq c \lambda^n,\ \forall w\in\A^n
\end{align*}
\end{lemma}
\begin{proof}
We know from Corollary~\ref{cor49} that
\begin{align*}
\operatorname{diam}(K_w,R_\R)\leq c\delta_n
\end{align*}
Since $\Sigma_w\subset K_w$ we get
\begin{align*}
\operatorname{diam}(\Sigma_w,R_\R)\leq \operatorname{diam}(K_w,R_\R)\leq c\delta_n\leq c\kappa_1 \lambda^n
\end{align*}
\end{proof}

\begin{lemma} \label{lem63}Let $\R$ be a regular sequence of harmonic structures that fulfills the conditions. Then there is an $M\geq 0$ and $c>0$ such that for all $x\in \Sigma$ we have
\begin{align*}
\#\{w\in \A^n \ | \ R_\R(x,\Sigma_w)\leq c\lambda^n\}\leq M+1, \ \forall n\in \mathbb{N}
\end{align*}
\end{lemma}
\begin{proof}
Since 
\begin{align*}
R_\R(x,y)=\sup\left\{\frac{|u(x)-u(y)|^2}{\E_\R(u)}\ | \ u\in \F_\R, \ \E_\R(u)>0\right\}
\end{align*}
we get for a fixed $u\in \F_\R$ with $u(x)=0$ and $u(y)=1$
\begin{align*}
R_\R(x,y)\geq \frac 1{\E_\R(u)}
\end{align*}
We are looking for a $u$ such that this estimate is good enough. Let $w\in \A^n$, $y\in \Sigma_w$ and $x\in \Sigma\backslash \Sigma_w$. We want to define a function $u_n$ on $V_n$ and then extend it harmonically to a $\tilde u_n\in\F_\R$. Under harmonic extension the energy doesn't change, so we are able to calculate $\E_\R(\tilde u_n)$. Define 
\begin{align*}
u_n:= 1, \ \text{ on } G_w(V_0)
\end{align*}
Now search for all $n$-cells  that are connected to $G_w(V_0)$ over some $c\in C_\ast$. There are at most $M:=\#\C\#V_0$ many of those (see \cite[Lemma 3.3]{kig95}). Set $u_n=1$ on all $c\in C_\ast$ that are connected to $G_w(V_0)$ by some line in $J$ and also $1$ on the endpoints that intersect with the other $n$-cells. Set $u_n=0$ on all remaining points of $V_n$. This procedure is illustrated in this picture for the Stretched Level 3 Sierpinski Gasket.

\begin{figure}[H]
\centering
\includegraphics[width=0.7\textwidth]{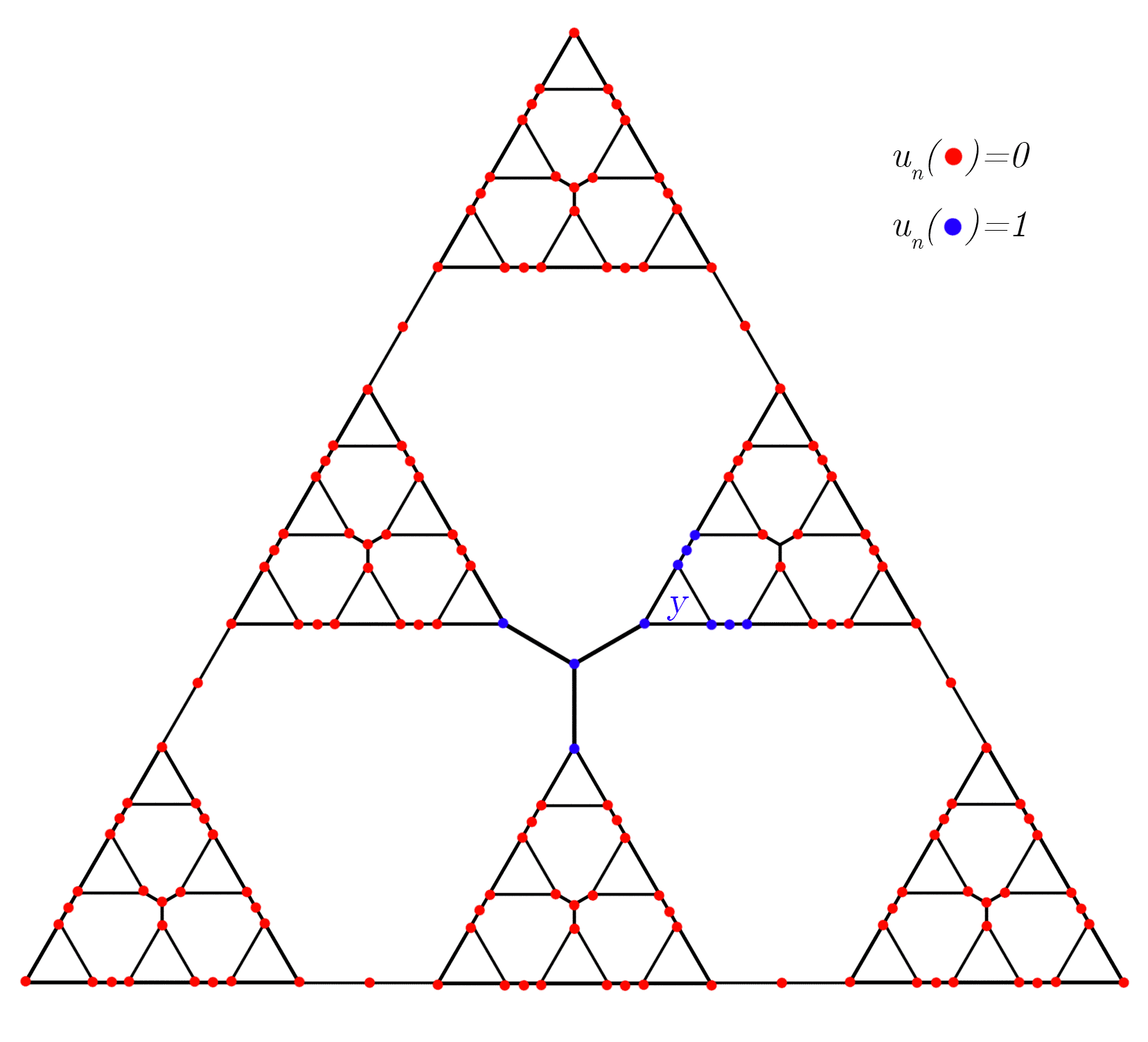}
\caption{Construction of $\protect u_n$}
\end{figure}

Next we extend $u_n$ harmonically to $\tilde u_n\in \F_\R$. We can calculate the energy by
\begin{align*}
\E_\R(\tilde u_n)=\E_{\R,n}(u_n)&\leq M\cdot \#E_0\cdot \frac 1{\delta_n\min_{e\in E_0}r_0(e)}\\
&\leq \frac{M\#E_0}{\kappa_1\min_{e\in E_0}r_0(e)}\cdot \lambda^{-n}
\end{align*}
This leads to
\begin{align*}
R_\R(x,y)\geq \frac{\kappa_1\min_{e\in E_0}r_0(e)}{M\#E_0}\cdot \lambda^n
\end{align*}
This procedure can be done for all $y\in \Sigma_w$ and $x$ such that $\tilde u_n(x)=0$, that means all $x$ that are not in $\Sigma_w$ and all other connected $n$-cells. There are, therefore, at most $M+1$ many $n$-cells (including $\Sigma_w$ itself) for which this construction doesn't work. This gives us the desired result.
\end{proof}
Now we are able to calculate the Hausdorff dimension of $K$.
\begin{theorem} \label{theodim}Let $\R=(\lambda_i,\boldsymbol\rho^i)_{i\geq 1}$ be a regular sequence of harmonic structures that fulfills the conditions, then
\begin{align*}
\dim_{H,R_\R}(K)=\max\left\{1,\frac{\ln N}{-\ln \lambda}\right\}
\end{align*}
\end{theorem}
\begin{proof}
From Lemmata~\ref{lem62} and \ref{lem63} it follows with \cite[Theo. 2.4. or Cor. 1.3]{kig95} that
\begin{align*}
\dim_{H,R_\R}(\Sigma)=\frac{\ln N}{-\ln \lambda}
\end{align*}
With Lemma~\ref{lem61} and $K=\Sigma\cup J$ we get the result.
\end{proof}
\subsection{Examples}
With this result and the harmonic structures that we calculated in chapter~\ref{chapexamples} we are now able to calculate the values of the Hausdorff dimension w.r.t the resistance metric of these stretched fractals for different choices of regular sequences of harmonic structures that fulfill the conditions of chapter~\ref{conditions}. For comparison we also list the values in the self-similar case. In the second column we list all possible values in the stretched case.
\renewcommand{\arraystretch}{2}
 \begin{center}
  \begin{tabular}{|C{3cm}|C{3cm}|C{2.5cm} L{2.1cm}|}
\cline{2-4}
 \multicolumn{1}{c|}{}&  \multicolumn{3}{c|}{$\operatorname{dim}_{H,R_\R}$ } \\
\cline{2-4}
 \multicolumn{1}{c|}{}& self-similar & \multicolumn{2}{c|}{stretched}\\
\hline 
Sierpinski Gasket & $\frac{\ln 3}{-\ln \frac 35}$ &  $\max\{1,\frac{\ln 3}{-\ln \lambda}\}$, &  $ \lambda\in(0,\frac 35]$ \\
\hline 
Level 3 Sierpinski Gasket &$\frac{\ln 6}{-\ln \frac 7{15}}$ & $\max\{1,\frac{\ln 6}{-\ln \lambda}\}$, & $  \lambda\in(0,\frac 7{15}]$  \\
\hline 
Sierpinski Gasket in $\mathbb{R}^d$ & $\frac{\ln (d+1)}{-\ln (\frac {d+1}{d+3})}$ & $\max\{1,\frac{\ln (d+1)}{-\ln \lambda}\}$,& $ \lambda\in(0,\frac {d+1}{d+3}]$ \\
\hline 
Vicsek Set & $\frac{\ln 5}{-\ln \frac 13}$ & $\max\{1,\frac{\ln 5}{-\ln \lambda}\}$,& $  \lambda\in(0,\frac 13]$\\
\hline 
Hata's tree & $\frac{\ln 2}{\ln 2-\ln(\sqrt{5}-1)}$ &$\max\{1,\frac{\ln 2}{-\ln\lambda} \}$, & $\lambda\in (0,\frac{\sqrt{5}-1}2)$ \\
\hline 
 \end{tabular}
 \end{center}
\renewcommand{\arraystretch}{1}

The values of the self-similar case was calculated in general by \cite{kig95}. With the renormalization factors we get the according values. In general the value in the stretched case is less or equal than in the self-similar case. We, however, are able to get the same value in all but one case. For Hata's tree we saw that we can only choose constant sequences of $\lambda_i$. Therefore, they cannot converge to the upper bound and thus we can't reach the same value as in the self-similar case.

\section{Spectral asymptotics}\label{chapter7}
Let $\mu$ be any of the allowed measures $\mu_\eta$ with $\eta\in(0,1]$ and $\R$ a regular sequence of harmonic structures.
Due to Lemma~\ref{lem53} we can write the eigenvalues in non-decreasing order and study the eigenvalue counting function. We denote by $\lambda_k^{N,\mu,\R}$ the $k$-th eigenvalue of $-\Delta_N^{\mu,\R}$ resp. $\lambda_k^{D,\mu,\R}$ for $-\Delta_D^{\mu,\R}$ with $k\geq 1$. Now we can define the eigenvalue counting functions
\begin{align*}
N_N^{\mu,\R}(x):=\# \{k\geq 1 | \lambda_k^{N,\mu,\R}\leq x\}\\
N_D^{\mu,\R}(x):=\# \{k\geq 1 | \lambda_k^{D,\mu,\R}\leq x\}
\end{align*}
Since $\D_\R^0\subset \D_R$ and $\dim \D_\R /\D_\R^0=N$ (since this quotient space consists of the harmonic functions) we get
\begin{align*}
N_D^{\mu,\R}(x)\leq N_N^{\mu,\R}(x)\leq N_D^{\mu,\R}(x)+N, \ \forall x
\end{align*}

We want to study the asymptotic behavior of the eigenvalue counting functions. However, we can only calculate the leading term of the eigenvalue counting functions for regular sequences of harmonic structures that fulfill the conditions of chapter~\ref{chap_cond_haus}. In the following paragraph we will state the results for such sequences.

\subsection{Results}
The next theorem summarizes the results for the leading term for various regular sequences of harmonic structures and measures.
\begin{theorem}\label{theospec} Let $\R$ be a regular sequence of harmonic structures that fulfills the conditions and $\mu=\mu_\eta$ with $\eta\in(0,1]$. Then there exist constants $0<C_1,C_2<\infty$ and $x_0\geq 0$ such that for all $x\geq x_0$:
\begin{align*}
C_1x^{\frac 12 d_S^{\R,\mu}}\leq N_D^{\mu,\R}(x)\leq N_N^{\mu,\R}(x)\leq C_2x^{\frac 12 d_S^{\R,\mu}}
\end{align*}
with 
\begin{align*}
d_S^{\R,\mu}=\begin{cases}
\max\{1,\frac{\ln N^2}{\ln N- \ln \lambda}\},  \ \text{for } \mu=\mu_\eta \text{ with } \eta \in (0,1)\\[.2cm]
\max\{1,\frac{\ln N^2}{- \ln( \beta\lambda)}\},  \ \text{for } \mu=\mu_1=\mu_I, \ \text{if} \ \beta\neq \frac 1{N^2\lambda}
\end{cases}
\end{align*}
\end{theorem}\vspace*{0.3cm}
The leading term is the maximum of the two values. One value corresponds to the fractal part inside the stretched fractal. However, if $\lambda$ gets too small the one-dimensional lines become the dominant part and leading term becomes $1$.\\

The constants $C_1$ and $C_2$ depend on $\R$ and $\mu$. We call the value $d_S^{\R,\mu_{0.5}}(K)=:d_S^{\R}(K)$ the spectral dimension of the stretched fractal $K$. We see that the scaling parameter $\beta$ of the line part of the measure doesn't appear in the leading term if the fractal part of the measure exists. \\

We see that the choice of the regular sequence of harmonic structures as well as the choice of the measure has a big influence on the analysis on $K$. 
\begin{remark}In Theorem~\ref{theodim} we calculated $\dim_{H,R_\R}(K)=\max\{1,\frac{\ln N}{-\ln \lambda}\}$. We see that the following relation holds:
\begin{align*}
d_S^{\R}(K)=\frac{2\dim_{H,R_\R}(K)}{\dim_{H,R_\R}(K)+1}
\end{align*}
This relation was shown to hold for p.c.f. self-similar sets in $\cite{kl93}$ and we just saw that it is also valid for stretched fractals.
\end{remark}
\subsection{Examples}
We want to list the values for the examples for which we calculated the harmonic structures and compare them to the self-similar case. The measure $\tilde\mu_\Sigma$ that we use in the self-similar case is the self-similar measure that assigns each n-cell the same weight.

\renewcommand{\arraystretch}{2.3}
 \begin{center}
  \begin{tabular}{|C{3cm}|C{1.8cm}|C{3cm}| C{3cm}|}
\cline{2-4}
 \multicolumn{1}{c|}{}&  \multicolumn{3}{c|}{$d_S^{\mu,\R}$ } \\
\cline{2-4}
 \multicolumn{1}{c|}{}& self-similar & \multicolumn{2}{c|}{stretched}\\
 \hline 
 Measure & $\tilde \mu_\Sigma$ & $\mu_\eta$,\ $\eta\in(0,1)$ &$\mu_1$\\
 \hline \hline
 & &  $\max\{1,\frac{\ln 9}{\ln 3- \ln \lambda}\}$ &  $\max\{1,\frac{\ln 9}{- \ln \beta \lambda}\}$ \\
\cline{3-4}
Sierpinski Gasket&$\frac{\ln 9}{\ln 5} $  &\multicolumn{2}{c|}{$ \lambda\in(0,\frac 35]$} \\
\cline{3-4}
&&$\beta \in (0,\frac 13)$&$\beta \in (0,\frac 13), \beta\neq \frac 1{9\lambda}$ \\

\hline 
 & & $\max\{1,\frac{2\ln 6}{\ln 6- \ln \lambda}\}$ &  $\max\{1,\frac{2\ln 6}{- \ln \beta \lambda}\}$  \\
\cline{3-4}
\begin{minipage}[c][0.5cm]{\linewidth}\begin{center}  Level 3 Sierpinski Gasket\end{center}\end{minipage} & $\frac{2\ln 6}{\ln 6-\ln\frac 7{15}} $&\multicolumn{2}{c|}{$  \lambda\in(0,\frac 7{15}]$}\\
\cline{3-4}
& &$\beta \in (0,\frac 16)$&$\beta \in (0,\frac 16), \beta\neq \frac 1{6^2\lambda}$ \\
\hline 
 &  & $\max\{1,\frac{2\ln(d+1)}{\ln (d+1)- \ln \lambda}\}$&  $\max\{1,\frac{2\ln(d+1)}{- \ln \beta \lambda}\}$\\
\cline{3-4}
\begin{minipage}[c][0.5cm]{\linewidth}\begin{center}  Sierpinski Gasket in $\mathbb{R}^d$\end{center}\end{minipage}&$\frac{2\ln(d+1)}{\ln(d+3) } $&\multicolumn{2}{c|}{$ \lambda\in(0,\frac {d+1}{d+3}]$}\\
\cline{3-4}
&&$\beta \in (0,\frac 1{d+1})$& \renewcommand{\arraystretch}{1.2} \begin{tabular}{c} $\beta \in (0,\frac 1{d+1})$,\\$ \beta\neq \frac 1{(d+1)^2\lambda}$ \end{tabular}\renewcommand{\arraystretch}{2} \\
\hline 
 &   & $\max\{1,\frac{2\ln 5}{\ln 5- \ln \lambda}\}$& $\max\{1,\frac{2\ln 5}{- \ln \beta \lambda}\}$\\
\cline{3-4}
Vicsek Set&$\frac{2\ln 5}{\ln 15} $&\multicolumn{2}{c|}{ $  \lambda\in(0,\frac 13]$} \\ 
\cline{3-4}
&&$\beta \in (0,\frac 15)$&$\beta \in (0,\frac 15), \beta\neq \frac 1{5^2\lambda}$ \\
 \hline 
&& $\max\{1, \frac{\ln 4}{\ln 2-\ln\lambda} \}$& $ \max\{1,\frac{\ln 4}{-\ln\beta\lambda}  \}$\\
\cline{3-4}
Hata's tree & $\frac{\ln 4}{\ln 4-\ln (\sqrt{5}-1)}$&\multicolumn{2}{c|}{ $  \lambda\in(0,\frac{\sqrt{5}-1}2)$} \\ 
\cline{3-4}
&&$\beta \in (0,\frac 12)$&$\beta \in (0,\frac 12), \beta\neq \frac 1{4\lambda}$ \\
 \hline 
 \end{tabular} 
 \end{center}
\renewcommand{\arraystretch}{1}

The values for the self-similar column come from the result from \cite{kl93} together with the renormalization factors for these examples. As for the Hausdorff dimension the values for $d_S^{\mu,\R}$ are less or equal than the corresponding values in the self-similar case. They can reach the same value in all examples but Hata's tree.

\subsection{Proof of Theorem~\protect\ref{theospec}}
We will carry out the proof for $\mu=\mu_\eta$ with $\eta\in(0,1)$ and in the end show what happens for $\mu=\mu_1$. The main technique for the proof is the Dirichlet-Neumann bracketing as in \cite{kaj10} where it was applied to self-similar sets. We split the proof in the upper and lower estimate.
\subsubsection{Upper estimate}
We obtain the upper estimate by successively adding new Neumann boundary conditions at the points $V_m\backslash V_0$ thus making the domain bigger and, therefore, increasing the eigenvalue counting function. We can introduce the Neumann conditions by defining the following domains:
\begin{align*}
\D_\R^{K_m}&:=\{u|u\in L^2(K_m,\mu|_{K_m}),\ \exists f\in \D_\R : f|_{K_m}=u\}\\
\D_\R^{J_m}&:=\{u|u\in L^2(J_m,\mu|_{J_m}), \ \forall e^w_{c,l}\subset J_m \exists f\in \D_\R : f|_{e^w_{c,l}}=u|_{e^w_{c,l}}\}
\end{align*}
Since the lines $e^w_{c,l}$ in $J_m$ are decoupled by the new Neumann boundary conditions we can see that 
\begin{align*}
\D_\R^{J_m}=\bigoplus_{\substack{c\in \C, l\in\{1,\ldots,\rho(c)\}\\w\in A^k, k<m-1}} H^1(e^w_{c,l})
\end{align*}
We also notice that $\D_\R^{K_m}\perp \D_\R^{J_m}$ and 
\begin{align*}
\D_\R\subset \D_\R^{K_m}\oplus\D_\R^{J_m}
\end{align*}
We can define a new quadratic form $\tilde \E_\R$ on this bigger domain for $f=g+h$ with $g\in \D_\R^{K_m}$ and $h\in \D_\R^{J_m}$.
\begin{align*}
\tilde \E_\R(f):=\E^\Sigma_\R(g)+\sum_{k=m+1}^\infty \frac 1{\gamma_k}\D_{\boldsymbol{\rho}^k,k}(g)+\sum_{k=1}^m\frac 1{\gamma_k}\D_{\boldsymbol{\rho}^k,k}(h)
\end{align*}
and
\begin{align*}
\E_\R^{K_m}(g)&:=\E^\Sigma_\R(g)+\sum_{k=m+1}^\infty \frac 1{\gamma_k}\D_{\boldsymbol{\rho}^k,k}(g)\\
\E_\R^{J_m}(h)&:=\sum_{k=1}^m\frac 1{\gamma_k}\D_{\boldsymbol{\rho}^k,k}(h)
\end{align*}
\begin{lemma} \label{lem72}
 $(\tilde{\mathcal{E}}_\R,\mathcal{D}_\R^{K_m}\oplus\mathcal{D}_\R^{J_m})$, $(\mathcal{E}_\R^{K_m},\mathcal{D}_\R^{K_m})$ and $(\mathcal{E}_\R^{J_m},\mathcal{D}_\R^{J_m})$ are regular Dirichlet forms with discrete non-negative spectrum and $\tilde{\mathcal{E}}_\R=\mathcal{E}_\R^{K_m}\oplus\mathcal{E}_\R^{J_m}$.
\end{lemma}
\begin{proof}
$(\mathcal{E}_\R^{J_m},\mathcal{D}_\R^{J_m})$ is just the sum of scaled Dirichlet energies on one-dimensional edges, hence it is a regular Dirichlet form on $L_2(J_m,\mu|_{J_m})$ with discrete non-negative spectrum. Since $K_m$ is closed $(\mathcal{E}_\R^{K_m},\mathcal{D}_\R^{K_m})$ is a regular resistance form due to \cite[Theo. 8.4]{kig12} and hence a regular Dirichlet form on $L_2(K_m,\mu|_{K_m})$ with \cite[Theo. 9.4]{kig12}. Due to the same Theorem \cite[Theo. 8.4]{kig12} it follows that the associated resistance metric equals the restriction of $R_\R$ to ${K_m}\times{K_m}$. Therefore, since $K_m$ is closed we know that $(K_m,R_\R|_{K_m})$ is compact. The proof for discrete non-negative spectrum works like in the proof of Lemma~\ref{lem53}.
The results for $\tilde{\mathcal{E}}_\R$ follow immediately.
\end{proof}
The eigenvalue counting function has many dependencies. For a Dirichlet form~$\E$ with domain $\D$ in the Hilbert space $L^2(K,\mu)$ we denote the eigenvalue counting function at point $x$ by $N(\E,\D,\mu,x)$. This is the same as the eigenvalue counting function of the self-adjoint operator associated to the Dirichlet form. In our case the measure is always $\mu$ or its restriction to the particular part. We will therefore omit it in the notation.

For the eigenvalue counting functions of the newly introduced Dirichlet forms this means:
\begin{align*}
N_N^{\mu,\R}(x)\leq N(\E_\R^{K_m},\D_\R^{K_m},x)+N(\E_\R^{J_m},\D_\R^{J_m},x), \ \forall x\geq 0 
\end{align*}
The introduction of the new Neumann boundary conditions leads to the decoupling of the $m$-cells and the lines adjoining them. Therefore, the calculations can be done separately. We start with $(\E_\R^{K_m},\D_\R^{K_m})$ which we will call the fractal part.\\

\textbf{U.1: Fractal part $(\E_\R^{K_m},\D_\R^{K_m})$}\\[.1cm]
Define new measures on $K$ as follows for $w\in \A^\ast$.
\begin{align*}
\mu^w:=\mu(K_w)^{-1}\mu\circ G_w
\end{align*}
$\mu^w$ is a measure on the whole $K$ but it only reflects the features of $\mu$ on $K_w$. We notice a few immediate properties.
\begin{align*}
\mu^w(K)=\mu(K_w)^{-1}\mu(K_w)=1, \ \forall w
\end{align*}
as well as
\begin{align*}
\int_K u\circ G_wd\mu^w=\mu(K_w)^{-1}\int_{K_w} ud\mu
\end{align*}

Now for the upper estimate of the fractal part we use the so called \textit{uniform Poincar\'e inequality} (see \cite{kaj10}) for a $C_{PI}\in (0,\infty)$. We define $\R^{(0)}:=\R$, then it holds for all $n\geq 0$ that for all $u\in \D_{\R^{(n)}}$
\begin{align*}
\E_{\R^{(n)}}(u)\geq C_{PI} \int_K |u-\overline u^{\mu^w}|^2d\mu^w
\end{align*}
where $\overline u^\nu=\int_K u d\nu$. The constant $C_{PI}$ is independent of $n$ as well as $w$. This can be seen easily. 
Due to Lemma~\ref{lem48} we know that there is a constant $\mathcal{M}\in(0,\infty)$ only depending on $\lambda^\ast$, $\rho^\ast$ and $r_0$ with
\begin{align*}
R_{\R^{(n)}}(p,q)\leq \mathcal{M}, \ \forall p,q\in K ,  \ \forall n
\end{align*}
\begin{align*}
\mathcal{M}\mathcal{E}_{\R^{(n)}}(u)\geq R_{\R^{(n)}}(p,q)\mathcal{E}_{\R^{(n)}}(u)&\geq |u(p)-u(q)|^2\\[0.2cm]
\Rightarrow \int_K\int_K \mathcal{M}\mathcal{E}_{\R^{(n)}}(u) d\mu^w(q)d\mu^w(p)&\geq \int_K\int_K |u(p)-u(q)|^2d\mu^w(q)d\mu^w(p)\\
&\geq \int_K \left( u(p)-\int_K u(q)d\mu^w(q)\right)^2d\mu^w(p)\\
&=\int_K|u(p)-\overline u^{\mu^w}|^2d\mu^w(p)
\end{align*}
\begin{align*}
\Rightarrow \mathcal{E}_{\R^{(n)}}(u)\geq \frac 1{\mathcal{M}\mu^w(K)^2}\int_K |u-\overline u^{\mu^w}|^2d\mu^w &=\frac 1{\mathcal{M}}\int_K |u-\overline u^{\mu^w}|^2d\mu^w
\end{align*}
That means we have $C_{PI}=\frac 1{\mathcal{M}}$ which holds for all $\R^{(n)}$.\\

We have $N^m$ independent cells in $K_m$ that means the first $N^m$ eigenvalues are all $0$, because the functions that are constant on each $m$-cell are in $\D_\R^{K_m}$. We are interested in the first non-zero eigenvalue which we will call $\lambda^m_{N^m+1}$.

Let $u\in \D_\R^{K_m}$ be the normalized eigenfunction to this eigenvalue $\lambda^m_{N^m+1}$, then $u$ is orthogonal to every $v$ that is constant on the $m$-cells, since this is a linear combination of eigenfunctions to lower eigenvalues. 
\begin{align*}
\lambda_{N^m+1}^m&=\E_\R^{K_m}(u)\\
&\stackrel{(1)}{=}\frac 1{\delta_m}\sum_{w\in \A^m} \E_{\R^{(n)}}(u\circ G_w)\\
&\stackrel{PI}{\geq}\frac 1{\kappa_2 \lambda^m} \sum_{w\in \A^m} C_{PI}\underbrace{\int_K|u\circ G_w-\overline{u\circ G_w}^{\mu^w}|^2d\mu^w}_{=:\star}
\end{align*}
In (1) we used the rescaling of the energy (Lemma~\ref{lemrescaling}).
For $\star$ we have
\begin{align*}
&\int_K (u\circ G_w -\overline{u\circ G_w}^{\mu^w})^2d\mu^w\\
&\hspace*{1.5cm}=\int_K (u\circ G_w)^2d\mu^w - 2\int_K u\circ G_w\cdot \overline{u\circ G_w}^{\mu^w}d\mu^w + \underbrace{\int_K  (\overline{u\circ G_w}^{\mu^w})^2d\mu^w}_{\geq 0}\\
&\hspace*{1.5cm}\geq \frac 1{\mu(K_w)}\int_{K_w} u^2d\mu-2\frac 1{\mu(K_w)}\underbrace{\int_{K_w} u \cdot \overline{u\circ G_w}^{\mu^w}d\mu}_{=0, \text{ since u orth. on const.}}\\
&\hspace*{1.5cm}= \frac 1{\mu(K_w)} \int_{K_w} u^2d\mu 
\end{align*}
Back to $\lambda_{N^m+1}^m$:
\begin{align*}
\Rightarrow \lambda^m_{N^m+1} &\geq \frac 1{\kappa_2 \lambda^m}\sum_{w\in\mathcal{A}^m} C_{PI}\frac 1{\mu(K_w)} \int_{K_w} u^2d\mu \\
&\geq \lambda^{-m} \frac {C_{PI}}{\kappa_2\max \mu(K_w)}\int_K u^2d\mu \\
&\geq \frac{\lambda^{-m}}{N^{-m}} \frac{C_{PI}}{\kappa_2}=C_u\left(\frac N\lambda\right)^m
\end{align*}
We have, $\lambda_{N^m+1}^m\geq C_u (N/\lambda)^m$, that means
 $$x< C_u(N/\lambda)^m \Rightarrow N(\mathcal{E}_\R^{K_m},\mathcal{D}_\R^{K_m},x)\leq N^m$$
For $x\geq C_u$ take $m\in \mathbb{N}$ such that $C_u(N\lambda^{-1})^{m-1}\leq x<  C_u(N\lambda^{-1})^m$. (It's always true that $N\lambda^{-1}>1$.)
\begin{align*}
\Rightarrow N(\mathcal{E}_\R^{K_m},\mathcal{D}_\R^{K_m},x)&\leq N^m \leq N\cdot N^{m-1}= N \left( \left(\frac N\lambda\right)^{\frac{\ln(N)}{\ln(N/\lambda)}}\right)^{m-1}\\
&=N \left(\left( \frac N\lambda\right)^{m-1}\right)^{\frac{\ln(N)}{\ln(N/\lambda)}}\leq N\left( \frac x{C_u} \right)^{\frac{\ln(N)}{\ln(N/\lambda)}}\\
&\leq \underbrace{N C_u^{-{\frac{\ln(N)}{\ln(N/\lambda)}}}}_{C_2^\prime:=} x^{\frac{\ln(N)}{\ln(N/\lambda)}}
\end{align*}

\textbf{U.2 Line part $(\E_\R^{J_m},\D_\R^{J_m})$}\\[.1cm]
Due to the decoupling through the Neumann boundary conditions the domain and form split into
\begin{align*}
\mathcal{E}^{J_m}_\R&=\bigoplus_{\begin{array}{c} c\in \C, l\in\{1,\ldots,\rho(c)\}\\ w\in \A^n , n<m\end{array} }\frac 1{\gamma_{|w|+1}\rho^{|w|+1}_{c,l}} \int_0^1 \left(\frac{d(\cdot\circ \xi_{e^w_{c,l}})}{dx}\right)^2 d\mu\\
\mathcal{D}_\R^{J_m}&=\bigoplus_{\begin{array}{c}c\in \C, l\in\{1,\ldots,\rho(c)\}\\ w\in \A^n , n<m\end{array} }H^1(e^w_{c,l})
\end{align*}
Then it holds for the eigenvalue counting function that
\begin{align*}
N&(\mathcal{E}^{J_m}_\R,\mathcal{D}^{J_m}_\R,x)= \\&\sum_{\begin{array}{c} c\in \C, l\in\{1,\ldots,\rho(c)\}\\ w\in \A^n , n<m\end{array} }N\left(\frac 1{\gamma_{|w|+1}\rho^{|w|+1}_{c,l}} \int_0^1 \left(\frac{d(\cdot\circ \xi_{e^w_{c,l}})}{dx}\right)^2 d\mu, H^1(e^w_{c,l}),x\right)
\end{align*}
The scaling parameter for the measure on the line part scales the integral in the following way:
\begin{align*}
\int_0^1 \left(\frac{d(u\circ \xi_{e^w_{c,l}})}{dx}\right)d\mu =\frac 1{(1-\eta)a_{c,l}\beta^{|w|}}\int_0^1\left(\frac{d(u\circ \xi_{e^w_{c,l}})}{dx}\right)^2dx
\end{align*}
Therefore, there is a one-to-one correspondence of the eigenvalues between the standard Neumann Laplacian on $(0,1)$ and the restriction of the energy to one edge.
\begin{align*}
N\left(\frac 1{\gamma_{|w|+1}\rho^{|w|+1}_{c,l}} \int_0^1 \left(\frac{d(\cdot\circ \xi_{e^w_{c,l}})}{dx}\right)^2 d\mu, H^1(e^w_{c,l}),x\right)&\\=N(-\Delta_N|_{(0,1)},&(1-\eta)a_{c,l}\beta^{|w|}\gamma_{|w|+1}\rho^{|w|+1}_{c,l}x)
\end{align*}
With 
\begin{align*}
N(-\Delta_N|_{(0,1)},x)\leq \frac 1\pi \sqrt{x}+1, \ \forall x\geq 0
\end{align*}
we get
\begin{align*}
N(\E_\R^{J_m},\D_\R^{J_m},x)&=\hspace*{1.3cm}\smashoperator[lr]{\sum_{\begin{array}{c}c\in \C, l\in\{1,\ldots,\rho(c)\}\\ w\in \A^n , n<m\end{array}} }\hspace*{1cm}N(-\Delta_N|_{(0,1)},(1-\eta)a_{c,l}\beta^{|w|}\gamma_{|w|+1}\rho^{|w|+1}_{c,l}x)\\
&\leq \sum_{\begin{array}{c}c\in \C, l\in\{1,\ldots,\rho(c)\}\\ w\in \A^n , n<m\end{array}} \frac 1\pi \sqrt{(1-\eta)a_{c,l}\beta^{|w|}\gamma_{|w|+1}\rho^{|w|+1}_{c,l}x}
\end{align*}
Since $(1-\eta)\leq 1$ as well as $a_{c,l}\leq 1$ and $\rho^{k}_{c,l}\leq \rho^\ast$ we have \setcounter{equation}{1} 
\begin{align}
N(\E_\R^{J_m},\D_\R^{J_m},x)&\leq \sum_{\begin{array}{c}c\in \C, l\in\{1,\ldots,\rho(c)\}\\ w\in \A^n , n<m\end{array}} \frac 1\pi \sqrt{\beta^{|w|}\gamma_{|w|+1}\rho^\ast x}+1\nonumber\\
&=\sum_{\begin{array}{c} w\in\A^n \\ n<m\end{array}}\#E_I^1( \frac{ \rho^\ast}\pi\sqrt{\beta^{|w|}\gamma_{|w|+1} x}+1)\nonumber\\
&=\sum_{k=0}^{m-1} N^k\#E_I^1( \frac{ \rho^\ast}\pi\sqrt{\beta^{k}\gamma_{k+1} x}+1)\nonumber\\
&=\sum_{k=0}^{m-1} \#E_I^1 N^k + \sum_{k=0}^{m-1}\#E_I^1\frac {\rho^\ast}\pi\sqrt{N^{2k}\beta^k\gamma_{k+1} x}\nonumber\\
&\leq\#E_I^1 \frac{N^m-1}{N-1}+\sum_{k=0}^{m-1}\#E_I^1\frac {\rho^\ast\sqrt{\kappa_2}}\pi\sqrt{N^{2k}\beta^k\lambda^k x}\nonumber\\
&\leq \frac {\#E_I^1}{N-1} N^m+ \frac{\#E_I^1\rho^\ast\sqrt{\kappa_2x}}\pi \sum_{k=0}^{m-1} \sqrt{N^2\beta \lambda}^k\label{equpperbound}
\end{align}
From here on we have to distinguish a few cases:
\begin{enumerate}
\item $\lambda>\frac 1N$ and $\frac 1{N^2\lambda}\leq\beta<\frac 1N$
\item ($\lambda>\frac 1N$ and $0<\beta<\frac 1{N^2\lambda}$) or $\lambda \leq \frac 1N$
\end{enumerate}

Let us consider the first case and additionally assume that $\beta\neq \frac 1{N^2\lambda}$. Then $N^2\beta\lambda>1$ and we get from (\ref{equpperbound}):
\begin{align*}
N(\E_\R^{J_m},\D_\R^{J_m},x)&\leq \frac {\#E_I^1}{N-1} N^m+ \frac{\#E_I^1\rho^\ast\sqrt{\kappa_2}}{\pi(\sqrt{N^2\beta\lambda}-1)}\sqrt{N^2\beta \lambda}^m\sqrt x
\end{align*}
For the fractal part we chose $m$ according to $x$ by $C_u(N\lambda^{-1})^{m-1}\leq x<  C_u(N\lambda^{-1})^m$. Therefore, 
\begin{align*}
N(\E_\R^{J_m},\D_\R^{J_m},x)&\leq \frac {\#E_I^1}{N-1} N^m+ \frac{\#E_I^1\rho^\ast\sqrt{\kappa_2}}{\pi(\sqrt{N^2\beta\lambda}-1)}\sqrt{N^2\beta\lambda}^m \sqrt{C_u(N\lambda^{-1})^m}\\
&=\frac {\#E_I^1}{N-1} N^m+ \frac{\#E_I^1\rho^\ast\sqrt{\kappa_2C_u}}{\pi(\sqrt{N^2\beta\lambda}-1)}\sqrt{N^3\beta}^m
\end{align*}
Since $\beta<\frac 1N$ we get
\begin{align*}
N(\E_\R^{J_m},\D_\R^{J_m},x)&\leq \frac {\#E_I^1}{N-1} N^m+ \frac{\#E_I^1\rho^\ast\sqrt{\kappa_2C_u}}{\pi(\sqrt{N^2\beta\lambda}-1)}N^m
\end{align*}
Now if $\beta=\frac 1{N^2\lambda}$ we can change to $\tilde \beta:=\beta+\epsilon$ with $\frac 1{N^2\lambda}< \tilde \beta <\frac 1N$ and still get the result.\\

This means we get a constant $C_2^{\prime\prime}$ such that for $x$ with $C_u(N\lambda^{-1})^{m-1}\leq x<  C_u(N\lambda^{-1})^m$ we have
\begin{align*}
N(\E_\R^{J_m},\D_\R^{J_m},x)\leq C_2^{\prime\prime} N^m
\end{align*}
With the same calculations as for the fractal part we get the same order $\frac{\ln (N)}{\ln (N/\lambda)}$ for the upper bound. That means for $x\geq C_u$ there is a constant $C_2$, such that
\begin{align*}
N_N^{\mu,\R}(x)\leq C_2 x^{\frac{\ln (N)}{\ln (N/\lambda)}}
\end{align*}

We still have to show the second case. Here we always have $N^2\beta\lambda<1$. This means we get from (\ref{equpperbound}):
\begin{align*}
N(\E_\R^{J_m},\D_\R^{J_m},x)&\leq \frac {\#E_I^1}{N-1} N^m+ \frac{\#E_I^1\rho^\ast\sqrt{\kappa_2x}}\pi \sum_{k=0}^{\infty} \sqrt{N^2\beta \lambda}^k\\
&=\leq \frac {\#E_I^1}{N-1} N^m+ \frac{\#E_I^1\rho^\ast\sqrt{\kappa_2}}\pi \frac{1}{1-\sqrt{N^2\beta\lambda}}\cdot x^\frac 12
\end{align*}
For the first term with $N^m$ the calculation from before gives us the upper bound with order $\frac{\ln(N)}{\ln(N/\lambda)}$. Now if $\lambda>\frac 1N$ this is bigger than $\frac 12$ and thus it is the bigger order of asymptotic growing. \\

However, if $\lambda\leq \frac 1N$ we have $\frac{\ln(N)}{\ln(N/\lambda)}\leq \frac 12$ and thus $x^{\frac 12}$ is the leading term. \\

These estimates give us the desired upper bounds.
\subsubsection{Lower estimate}
The idea to get a lower bound is to add new Dirichlet boundary conditions on $V_m$ which makes the domain smaller and thus lowers the eigenvalue counting function.
\begin{align*}
\D_{\R,m}^0&:=\{u|u\in \D_\R^0,\ u|_{V_m}\equiv 0\}\\
\D_{\R,w}^0&:=\{u|u \in \D_{\R,m}^0 ,\ u|_{K_w^c}\equiv 0\}, \ w\in \A^m\\
\D_{\R, e^w_{c,l}}^0&:= \{u|u\in D_{\R,m}^0 ,\ u|_{(e^w_{c,l})^c}\equiv 0\},\ w\in \A^k, k<m 
\end{align*}
With this we have
\begin{lemma}\label{lem73}
$(\mathcal{E}_\R,\mathcal{D}_{\R,m}^0)$, $(\mathcal{E}_\R,\mathcal{D}_{\R,w}^0)$ and $(\mathcal{E}_\R,\mathcal{D}^0_{\R,e_{c,l}^w})$ are regular Dirichlet forms with discrete non-negative spectrum.
\end{lemma}
\begin{proof}
Since $K\backslash V_m$ is open,  $(\E_\R,\D_{\R,m}^0)$ is a regular Dirichlet form with \cite[Theorem 10.3]{kig12} or \cite[Theorem 4.4.3]{fot}. Since $\D_{R,m}^0\subset \D_\R^0$ the spectrum is discrete and non-negative with \cite[Theo. 4 Chap. 10]{bs87}. Since $K_w\backslash V_m$ for $w\in A^m$ and $e^w_{c,l}\backslash V_m$ for $w\in A^{k}$ with $k<m$ are also open the rest of the statement follows analogously.

\end{proof}
Again we get an estimate on the eigenvalue counting function:
\begin{align*}
N(\E_\R|_{\D_{\R,m}^0\times \D_{\R,m}^0}, \D_{\R,m}^0,x)\leq N_D^{\mu,\R}(x)
\end{align*}
Due to the finite ramification and the fact that functions in $\D_{\R,m}^0$ have to be zero in $V_m$, this domain splits into the domain restricted to the different parts.
\begin{align*}
\D_{\R,m}^0=\left(\bigoplus_{w\in\A^m}\D_{\R,w}^0\right) \bigoplus \left(\bigoplus_{\substack{c\in \C, l\in\{1,\ldots,\rho(c)\}\\ w\in \A^n, n<m}}\D_{\R, e^w_{c,l}}^0\right)
\end{align*}
That means for the eigenvalue counting function, $\forall x\geq 0$
\begin{align*}
\sum_{w\in\A^m} N(\E_{\R},\D_{\R,w}^0,x)+\sum_{\substack{c\in \C, l\in\{1,\ldots,\rho(c)\}\\ w\in \A^n, n<m}}N(\E_\R,D_{\R,e^w_{c,l}}^0,x)\leq N_D^{\mu,\R}(x)
\end{align*}
Again due to the decoupling, the individual eigenvalue counting functions can be calculated separately.\\

\textbf{L.1 Fractal part $(\E_\R,\D_{\R,w}^0)$}\\[.1cm]
This time we want an upper estimate on the first eigenvalue of $(\E_\R, \D_{\R,w}^0)$ which is positive due to the Dirichlet boundary conditions. This gives us a lower estimate for $N(\E_\R,\D_{\R,w}^0,x)$. The first eigenvalue can be calculated via the following fact
\begin{align*}
\lambda_1^w=\inf_{u\in \D_{\R,w}^0} \frac{\E_\R(u)}{||u||_\mu^2}
\end{align*}
where $||u||_\mu$ denotes the $L^2$ norm with respect to $\mu$. This leads to 
\begin{align*}
\lambda^w_1\leq \frac{\E_\R(u)}{||u||_\mu^2}, \ \text{ for each } u\in \D_{\R,w}^0
\end{align*}
The idea is to find a $u\in \D_{\R,w}^0$ which is ``good enough''.\\

Let us consider the fixed $m$-cell $K_w$. We have Dirichlet boundary conditions on $V_m$. There are $\#V_0$ many points of $V_m$ in $K_w$. Take the smallest $j\in\mathbb{N}$ such that $N^j>\#V_0$. There are $N^j$ many $m+j$-cells inside $K_w$ which means that there is at least one that doesn't include any points of $V_m$. Therefore, there are no Dirichlet boundary conditions anywhere in this $m+j$-cell $K_{\hat w}$ with $|\hat w|=m+j$.

We, however, have to look for an even smaller cell. We want to do the same procedure again and look for a cell that has no common points of $V_{m+j}$ with $K_{\hat w}$. With the same arguments there is an $m+2j$-cell $K_{\tilde w}$ with $|\tilde w|=m+2j$ inside $K_{\hat w}$ that fulfills this requirement.\\

We now want to construct a function on $K_w$ that is in $\D_{\R,w}^0$ with the help of $K_{\tilde w}$. The construction is very similar to the one in the proof of Lemma~\ref{lem63} where we calculated the Hausdorff dimension of $K$ with respect to the resistance metric. Define $u_m$ on $G_{\tilde w}(V_0)$ to be constant $1$. Now search for all $m+2j$-cells that are connected to $K_{\tilde w}$ over some $c\in \C_\ast$. There are at most $M=\#\C\#V_0$ many of those. Set $u_m=1$ on all $c\in C_\ast$ that are connected to $G_{\tilde w}(V_0)$ in $E_{m+2j}$ and also $1$ on all other points that are connected to these $c$. By the way we chose $K_{\tilde w}$ and $K_{\hat w}$ we made sure that all points where we set $u_m$ to be $1$ are not in $V_m$. On all other points of $V_{m+2j}$ we choose $u_m$ to be $0$. Then extend $u_m$ harmonically to be a function in $\D_{\R,w}^0$. Note that the Dirichlet conditions on $V_m$ are fulfilled. Again similar to Lemma~\ref{lem63} we can calculate the energy of $u_m$. 
\begin{align*}
\E_\R(u_n)=\E_{\R,m+2}(u_n)&\leq M\cdot \#E_0\cdot \frac 1{\delta_{m+2j}\min_{e\in E_0}r_0(e)}\\
&\leq \frac{M\#E_0}{\kappa_1 \min_{e\in E_0}r_0(e)}\cdot \lambda^{-(m+2j)}
\end{align*}
\newpage
We also need a lower estimate for the $L^2$-norm of $u_m$ to get an upper estimate on $\lambda_1^w$. But we know that $u_m$ is constant $1$ on $K_{\tilde w}$. Therefore
\begin{align*}
||u_m||_\mu^2&=\int_{K_w}|u_m|^2d\mu\\
&\geq \int_{K_{\tilde w}} \underbrace{|u_m|^2}_{=1}d\mu\\
&=\mu(K_{\tilde w})\\[.2cm]
\Rightarrow \lambda_1^w&\leq \frac {M\#E_0}{\kappa_1\min_{e\in E_0}r_0(e)} \frac{\lambda^{-(m+2j)}}{\mu(K_{\tilde w})}
\end{align*}
We recall that our measures $\mu=\mu_\eta=\eta\mu_I+(1-\eta)\mu_\Sigma$. That means we have 
\begin{align*}
\mu(K_w)=\mu_\eta(K_w)&\geq (1-\eta)\mu_\Sigma(K_w)\\
&=(1-\eta)\left(\frac 1N\right)^{|w|}
\end{align*}
This leads to 
\begin{align*}
\lambda_1^w&\leq \frac {M\#E_0}{\kappa_1\min_{e\in E_0}r_0(e)(1-\eta)}(N\lambda^{-1})^{m+2j}\\
&=\underbrace{\frac {M\#E_0(N\lambda^{-1})^{2j}}{\kappa_1\min_{e\in E_0}r_0(e)(1-\eta)}}_{C_l:=} \cdot \left(\frac N\lambda\right)^m
\end{align*}
Note that $j$ is independent of $m$. For $x\geq C_l(N\lambda^{-1})$ choose $m\in \mathbb{N}$ such that
\begin{align*}
C_l(N\lambda^{-1})^m\leq x<C_l(N\lambda^{-1})^{m+1}
\end{align*}
For these $x$ it holds that at least one eigenvalue of $(\E_\R,\D_{\R,w}^0)$ is smaller than $x$.
\begin{align*}
N(\E_\R,\D_{\R,w}^0,x)&\geq 1\\
\Rightarrow \sum_{w\in \A^m} N(\E_\R,\D_{\R,w}^0,x)&\geq N^m = \frac 1N ((N\lambda^{-1})^{m+1})^{\frac{\ln N}{\ln (N\lambda^{-1})}}\\
&\geq \underbrace{\frac 1N C_l^{\frac{\ln N}{\ln (N\lambda^{-1})}}}_{C_1:=}x^{\frac{\ln N}{\ln (N\lambda^{-1})}}
\end{align*}

\textbf{L.2 Line part}\\[.1cm]
In the previous calculations we saw that the fractal part already gives a lower bound with the same order as the upper bound for $\lambda>\frac 1N$. Therefore, the influence of the line part cannot be bigger than the fractal part. We can use the trivial estimate
\begin{align*}
\sum_{\begin{array}{c}c\in \C, l\in\{1,\ldots,\rho(c)\}\\ w\in \A^n , n<m\end{array}}N(\E_\R|_{\D_{\R,e^w_{c,l}}^0\times \D_{\R,e^w_{c,l}}^0}, \D_{\R,e^w_{c,l}}^0, x)\geq 0
\end{align*}
If, however, $\lambda\leq \frac 1N$ this order of $\frac{\ln(N)}{\ln(N/\lambda)}$ is at most $\frac 12$, so it is not the one we want. To achieve the right one, we can use just one of the one-dimensional lines, say $e_{c,l}$:
\begin{align*}
\sum_{\begin{array}{c}c\in \C, l\in\{1,\ldots,\rho(c)\}\\ w\in \A^n , n<m\end{array}}&N(\E_\R|_{\D_{\R,e^w_{c,l}}^0\times \D_{\R,e^w_{c,l}}^0}, \D_{\R,e^w_{c,l}}^0, x)\\
&\geq N(\E_\R|_{\D_{\R,e_{c,l}}^0\times \D_{\R,e_{c,l}}^0}, \D_{\R,e_{c,l}}^0, x)\\
&\geq N(-\Delta_D|_{(0,1)},a_{c,l}\rho^1_{c,l}x)\\
&\geq \frac 1\pi \sqrt{a_{c,l}\rho^1_{c,l}} \cdot x^{\frac 12}
\end{align*}

This suffices to show the desired result if our measure includes the fractal part.\\

\textbf{Remaining: }$\boldsymbol{\mu=\mu_1=\mu_I}$\\[.1cm]
We still need to show the case if $\mu=\mu_1=\mu_I$. Then we know that
\begin{align*}
\mu_I(K_w)=\beta^{|w|}
\end{align*}
Whenever we used $(1-\eta)\left(\frac 1N\right)^{|w|}\leq \mu_\eta(K_w)\leq \left(\frac 1N\right)^{|w|}$ in the proof for $\mu_\eta$ with $\eta\in (0,1)$ we can exchange this estimate with 
\begin{align*}
\mu(K_w)=\beta^{|w|}
\end{align*}
For $\beta\neq\frac 1{N^2\lambda}$ the rest of the proof works exactly the same as in the case $\eta\in(0,1)$ and this leads to the asymptotic growing
\begin{align*}
\max\left\{\frac{\ln N}{-\ln (\beta \lambda)},1\right\}
\end{align*}
However, if $\beta=\frac 1{N^2\lambda}$, i.e. $N^2\beta\lambda=1$ we can't change $\beta$ to $\tilde \beta=\beta+\epsilon$ as in the case $\eta\in(0,1)$ since we need the exact value $\beta$ for the future calculation. This leads to an additional $log(x)$ term in the upper bound. We will not include this result in the theorem since it doesn't fit to the other cases.
\hfill $\square$

\section{Outlook and further research}\label{chapter8}
\subsection*{Existence of regular harmonic structures}
The idea of this work was very similar to \cite{kl93}. Namely, if we have a regular harmonic structure we can choose a sequence and thus get resistance forms. After choosing a measure we get Dirichlet forms and thus operators. We showed the existence of regular harmonic structures for a few examples by explicitly calculating them. The question remains, in which cases such a regular harmonic structure exists. One possible approach is to show that if we have a harmonic structure in the self-similar case, this also induces one in the stretched case. In all our examples this was the case, since we always used the same resistances on $(V_0,E_0)$ as in the self-similar case. This means, the choice of $r_0$ was influenced by the existence of a regular harmonic structure on the self-similar set.

If we have no way to compare it to the self-similar case we would still like to prove existence for as many sets as possible. The first set of fractals for which we would like to try this would be stretched nested fractals. As in \cite{lin90} this could mean getting the existence without knowing the value of $\lambda$.

\subsection*{Comparison of $\protect d_S$ in the self-similar and the stretched case}
We saw in the examples that the values for Hausdorff dimension and leading term for the asymptotics in the stretched case are less or equal than in the self-similar case. We believe this is always true. 

We can give heuristic arguments for this conjecture. If we set the resistances $\rho=0$ on all connecting edges in $E^I_1$ this would mean that points that were connected by this edge get identified with each other. This gives us back the first graph approximation in the self-similar case. $(V_0,E_0)$ is the same for either self-similar and stretched case. By increasing $\rho>0$ on the edges in $E^I_1$ we still want to have an equivalent network for $(V_0,E_0)$. This means that the effective resistance between those points in $V_0$ has to stay the same. However, we know from general electrical theory that if we increase the resistances on the connecting edges, the resistances on the fractal edges in $E_1^\Sigma$ have to decrease in order to keep the effective resistances at the same level.

For Hata's tree we saw that the same values as for the self-similar case was not possible. This gives rise to the question in which cases this is possible and to find criteria to characterize stretched fractals.

\subsection*{More general harmonic structures and measures}
The harmonic structures that we used are very symmetric. We have the same renormalization in each cell. This is a big restriction and there will likely be stretched fractals for which there is no regular harmonic structure that fulfills this symmetry. This means we need to generalize our notion of harmonic structure to allow different scaling in different cells. 

We, however, believe that there isn't any new difficulty in obtaining Hausdorff dimension and the leading term in the asymptotics. We need to introduce a few more indices and to make sure the scalings for different cells converge on their own to a limit. With such conditions we should be able to prove the results with a combination of the proof in this work and the ideas from \cite{kl93} or \cite{kaj10} concerning partitions of the word space.

The same holds for the measures that we used. These were very symmetric and we should substitute them for more general ones. We want to allow different scaling in different cells for both fractal- and line-part of the measure. But again, there should be no new difficulties in obtaining Hausdorff dimension and leading term of the spectral asymptotics by connecting the ideas of this work and \cite{kl93,kaj10}.

\subsection*{Does the fractal part of the resistance form really exist}
Besides the construction of resistance forms on the Stretched Sierpinski Gasket, the main result of \cite{afk17} examined the resistance forms $\E_\R$. In particular the authors studied the fractal part $\E_\R^\Sigma$ and showed that it only survives in a special case. In our notion this is the case $\sum_{i\geq 1} |\lambda_i-\frac 35|<\infty$. In all other cases we have $f\in \F_\R \Rightarrow \E_\R^\Sigma(f)=0$. 

This question can be generalized to stretched fractals. When does the fractal part $\E_\R^\Sigma$ in the resistance form of a stretched fractal survive. The immediate conjecture is: If $\lambda_{ss}$ is the renormalization in the self-similar case we conjecture that $\E_\R^\Sigma$ survives if and only if we have $\sum_{i\geq 1} |\lambda_i-\lambda_{ss}|<\infty$ for the sequence of regular harmonic structures.

However, it is not possible to apply the same proof as in \cite{afk17} since it strongly depends on the value of $\lambda_{ss}=\frac 35$. These values are not known in general.

\subsection*{More stretching}
We were able to stretch p.c.f. self-similar fractals that fulfilled a certain connectivity condition (\ref{pcfcond},\ref{pcfcond2}). We did this by introducing one-dimensional lines. We could fill the holes with other objects than just lines. For example for each $c\in \C$ we could fill the hole with a fractal that has $\rho(c)$ many boundary points. We saw that the one-dimensionality of the lines influenced the dimension as well as the leading term of the spectral asymptotics. It would be interesting to see how other objects would influence these values. 
\begin{figure}[H]
\centering
\includegraphics[width=0.6\textwidth]{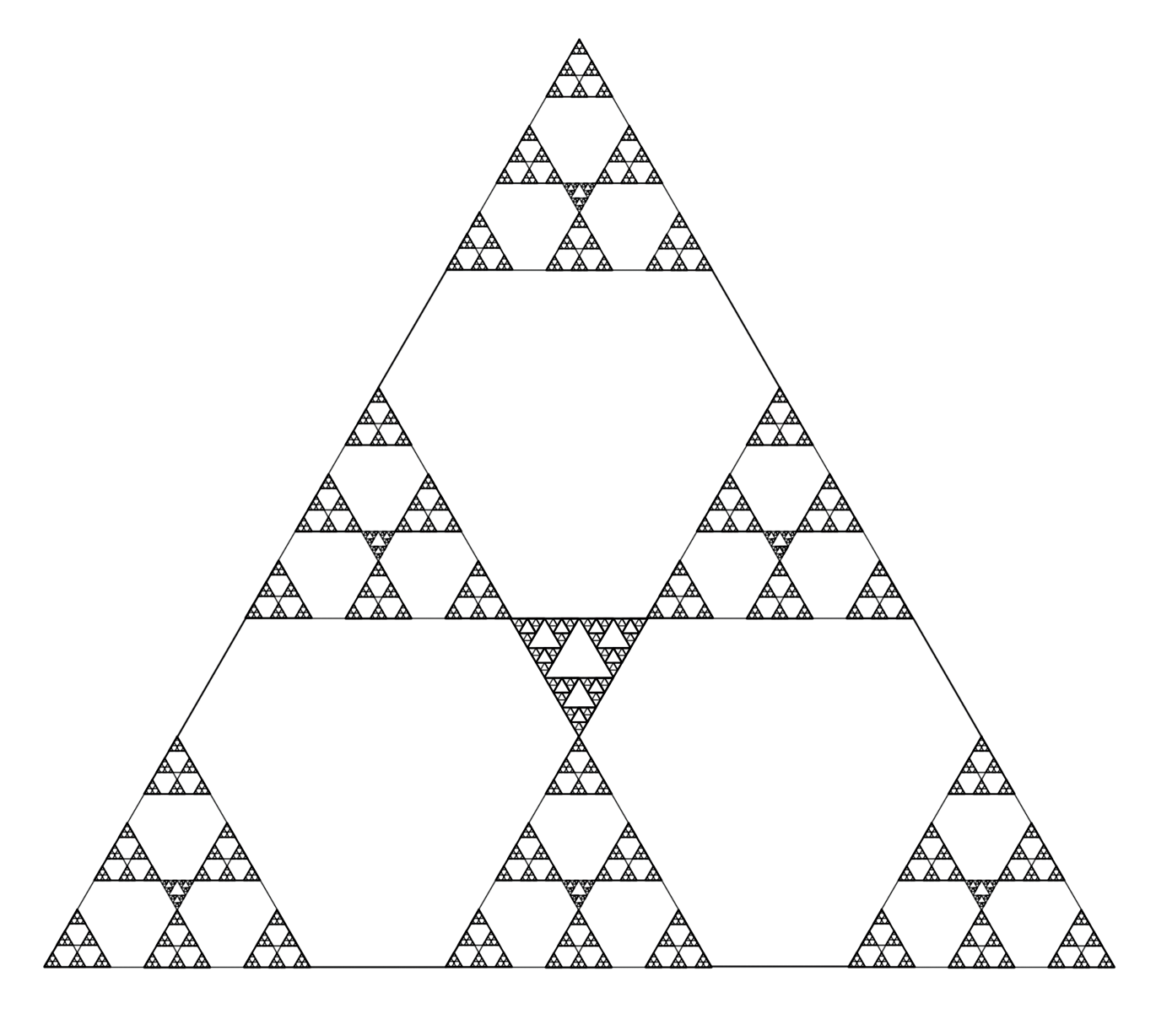}
\caption{Filling the hole with self-similar Sierpinski Gasket}
\end{figure}
We can also stretch sets that are not p.c.f., for example, the unit square $[0,1]^2$. It is the attractor of four similitudes
\begin{align*}
F_1\begin{pmatrix}x_1\\x_2\end{pmatrix}&=\begin{pmatrix}0.5&0\\0&0.5  \end{pmatrix}\begin{pmatrix}x_1\\x_2\end{pmatrix}\\
F_2\begin{pmatrix}x_1\\x_2\end{pmatrix}&=\begin{pmatrix}0.5&0\\0&0.5 \end{pmatrix}\begin{pmatrix}x_1\\x_2\end{pmatrix}+\begin{pmatrix}0.5\\0\end{pmatrix}\\
F_3\begin{pmatrix}x_1\\x_2\end{pmatrix}&=\begin{pmatrix}0.5&0\\0&0.5  \end{pmatrix}\begin{pmatrix}x_1\\x_2\end{pmatrix}+\begin{pmatrix}0\\0.5\end{pmatrix}\\
F_4\begin{pmatrix}x_1\\x_2\end{pmatrix}&=\begin{pmatrix}0.5&0\\0&0.5  \end{pmatrix}\begin{pmatrix}x_1\\x_2\end{pmatrix}+\begin{pmatrix}0.5\\0.5\end{pmatrix}
\end{align*}
However, there is not an obvious way to connect the copies if we lower the contraction ratios. We could still be using one-dimensional lines. It is also not obvious how we have to place these lines. This procedure, however, changes the connectedness of the fractal and gives us a completely new fractal which has to be analyzed geometrically and analytically. We have to place the lines in such a way that it connects the \mbox{1-cells}~$\Sigma_i$ to ensure connectedness.
\begin{figure}[H]
\centering
\includegraphics[width=0.4\textwidth]{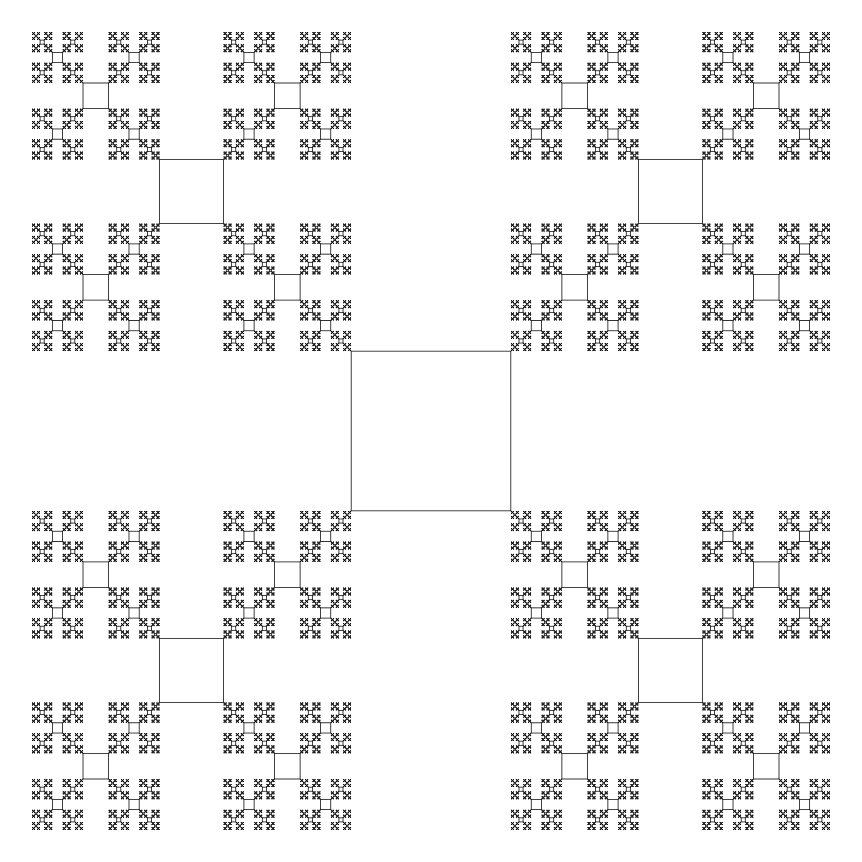}
\caption{Stretched unit square - version 1}
\end{figure}
We can also place the lines somewhere else.
\begin{figure}[H]
\centering
\includegraphics[width=0.4\textwidth]{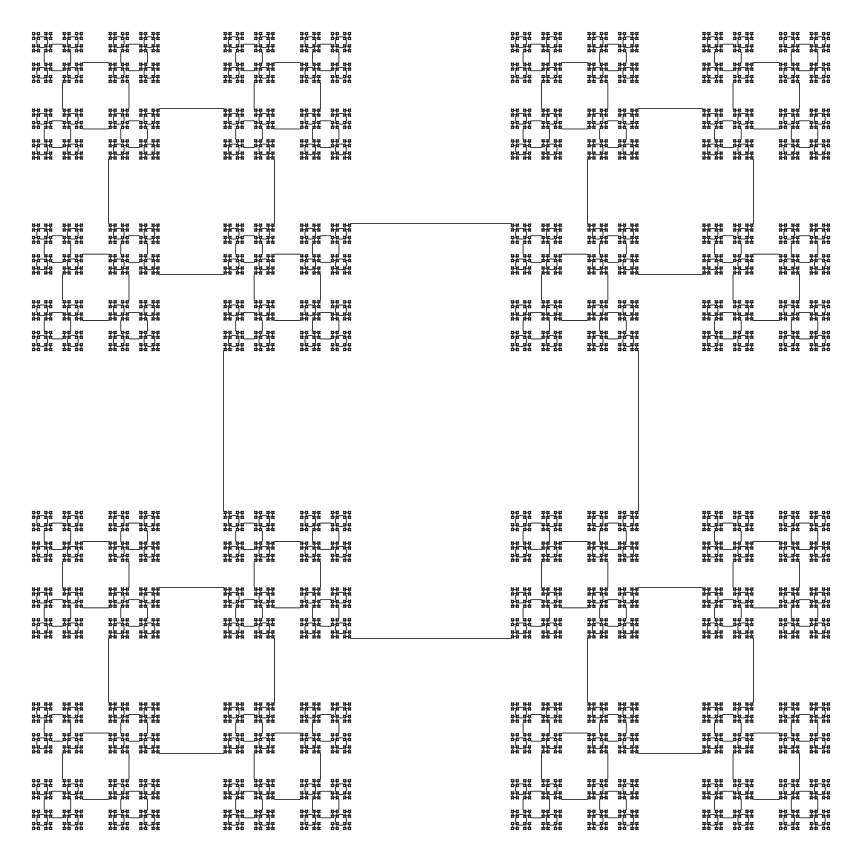}
\caption{Stretched unit square - version 2}
\end{figure}
Another way to connect the copies is to use two-dimensional areas.
\begin{figure}[H]
\centering
\includegraphics[width=0.4\textwidth]{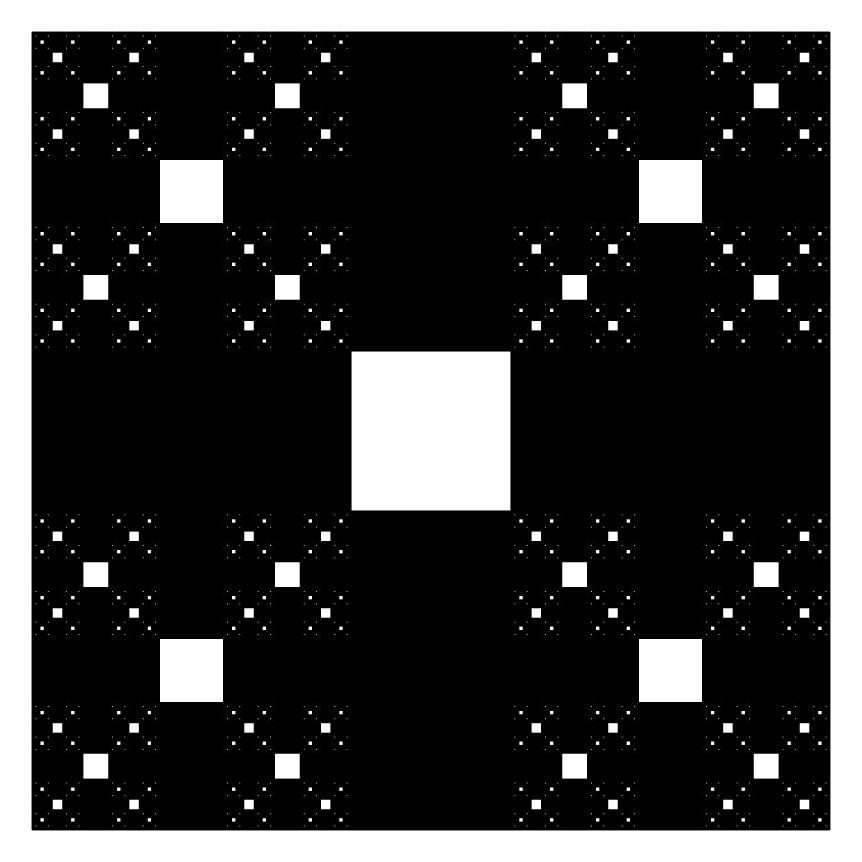}
\caption{Stretched unit square - version 3}
\end{figure}
This gives us a completely different fractal. The two-dimensional part will dominate the geometric and analytical appearance. 

There are many ways to connect the copies between one- and two-dimensional objects. This gives rise to many new and interesting fractals. 



\newpage
\begin{thebibliography}{99}

\bibitem{af12}
P. Alonso Ruiz and U.R. Freiberg, Hanoi attractors and the Sierpi\'nski gasket. \textit{Int. J. Math. Model. Numer. Optim.} \textbf{3} (2012), no. 4, 251--265. 

\bibitem{af13}
P. Alonso Ruiz and U.R. Freiberg, Weyl asymptotics for Hanoi attractors. \textit{Forum Math.} \textbf{29} (2017), no. 5, 1003--1021.

\bibitem{akt16}
P. Alonso Ruiz, D.J. Kelleher and A. Teplyaev, Energy and Laplacian on Hanoi-type fractal quantum graphs. \textit{J. Phy. A} \textbf{49} (2016), no. 16, 165206, 36pp.

\bibitem{afk17}
P. Alonso Ruiz, U. Freiberg and J. Kigami, Completely symmetric resistance forms on the stretched Sierpi\'nski gasket. \textit{J. of Fractal Geometry} \textbf{5} (2018), 227--277. 

\bibitem{barl98}
M.T.Barlow, Diffusions on fractals. \textit{Lectures on probability theory and statistics (Saint-Flour, 1995)}, 1--121, \textit{Lecture Notes in Math.} \textbf{1690}, Springer, Berlin, (1998) 


\bibitem{bar06}
M.F. Barnsley, \textit{Superfractals}, Cambridge University Press, Cambridge, 2006 

\bibitem{ber1}
M.V. Berry, \textit{Distribution of modes in fractal resonators}. In: G\"uttinger, W., Eikemeier, H. (eds), \textit{Structural stability in physics}, Springer Ser. Synergetics, 4, Springer, Berlin, 1979 pp. 51-53

\bibitem{ber2}
M.V. Berry, \textit{Some geometric aspects of wave motion: wavefront dislocations, diffraction catastrophes, diffractals}, In: \textit{Geometry of the Laplace operator}, Proc. Symp. Pure Math., vol. 36, Providence, R.I.: Am. Math. Soc. 1980, pp. 13-38

\bibitem{bs87}
M.S. Birman, Solomjak M.Z., \textit{Spectral theory of self-adjoint operators in hilbert space}, D. Reidel Publishing Company, Dordrecht, Holland, 1987 

\bibitem{ca11}
G.A. Campbell, Cisoidal oscillations. \textit{Trans. Amer. Inst. Electrical Engineers}, \textbf{30} (1911), no. 4 , 789--824

\bibitem{dh77}
D. Dhar, Lattices of effectively nonintegral dimensionality, \textit{Journal of Mathematical Physics} \textbf{18} (1977), no. 4, 577--585



\bibitem{fushim}
M. Fukushima and T. Shima, On a spectral analysis for the Sierpinski gasket \textit{Potential Analysis} \textbf{1} (1992), no. 1, 1-35

\bibitem{fot}
M. Fukushima, Y. Oshima, and M. Takeda, \textit{Dirichlet forms and symmetric Markov processes}, extended ed., de Gruyter Studies in Mathematics, vol. 19, Walter de Gruyter \& Co., Berlin, 2011 

\bibitem{fra18}
J. Fraser, Inhomogeneous self-similar sets and box dimensions. \textit{Studia Math.} \textbf{213} (2012), no. 2, 133--156 

\bibitem{hata85}
Hata, M.: On the structure of self-similar sets.  \textit{Japan J. Appl. Math.} \textbf{2} (1985), no. 2, 381--414

\bibitem{haus17}
E. Hauser, Spectral asymptotics on the Hanoi attractor. Preprint, 2017. arXiv:1710.06204 [math.SP]

\bibitem{haus18}
E. Hauser, Oscillations on the Stretched Sierpinski Gasket. Preprint, 2018. arXiv:1807.08510 [math.SP]

\bibitem{kaj10}
N. Kajino, Spectral asymptotics for Laplacians on self-similar sets. \textit{J. Funct. Anal.} \textbf{258} (2010), no. 4, 1310--1360 

\bibitem{kig89}
J. Kigami, A harmonic calculus on the Sierpinski spaces, \textit{Jpn. J. Appl. Math.} \textbf{6} (1989),no. 2, 259--290 

\bibitem{kig93}
J. Kigami, Harmonic Calculus on P.C.F. Self-Similar Sets. \textit{Trans. Amer. Math. Soc.} \textbf{335} (1993), no. 2, 721--755 

\bibitem{kl93}
J. Kigami and M.L. Lapidus, Weyl's problem for the spectral distribution of Laplacians on p.c.f. self-similar fractals. \textit{Comm. Math. Phys.} \textbf{158} (1993), no. 1, 93--123 

\bibitem{kig95}
J. Kigami, Hausdorff dimensions of self-similar sets and shortest path metrics, \textit{J. Math. Soc. Japan} \textbf{47} (1995), no. 3, 381--404 


\bibitem{kig03}
J. Kigami, Harmonic analysis for resistance forms. \textit{J. Funct. Anal.} \textbf{204} (2003), no. 2, 399--444 

\bibitem{kig12}
J. Kigami, Resistance forms, quasisymmetric maps and heat kernel estimates. \textit{Mem. Amer. Math. Soc.} \textbf{216} (2012), no. 1015, vi+132 pp. 

\bibitem{lin90}
T. Lindstr\o m, Brownian Motion on Nested Fractals. \textit{Mem. Amer. Math. Soc.} \textbf{83} (1990), no. 420, iv+128pp 

\bibitem{sn08}
N. Snigireva, \textit{Inhomogeneous self-similar sets and measures}. Ph.D thesis. University of St Andrews, St. Andrews, (2008)

\bibitem{shim}
T. Shima, On eigenvalue problems for the random walks on the Sierpinski pre-gaskets. \textit{Jpn. J. Indust. Appl. Math.} \textbf{8} (1991), no. 1, 127--141 

\bibitem{str00}
R.S. Strichartz, Taylor Approximations on Sierpinski Gasket Type Fractals. \textit{J. Funct. Anal.} \textbf{174} (2000), no. 1, 76--127 

\bibitem{we11}
H. Weyl, \"Uber die asymptotische Verteilung der Eigenwerte. \textit{G\"ott. Nach.}, 110-117 (1911)

\end{thebibliography}
\end{document}